\documentclass[11pt]{article}

\usepackage{float}
\usepackage{tikz}
\usepackage[all]{xy}
\usepackage{wrapfig}
\usepackage{enumitem}
\usepackage{multirow, array} 
\usepackage{setspace}  

\usepackage{latexsym}
\usepackage{amssymb,amsbsy,amsmath,amsfonts,amssymb,amscd}
\usepackage{epsfig, graphicx}
\usepackage{hyperref}
\newcommand{\commentout}[1]{}
\newcommand{\R}{\mathbb{R}}

\newcommand {\Chi} {{\bf \raise 2pt \hbox{$\chi$}} }

\newcommand {\proof} {\noindent {\bf Proof}. }
\newcommand{\beq}{\begin{equation}}
\newcommand{\eeq}{\end{equation}}
\newcommand{\bea} {\begin{array}{rl}}
\newcommand{\eea} {\end{array}}
\newcommand{\bepa}{\left\{ \begin{array}{l}}
\newcommand{\eepa} {\end{array}\right.}
\newtheorem{theorem}{Theorem}[section]

\newtheorem{definition}[theorem]{Definition}
\newtheorem{remark}[theorem]{Remark}

\numberwithin{equation}{section}

\newcommand{\qed}{{ \hfill
                      {\unskip\kern 6pt\penalty 500 \raise -2pt\hbox{\vrule\vbox to 6pt{\hrule width 6pt
                      \vfill\hrule}\vrule} \par}   }}

\title{\Large \bf 
Towards a realistic NNLIF model:
Analysis and numerical solver  for excitatory-inhibitory networks
 with delay and refractory periods}

\author{Mar\' {\i}a J. C\'aceres
\and Ricarda Schneider
}

\date{}

\begin{document}
\maketitle
\pagestyle{plain}
\pagenumbering{arabic}


\begin{abstract}
  The Network of Noisy Leaky Integrate and Fire (NNLIF) model
  describes the behavior of a neural network at mesoscopic level. It
  is one of the simplest self-contained mean-field models considered
  for that purpose. Even so, to study the mathematical properties of
  the model some simplifications were necessary \cite{CCP,CP,CS17},
  which disregard crucial phenomena. In this work we deal with the
  general NNLIF model without simplifications. It involves a network
  with two populations (excitatory and inhibitory), with transmission
  delays between the neurons and where the neurons remain in a
  refractory state for a certain time.
  We have studied the number of steady states in terms of the model
  parameters, the long time behaviour via the entropy method and
  Poincar\'e's inequality,
  blow-up phenomena, and the importance of 
  transmission delays between excitatory neurons to prevent blow-up
  and to give rise to synchronous solutions.  Besides analytical
  results, we have presented a numerical resolutor for this model,
  based on high order flux-splitting WENO schemes and an explicit
  third order TVD Runge-Kutta method, in order to describe the wide
  range of phenomena exhibited by the network: blow-up,
  asynchronous/synchronous solutions and instability/stability of the
  steady states; the solver also allows us to observe the time
  evolution of the firing rates, refractory states and the probability
  distributions of the excitatory and inhibitory populations.
\end{abstract}

\hspace{0.5cm}\rule{65mm}{0.1mm}

{\footnotesize
\hspace{0.5cm} \noindent \emph {2010 Mathematics Subject Classification.}  
35K60, 35Q92, 82C31, 82C32, 92B20

\hspace{0.5cm}\noindent \emph{Key words and phrases.} 
 Neural networks; Leaky integrate and fire models; noise; blow-up;
 steady states;

\hspace{0.5cm}\noindent entropy; long time behaviour; 
refractory states; transmission delay.




$\ $
}

\vspace{4cm}

\bigskip{\scriptsize 
  \noindent
  \textsc{Mar\'\i a J.~C\'aceres, Departamento de Matemática Aplicada,
    Campus de Fuentenueva,
    Universidad de Granada, 18071 Granada, Spain.
Phone: +34 958246301.}
  \textit{E-mail address}: \texttt{\href{mailto:caceresg@ugr.es}{caceresg@ugr.es}}

  \medskip
  \noindent
  \textsc{Ricarda Schneider, Departamento de Matemática Aplicada,
    Campus de Fuentenueva,
    Universidad de Granada, 18071 Granada, Spain.
Phone: +34 958240509.}
  \textit{E-mail address}:
  \texttt{\href{mailto:ricardaschneider@ugr.es}{ricardaschneider@ugr.es}}

\noindent
\thanks{Corresponding author: caceresg@ugr.es}
}

\newpage




\section{Introduction}

{ A wide variety of models have been usually considered in
  neuroscience, but their mathematical pro\-per\-ties remain poorly
  understood.  Mathematical studies on these models have advanced
  rapidly in the recent years, shedding light in this direction.  } In
this line, we analyze in this paper the Network of Noisy Leaky
Integrate and Fire (NNLIF) model, which describes the behavior of a
neural network at mesoscopic level and is one of the simplest
self-contained mean-field models used for that purpose.
We refer to \cite{BrHa,RBW,Touboul_AQIF, BrGe,RB,brunel, Touboul_2008,
  Henry:13, GK, T, G}, and references therein, for a background on
Integrate and Fire neuron models.

 This mesoscopic model is based on a nonlinear system of two Partial
  Differential Equations (PDEs) of Fokker-Planck type and two Ordinary
  Differential Equations (ODEs), which are all nonlinearly
  coupled. Moreover, some terms include time delays.  The system
  describes the behaviour of a network with excitatory and inhibitory
  neurons, which are considered as different populations.  Thus, the
  unknowns of the system are the probability densities
  $\rho_\alpha(t,v)$ of finding a neuron of the excitatory population
  ($\alpha=E$) and the inhibitory one ($\alpha=I$), whose membrane
  potential is $v$ at time $t$; together with the refractory states
  $R_\alpha(t)$, one for each population, which represent the 
  proportion of neurons that does not respond to stimuli.

 As a starting point, crucial phenomena have been disregarded in
  order to deal with this model. For example, transmission delay of
  the neural spike, the existence of refractory states, or the fact
  that there are two populations, have been neglected in order to
  simplify it \cite{CCP,CP,CS17}.  The simplest NNLIF model, widely
  studied in \cite{CCP,carrillo2014qualitative,CGGS}, corresponds to
  the case in which the neural network is assumed to be composed just
  by one population, which can be excitatory or inhibitory (in
  average), and where the neurons always respond to stimuli.  In
  mathematical terms this is translated into a unique PDE, with a
  \emph{connectivity parameter} $b$ whose sign determines whether the
  population is excitatory (positive $b$) or inhibitory (negative
  $b$). Many works have been developed in order to make the model
  more realistic: 
in \cite{CP}, {the authors} analyzed a model for one population
including the refractory state; in \cite{CS17}, a model for two
populations was considered; and in \cite{brunel}, a quite complete
model was studied that includes either one or two populations,
refractory states and transmission delays.

\

In the current work we aim to study a more realistic NNLIF model
consisting of two populations with refractory states and transmission
delays, completing the results of \cite{brunel}. We demonstrate that neural networks with part
of their neurons in a refractory state always have steady
states---which has been proved for the simpler case of only one
population \cite{CP}. This shows that in the complete model with
refractory states there is always at least one steady state, while in
the absence of refractory states \cite{CS17} there are some values of
the parameters for which the model has no steady states.  We are also
able to give conditions for the values of the model parameters which
ensure the uniqueness of the steady state.  This result is completed
with a proof of exponential convergence of the solution to the steady
state for networks with small connectivity parameters and without
transmission delay. The entropy method
\cite{carrillo2014qualitative,CS17} will be used to achieve this goal,
with the additional difficulty that we deal with a  complex system
 involving four equations, for which the entropy functional
is composed of excitatory and inhibitory densities and their
corresponding refractory probabilities.  Moreover, we extend to this case
the analysis of blow-up phenomena started in \cite{CP, CS17}.  We will
observe that the network can blow-up in finite time if the
transmission delay between excitatory neurons vanishes, even if there
are transmission delays between inhibitory neurons or between
inhibitory and excitatory neurons.  {Consequently, we show that the
  only way to avoid the blow-up is to consider a nonzero transmission
  delay between excitatory neurons.}  At the microscopic level, it is
known that global-in-time solutions exist if there is transmission
delay in the case of only one average-excitatory population (see
\cite{delarue2015global} and \cite{delarue2015particle}).

On the other hand, in order to better understand some of the
analytical open problems related to this model and show visually the
behaviour of the network, we develop a numerical solver for the full
model.
{Our solver is based on high order flux-splitting WENO schemes, TVD
  Runge-Kutta methods, and an efficient numerical strategy to deal
  with the saving and recovering of data needed to take the delays
  into account.  This new numerical solver improves our previous ones
  \cite{CCP,CP,CS17} not only because it describes the complete NNLIF
  model, but also due to it being optimized. It allows us to describe
  the wide range of phenomena displayed by the network: blow-up,
  asynchronous/synchronous solutions, instability/stability of the
  steady states, as well as the time evolution of the firing rates,
  the proportion of refractory states, and the probability
  distributions of the excitatory and inhibitory populations.
  Besides, we explore numerically the importance of the transmission
  delay between excitatory neurons to avoid the blow-up phenomenon;
  situations which present blow-up without delay are prevented it if a
  nonzero transmission delay is considered. Instead of blowing-up,
  solutions approach a stationary solution or \emph{synchronous state}.}

Our numerical scheme reproduces situations studied in \cite{brunel}
and completes them with the time evolution of the macroscopic (firing
rates and refractory states) and the mesoscopic quantities
(probability distributions). In this sense, our paper complements the
work in \cite{brunel} with the analysis of the number of steady
states, their stability for small connectivity parameters, the study
of the blow-up phenomenon and a numerical solver, which describes the
evolution in time of the system.

{To our knowledge, the numerical solver presented in this paper is the
  first deterministic solver to describe the behavior of the full
  NNLIF system including all the characteristic phenomena of real
  networks. Including all relevant phenomena is essential to explore
  some open problems, as for instance the stability in the case of
  large connectivity parameters, the importance of the transmission
  delay to avoid the blow-up of the solutions and to produce periodic
  solutions or the study of conditions for which synchronous solutions
  appear.}

In the rest of this introduction we describe the model and the concept
of solution considered. In Section \ref{sec: steady} we analyze the
number of steady states, prove exponential convergence to the unique
stationary solution when the connectivity parameters are small enough,
and present a criterion to obtain solutions that blow-up in finite
time.  All of these results are illustrated in Section \ref{sec:
  numerics}, where we present our numerical scheme and explore the
complex dynamics of the NNLIF model.



\subsection{The model}

Let us consider a neural network composed of an excitatory population
 and an inhibitory population. 
We denote by $\rho_\alpha(v,t)$  the probability density of finding 
a neuron in the population $\alpha$,
 with a voltage 
$v \in (-\infty, V_F]$  at a time $t\ge 0$,
 where $\alpha=E$, if the population
is excitatory, and $\alpha=I$, if it is inhibitory. 
We also consider the  NNLIF model \cite{brunel, CS17}
to describe the network, taking into account the 
transmission delay and the refractory state.
We obtain a complicated system of two PDEs 
for the evolution of these proba\-bi\-lity densities $\rho_\alpha(v,t)$, 
 coupled with another two ODEs
for the refractory states,  $R_\alpha(t)$,  for $\alpha=E,I$: 
\begin{equation}
\left \{ 
\begin{split}
\frac{\partial\rho_\alpha(v,t)}{\partial t}+ 
\frac{\partial}{\partial v} [h^\alpha(v,N_\alpha(t-D_E^\alpha),N_I(t-D_I^\alpha))
\rho_\alpha(v,t)]-  a_\alpha(N_E(t-D_E^\alpha), N_I(t-  &  D_I^\alpha))  
 \frac{\partial^2\rho_\alpha(v,t)}{\partial v^2}
=  
\\
& M_\alpha(t)\delta(v-V_R),
\\
\frac{dR_\alpha(t)}{dt} = N_\alpha(t)-M_\alpha(t), \qquad \qquad  \qquad \qquad \qquad \qquad  \qquad \qquad    \qquad \qquad \qquad  & 
\\
\\
N_\alpha (t)  =  -a_\alpha (N_E(t-D_E^\alpha), N_I(t-D_I^\alpha))
\frac{\partial \rho_\alpha }{\partial v}(V_F,t) \ge 0, \qquad    \qquad \qquad \qquad \qquad     &  
\\
\\
  \rho_\alpha(-\infty,t)=0, \quad \rho_\alpha(V_F,t)=0, \quad \rho_\alpha(v,0)=\rho_\alpha^0(v) \ge 0, \quad R_\alpha(0)=R_\alpha^0.  \quad   \qquad & 
\\
\label{modelo}
\end{split} 
\right.
\end{equation}
For each population $\alpha$, $R_\alpha(t)$  denotes the probability 
 to find a neuron  in the refractory state and $D_i^\alpha$, for
$i=E,I$, is the transmission delay of a spike arriving at a neuron
of population $\alpha$, coming from a neuron of population $i$.
The drift and 
diffusion coefficients are defined by
\begin{eqnarray}
h^\alpha (v,N_E(t),N_I(t))  & = &  
-v+b_E^\alpha N_E(t)-b_I^\alpha N_I(t)+(b_E^\alpha-b_E^E) \nu_{E ,ext}
, 
\label{drift}
\\ 
a_\alpha(N_E(t),N_I(t)) & = & {d}_\alpha + 
{d}_E^\alpha N_E(t)  
+
{d}_I^\alpha  N_I(t),  \quad \alpha=E,I,
 \label{diffusion}
\end{eqnarray}
 where, for $i,\alpha=E,I$, $b_i^\alpha>0$, ${d}_\alpha>0$ and $d_i^\alpha \ge 0$,
and  $b_i^\alpha$ are the connectivity parameters for a spike emitted 
by a neuron of population $i$ and arriving at a neuron of population 
$\alpha$, and $\nu_{E ,ext}\ge 0$ describes the external synapses.
Both populations (excitatory and inhibitory)
are  coupled by means of the drift
and diffusion coefficients.
Moreover, the system \eqref{modelo} is nonlinear because 
 the firing rates, $N_\alpha$,  are defined in terms of the boundary 
conditions  for $\rho_\alpha$.


Denoting the refractory period $\tau_\alpha$, 
different choices of  $M_\alpha(t)$ can be considered:
$M_\alpha(t)=N_\alpha(t-\tau_\alpha)$ (studied in \cite{brunel}),
and
$M_\alpha(t)=\frac{R_\alpha(t)}{\tau_\alpha}$ (analyzed in \cite{CP}). 
Depending on the refractory state used, 
slightly different behaviors of the solutions will appear.

On the other hand, since the number of neurons is assumed to be 
preserved, we have the conservation law: 
\begin{eqnarray}\label{masa}
\int_{-\infty}^{V_F}\rho_{\alpha}(v,t) \ dv + R_\alpha(t) = 
\int_{-\infty}^{V_F}\rho_{\alpha}^0(v) \ dv+R_\alpha^0 = 1 \quad \forall \ t\ge 0, 
\quad \alpha=E,\, I.
\end{eqnarray}
To finish the description of the model, we remark that
system \eqref{modelo} also includes the case
of only one population (in average excitatory or inhibitory),
with refractory state and transmission delay.
Specifically, we can remove $\alpha$ in \eqref{modelo} 
considering only one PDE for the probability density, $\rho(v,t)$, 
which is coupled to an ODE for the probability that a neuron is in a  
refractory state, $R(t)$:
\begin{equation}
\left \{ 
\begin{split}
& \frac{\partial\rho}{\partial t}(v,t)+
\frac{\partial}{\partial v} [h(v,N(t-D))
\rho(v,t)]-  a(N(t-D))  
\frac{\partial^2\rho}{\partial v^2}(v,t) 
= M(t)\delta(v-V_R),
\\
& \frac{dR(t)}{dt} = N(t)-M(t), 
\\
& N(t)  =  -a(N(t-D))
\frac{\partial \rho}{\partial v}(V_F,t) \ge 0, 
\\
&  \rho(-\infty,t)=0, \quad \rho(V_F,t)=0, \quad \rho(v,0)=
\rho^0(v) \ge 0, \quad R(0)=R^0, 
\\
\label{nota}
\end{split} \right.
\end{equation}
with drift and diffusion 
terms 
\begin{eqnarray}
h (v,N(t))  & = &  
-v+b N(t)+\nu_{ext}, 
\label{drift_E}
\\ 
a(N(t)) & = & d_0 + 
{d}_1N_E(t), 
 \label{diffusion_E}
\end{eqnarray}
where the connectivity parameter $b$ is positive  
 for an average-excitatory population and negative  
for an average-inhibitory population, and where  $d_0>0$,  $d_1 \ge0$,
and $\nu_{ext}$ describes the external synapses (note that this parameter
and $\nu_{E ,ext}$ have different units, since $\nu_{ext}$
includes other model constants).


As in \cite{CCP,CP,CS17}, the notion of solution that we consider
 is the following:
\begin{definition}
Let $\rho_\alpha \in   L^\infty(\mathbb{R}^+;L^1_+((-\infty, V_F)))$, 
$N_\alpha \in \ L^1_{loc,+}(\mathbb{R}^+)$ and $R_\alpha \in L^\infty_+(\mathbb{R}^+)$ for  $\alpha=E,I$. Then $(\rho_E, \rho_I, R_E, R_I, N_E, N_I)$ is a weak solution of (\ref{modelo})-(\ref{diffusion}) if for any test function $\phi(v,t) \in  C^\infty((-\infty, V_F]  \times [0,T])$ 
and such that 
$\frac{\partial^2 \phi}{\partial v^2}, 
\ v\frac{\partial \phi}{\partial v} \in L^\infty((-\infty, V_F) \times (0, T))$ 
the following relation 
\begin{equation} \label{debil}
\begin{split}
&\int_{0}^{T} \!\! 
\int_{-\infty}^{V_F} \!\!
\rho_\alpha (v,t) 
\left[-\frac{\partial \phi}{\partial t} -
\frac{\partial \phi}{\partial v} h^\alpha(v,N_E(t-D_E^\alpha),N_I(t-D_I^\alpha))-
a_\alpha(N_E(t-D_E^\alpha),N_I(t-D_I^\alpha))\frac{\partial^2 \phi}{\partial v^2}\right] \!\! dv dt 
 \\ 
&= 
\int_{0}^{T}\!\! [M_\alpha(t)\phi(V_R,t)-N_\alpha(t)\phi(V_F,t)]dt+ 
\int_{-\infty}^{V_F} 
\!\!
\rho_\alpha^0(v)\phi(v,0) dv - \int_{-\infty}^{V_F}
\!\!
\rho_\alpha(v,T)\phi(v,T) \,dv 
\end{split}
\end{equation}
is satisfied $\forall \ \alpha=E,I$, 
and $R_\alpha$, for $\alpha=E,I,$ are solutions of the ODEs
\begin{equation*}
\frac{dR_\alpha(t)}{dt}=N_\alpha(t)-M_\alpha(t).
\end{equation*}
\end{definition}
We recall some notations involved in Definition \ref{debil}.
For $1\leq p < \infty$,  $L^p(\Omega)$ is the space of functions such 
that $f^p$ is integrable in $\Omega$, 
$L^\infty (\Omega)$ is the space of essentially bounded functions in $\Omega$, 
$L^\infty_+ (\Omega)$ represents the space of non-negative essentially bounded functions in $\Omega$, 
$C^\infty(\Omega)$ is the set of infinitely differentiable functions in 
$\Omega$ and $L^1_{loc,+}(\Omega)$ denotes the set of non-negative 
functions that are locally integrable  in $\Omega$.

\section{Steady states and long time behavior}

\label{sec: steady}


The study of the number of steady states
for excitatory and inhibitory NNLIF neural networks,
with refractory periods and transmission delays of the spikes
\eqref{modelo}  (considering $R_\alpha$ either as defined in \cite{CP} or
  in \cite{brunel}), can be done combining the ideas of \cite{CCP, CP} and 
\cite{CS17},  with the 
additional difficulty that the system to be dealt
with is now more complicated.
The steady states $(\rho_E, \rho_I, N_E,N_I, R_E,R_I)$ of 
\eqref{modelo} satisfy
\begin{equation*}
\frac{\partial}{\partial v} [h^\alpha(v)\rho_\alpha(v)-a_\alpha(N_E,N_I)
\frac{\partial \rho_\alpha}{\partial v}(v)+
\frac{R_\alpha}{\tau_\alpha} H(v-V_R)]=0, \quad 
R_\alpha=\tau_\alpha N_\alpha, \qquad \alpha=E,I,
\end{equation*}
in the sense of distributions, with $H$ denoting the Heaviside function
and $h^\alpha(v, N_E,N_I)=V_0^\alpha(N_E,N_I)-v$,
where
$V_0^\alpha (N_E,N_I)=b_E^\alpha N_E - b_I^\alpha N_I +
(b_E^\alpha-b_E^E)\nu_{E,ext}$. 
We remark that this equation is the same as the equation
for  stationary solutions in a network without transmission delays.
Using the definition of $N_\alpha$ and the Dirichlet bounday
 conditions of \eqref{modelo} we obtain an initial value problem
 for every $\alpha=E,I$, 
whose  solutions are
\begin{equation}\label{p_steady}
\rho_\alpha(v)=\frac{N_\alpha}{a_\alpha(N_E,N_I)} 
e^{-\frac{(v-V_0^\alpha(N_E,N_I))^2}{2a_\alpha(N_E,N_I)}}
\int_{\max(v,V_R)}^{V_F} 
e^{\frac{(w-V_0^\alpha (N_E,N_I))^2}{2a_\alpha(N_E,N_I)}}\ dw \qquad \alpha=E,I.
\end{equation}
Moreover, the conservation of mass \eqref{masa}, which takes into account
 the refractory states, 
 yields a system of implicit
 equations for $N_\alpha$
\begin{equation}
1-\tau_\alpha N_\alpha=
\frac{N_\alpha}{a_\alpha(N_E,N_I)}\int_{-\infty}^{V_F} 
e^{-\frac{(v-V_0^\alpha(N_E,N_I))^2}{2a_\alpha(N_E,N_I)}} 
\int_{\max(v,V_R)}^{V_F}e^{\frac{(w-V_0^\alpha (N_E,N_I))^2}{2a_\alpha(N_E,N_I)}}\ dw 
\ dv.
\label{*}
\end{equation}
If this system could be solved, the profile \eqref{p_steady}
would provide an exact expression for $\rho_\alpha$.
In order to handle 
the previous system  more easily, 
 we use  two
changes of variables as in \cite{CS17}.
First:
\begin{eqnarray}
z&\! \! \! \! \! =\! \! \! \! \! &
\frac{v-V_0^E(N_E,N_I)}{\sqrt{a_E (N_E,N_I)}},  
\ u=\frac{w-V_0^E(N_E,N_I)}{\sqrt{a_E(N_E,N_I)}},  
\ w_F:=\frac{V_F-V_0^E(N_E,N_I)}{\sqrt{a_E(N_E,N_I)}},   \
w_R:=\frac{V_R-V_0^E(N_E,N_I)}{\sqrt{a_E(N_E,N_I)}},
\nonumber \\
\tilde{z}&\! \! \! \! \! =\! \! \! \! \! &
\frac{v-V_0^I(N_E,N_I)}{\sqrt{a_I(N_E,N_I)}},  \
 \tilde{u}=\frac{w-V_0^I(N_E,N_I)}{\sqrt{a_I(N_E,N_I)}},  \
 \tilde{w}_F:=\frac{V_F-V_0^I(N_E,N_I)}{\sqrt{a_I(N_E,N_I)}},   \
 \tilde{w}_R:=\frac{V_R-V_0^I(N_E,N_I)}{\sqrt{a_I(N_E,N_I)}},
\nonumber
\end{eqnarray}
and \eqref{*}  is then written as
\begin{eqnarray}
\frac{1}{N_E}- \tau_E & = & I_1(N_E,N_I),
\ \mbox{where} \ I_1(N_E,N_I)  =  
\int_{-\infty}^{w_F}e^{-\frac{z^2}{2}}
\int_{\max(z,w_R)}^{w_F}e^{\frac{u^2}{2}}du \ dz,
 \nonumber \\
\frac{1}{N_I} - \tau_I& = & I_2(N_E,N_I), 
\ \mbox{where} \
I_2(N_E,N_I) = \int_{-\infty}^{\tilde{w}_F}
e^{-\frac{z^2}{2}}\int_{\max(z,\tilde{w}_R)}^{\tilde{w}_F}
e^{\frac{u^2}{2}}du \ dz,
\label{ecN2}
\end{eqnarray}
with the additional restrictions
\begin{equation}
N_\alpha<\frac{1}{\tau_\alpha} \quad \alpha=E,I,
\label{ec: C-N}
\end{equation}
since
$R_\alpha=\tau_\alpha N_\alpha$ and $R_\alpha< 1$
(we also observe these restrictions by the positivity of
$I_\alpha$, see \eqref{ecN2}).
Next, the change of variables 
$s=\frac{z-u}{2}$ and $\tilde{s}=\frac{z+u}{2}$
allows to formulate the functions
$I_1$ and $I_2$  as
\begin{eqnarray}
I_1(N_E,N_I)=
\int_0^\infty 
\frac{e^{-\frac{s^2}{2}}}{s}(e^{s \, w_F}-e^{s \, w_R}) \ ds,
 \label{I1}
\\
I_2(N_E,N_I)=\int_0^\infty \frac{e^{-\frac{s^2}{2}}}{s}
(e^{s \, \tilde{w}_F}-e^{s \, \tilde{w}_R}) \ ds. \label{I2}
\end{eqnarray}
If $b_I^E=b_E^I=0$ the equations are uncoupled
 and the number of steady states can 
be studied in terms of the values of $b_E^E$, due to fact
 that for the inhibitory equation there is always a unique steady state
 \cite{CP}. The following theorem analyses the coupled case.
\begin{theorem}
\label{th: steady states}
Assume that $b_I^E>0, \, b_E^I>0, \, \tau_E>0, \, \tau_I>0$, 
 $a_\alpha(N_E, N_I)=a_\alpha$ constant,
 and $h^\alpha(v,N_E,N_I)=V_0^\alpha (N_E,N_I)-v$ with
$V_0^\alpha (N_E,N_I)=b_E^\alpha N_E - b_I^\alpha N_I + 
(b_E^\alpha-b_E^E)v_{E,ext}$ for all $\alpha = E,I$.
Then there is always an odd number of steady states for \eqref{modelo}.

Moreover, if $b_E^E$ is small enough or $\tau_E$ is large enough 
(in comparison with the rest of parameters), then
there is a unique steady state for \eqref{modelo}.
\end{theorem}
\proof
The proof 
is based on determining
the number of solutions of the system
\begin{align}
1 & = N_E\left(\tau_E+I_1(N_E,N_I)\right), \quad 0<N_E<\frac{1}{\tau_E},
\label{ec:system-IE}
\\
1 & = N_I\left(\tau_I+I_2(N_E,N_I)\right), \quad 0<N_I<\frac{1}{\tau_I}.
\label{ec:system-II}
\end{align}
With this aim, we adapt some ideas of \cite{CP} and \cite{CS17}
to the system 
\eqref{ec:system-IE}-\eqref{ec:system-II}. 
We refer to \cite{CS17} for details about the 
properties
of the functions $I_1$ and $I_2$ (see \eqref{I1} and \eqref{I2})
and their proofs. 

First, we observe that for every $N_E>0$ fixed, there is a unique solution
 $N_I(N_E)$ that solves \eqref{ec:system-II},
because for  $N_E>0$ fixed, the function 
$f(N_I)= N_I\left(\tau_I+I_2(N_E,N_I)\right)$ satisfies:
$f(0)=0$, $f(\frac{1}{\tau_I})=
1+\frac{I_2(N_2,\frac{1}{\tau_I})}{\tau_I}>1$
and is increasing, since $I_2(N_E,N_I)$ is an increasing, strictly
 convex function on $N_I$.

Then, taking into account that the function 
$\mathcal{F}(N_E):= N_E[I_1(N_E,N_I(N_E)+\tau_E]$ satisfies that
$\mathcal{F}(0)=0$ and 
$\mathcal{F}(\frac{1}{\tau_E})= 
1+\frac{I_1\left(\frac{1}{\tau_E},N_I(\frac{1}{\tau_E})\right)}{\tau_E}>1$,
it can be concluded that there is always an odd number of steady states.

Finally, to obtain values of the parameters such that there is
a unique steady state,  we analyze
the derivative of $\mathcal{F}$: 
\begin{equation}
\mathcal{F}'(N_E)= I_1(N_E,N_I(N_E))+\tau_E +
N_E\left[-\frac{b_E^E}{\sqrt{a_E}}+\frac{b_I^E}{\sqrt{a_E}}N_I'(N_E)\right]
\int_0^\infty e^{\frac{-s^2}{2}}(e^{sw_F}-e^{sw_R}) \ ds.
\nonumber
\end{equation}
It is non-negative for $0<N_E<\frac{1}{\tau_E}$,
for certain parameter values,  and therefore there is a unique 
steady state in these cases.
For $b_E^E$ small,  $\mathcal{F}'(N_E)$ is positive
since all the terms are positive, because $N'(N_E)$ is positive
(see the proof of Theorem 4.1 in \cite{CS17}).
For $\tau_E$ large, the proof of the positivity of $\mathcal{F}'(N_E)$
is 
more complicated. It is necessary to use 
\begin{equation}
N_I'(N_E)=
\frac{b_E^I N_I^2(N_E) I(N_E)}
{\sqrt{a_I}+ b_I^I N_I^2(N_E)I(N_E)},
\label{N'bis}
\end{equation}
where
$$
I(N_E)=
\int_0^\infty e^{-s^2/2}
e^{\frac{-(b_E^IN_E-b_I^IN_I(N_E)+(b_E^I-b_E^E)\nu_{E,ext})s}{\sqrt{a_I}}}
\left( e^{sV_F/\sqrt{a_I}}-e^{sV_R/\sqrt{a_I}} \right) \ ds.
$$
The function $N_I(N_E)$ is increasing and $I(N_E)$ is decreasing, since $0<N_I'(N_E)<\frac{b_E^I}{b_I^I}$
(see the proof of  Theorem 4.1 in \cite{CS17}).
Therefore, for $0<N_E<\frac{1}{\tau_E}$,  
\begin{equation}
A<-\frac{b_E^E}{\sqrt{a_E}}+\frac{b_I^E}{\sqrt{a_E}}N_I'(N_E)<B,
\nonumber
\end{equation}
where $A:=-\frac{b_E^E}{\sqrt{a_E}}+\frac{b_I^E}{\sqrt{a_E}}
\frac{b_E^I N_I^2(0) I(\frac{1}{\tau_E})}
{\sqrt{a_I}+ b_I^I N_I^2(\frac{1}{\tau_E})I(0)}$ and 
$B:=-\frac{b_E^E}{\sqrt{a_E}}+\frac{b_I^E}{\sqrt{a_E}}
\frac{b_E^I N_I^2(\frac{1}{\tau_E}) I(0)}
{\sqrt{a_I}+ b_I^I N_I^2(0)I(\frac{1}{\tau_E})}$.
Thus, if $0\le A$ it is obvious that $\mathcal{F}(N_E)$ is increasing.
For the case $A<0$, some additional computations are needed.
First, we consider $I_m:=\min_{0\le N_E \le
\frac{1}{\tau_E}}I_1(N_E,N_I(N_E))$.
Next,  since $A<0$,
\begin{equation}
I_m+\tau_E+\frac{A}{\tau_E} \tilde{I}(\tau_E)
\le 
\mathcal{F}'(N_E),
\nonumber
\end{equation}
where $\tilde{I}(\tau_E):=\displaystyle\int_0^\infty
e^{-\frac{s^2}{2}} e^{\frac{s b_I^E N_I(\frac{1}{\tau_E})}{\sqrt{a_E}}}
\left( 
e^{\frac{s V_F}{\sqrt{a_E}}}-
e^{\frac{s V_R}{\sqrt{a_E}}}
\right) \, ds$. Finally, if
$0<I_m+\tau_E+\frac{A}{\tau_E} \tilde{I}(\tau_E)$, or equiva\-lently
$-A \tilde{I}(\tau_E)<\tau_E(I_m+\tau_E)$,
then $\mathcal{F}(N_E)$ is increasing.
We observe that it happens for $\tau_E$ large enough. 
\qed
\begin{remark}
Analyzing in more detail the expression of $A$
in the previous proof ($A=-\frac{b_E^E}{\sqrt{a_E}}+\frac{b_I^E}{\sqrt{a_E}}
\frac{b_E^I N_I^2(0) I(\frac{1}{\tau_E})}
{\sqrt{a_I}+ b_I^I N_I^2(\frac{1}{\tau_E})I(0)}$),
we observe that for $b_E^Ib_I^E$ large or
$b_I^I$ small enough, in comparison with
the rest of parameters, there is also a unique stationary solution,
since $A>0$.

In other words, what we obtain is the uniqueness of the
steady state in terms of the size of
the parameters. More precisely: If one of the two pure
connectivity parameters, $b_E^E$ or $b_I^I$, is small, or
one of the two cross connectivity parameters,  $b_E^I$ or $b_E^I$,
is large, or the excitatory refractory period, $\tau_E$, is
large, then there exists a unique steady state.
\label{remark}
\end{remark}
\subsection{Long time behavior}
As proved in \cite{carrillo2014qualitative, CS17},
where no refractory states were considered, the
solutions converge exponentialy fast to the unique steady
state when the connectivity parameters are small enough.
We extend these results to the case in which refractory states are
included. We prove the result for the case of only one population 
in the following theorem, and then show the general case of two
populations.
\begin{theorem}
Consider system \eqref{nota}
 and $M(t)=\frac{R(t)}{\tau}$. Assume that the connectivity parameter 
$b$ is small enough, $|b|<<1$,  the diffusion term is constant, 
$a(N)=a$ for some $a>0$, there is no
transmission delay, $D=0$, and that the initial datum is close enough to 
the unique steady state $\left(\rho_\infty, R_\infty, N_\infty\right)$,
\begin{equation} \label{dato_inicial}
\int_{-\infty}^{V_F}\rho_\infty (v)
\left(\frac{\rho^0(v)-\rho_\infty(v)}{\rho_\infty(v)}\right)^2  \  dv +
R_\infty\left(\frac{R(0)}{R_\infty}-1\right)^2\leq \frac{1}{2|b|}.
\end{equation}
Then, for fast decaying solutions to \eqref{nota}
 there is a constant $\mu>0$ such that for all $t\ge0$
\begin{equation*}
\int_{-\infty}^{V_F} \!\! \!\!
\rho_\infty (v)
\left( \frac{\rho(v)-\rho_\infty(v)}{\rho_\infty(v)}\right)^2
\!\!  dv + \frac{(R(t)-R_\infty)^2}{R_\infty} \leq e^{-\mu t} 
\left[ \int_{-\infty}^{V_F} \!\! \!\!
\rho_\infty \left( \frac{\rho^0(v)-\rho_\infty(v)}{\rho_\infty(v)}
\right)^2 \!\! dv+ \frac{(R^0-R_\infty)^2}{R_\infty}\right].
\end{equation*}
\label{th: long1}
\end{theorem}

\begin{proof}
The proof  
combines  a relative entropy 
argument with the  Poincar\'e's inequality 
that is presented in  
\cite{CP}[Proposition 5.3].  Additionally,  
to deal with the nonlinearity (the connectivity parameter does not vanish)
 we follow some ideas of
 \cite{carrillo2014qualitative}[Theorem 2.1].
Notice that along the proof we will use the simplified notation
\begin{equation*}
p(v,t)=\frac{\rho(v,t)}{\rho_\infty(v)}, \qquad r(t)=\frac{R(t)}{R_\infty}, 
\qquad \eta(t)=\frac{N(t)}{N_\infty}.
\end{equation*}
First, for any smooth convex function $G: \mathbb{R}^+ \rightarrow \mathbb{R}$,
we recall that a natural relative entropy for equation \eqref{nota}  is defined as
\begin{equation}\label{energia}
E(t):= \int_{-\infty}^{V_F}\rho_\infty G(p(v,t)) \ dv +R_\infty G(r(t)).
\end{equation}
The time derivative of  the relative entropy \eqref{energia} can be written as
\begin{align}\label{derivada}
\frac{d}{dt}E(t)   = & -a\int_{-\infty}^{V_F} 
\rho_\infty(v)G''(p(v,t))\left[\frac{\partial p}{\partial v}\right]^2(v,t) \ dv \nonumber
\\
\  & -N_\infty\left[G(\eta(t))-G(p(V_R,t))-(r(t)-p(V_R,t))G'(p(V_R,t))-
(\eta(t)-r(t))G'(r(t)\right]
\\
 & + b(N(t)-N_\infty)
\int_{-\infty}^{V_F}\frac{\partial \rho_\infty}{\partial v}(v) 
\left[G(p(v,t))-p(v,t)G'(p(v,t))\right] \ dv. \nonumber
\end{align}
Expression \eqref{derivada} is achieved after some simple computations, 
taking into account that $(\rho,R,N)$ is a solution of equation 
\eqref{nota} and that $(\rho_\infty, R_\infty, N_\infty)$
is the unique steady state  of the same equation, thus given by
\begin{equation*}
\left \{ 
\begin{split}
&\frac{\partial}{\partial v} [h(v,N_\infty)
\rho_\infty(v)]-  a  
\frac{\partial^2\rho_\infty}{\partial v^2}(v) 
= \frac{R_\infty}{\tau}\delta(v-V_R),
\\
& R_\infty=\tau N_\infty,
\quad N_\infty  =  -a
\frac{\partial \rho_\infty}{\partial v}(V_F) \ge 0, 
\\
&  \rho_\infty(-\infty)=0, \quad \rho_\infty(V_F)=0.
\end{split} \right.
\end{equation*}
Specifically, we can obtain sucessively the following relations:
\begin{align}\label{rel_A}
 \frac{\partial p}{\partial t}-\left(v-b N +\frac{2a}{\rho_\infty}
\frac{\partial \rho_\infty}{\partial v} \right)
\frac{\partial p}{\partial v}-a\frac{\partial^2 p}{\partial v^2} 
= 
\frac{R_\infty}{\tau \rho_\infty}
\delta(v-V_R)
\left(r-p\right) - \frac{p}{\rho_\infty}b(N-N_\infty)
\frac{\partial \rho_\infty}{\partial v},
\end{align}
\begin{align}\label{rel_B}
\frac{\partial G\left(p\right)}{\partial t} 
-\left(v-b N +\frac{2a}{\rho_\infty}
\frac{\partial \rho_\infty}{\partial v}\right)
\frac{\partial  G\left(p\right)}{\partial v}
&-
a\frac{\partial^2  G\left(p\right)}{\partial v^2}
= -G'\left(p\right)\frac{p}{\rho_\infty}b(N-N_\infty)\frac{\partial \rho_\infty}{\partial v}
\nonumber\\
-&aG''\left(p\right)\left(\frac{\partial p}{\partial v}\right)^2
+G'\left(p\right)\frac{R_\infty}{\tau \rho_\infty}\delta(v-V_R)\left(r-p
\right),
\end{align}
and
\begin{align}\label{rel_C}
\frac{\partial}{\partial t}\rho_\infty G\left(p\right)
&-\frac{\partial}{\partial v}\left[(v-bN)\rho_\infty 
G\left(p\right)\right]
-a\frac{\partial^2}{\partial v^2}\left[\rho_\infty 
G\left(p\right)\right]
=b(N-N_\infty)\frac{\partial 
\rho_\infty}{\partial v}\left[G\left(p\right)-p G'\left(p\right)\right]
 \nonumber \\
&-a\rho_\infty G''\left(p\right)\left(\frac{\partial p}{\partial v} \right)^2
+\frac{R_\infty}{\tau}\delta(v-V_R)\left[\left(r-p\right)
G'\left(p\right)
+G\left(p\right)\right].
\end{align}
Finally, \eqref{derivada} is obtained after integrating \eqref{rel_C}
 with respect to $v$, between $-\infty$ and $V_F$, taking into account that
\begin{equation*}
a \frac{\partial}{\partial v} 
\left[\rho_\infty G\left(p\right)\right]_{v=V_F}
=-N_\infty  G\left(\eta\right),
\end{equation*}
due to the boundary condition at $V_F$ and the l'Hopital rule,
and adding 
\begin{align}\label{rel_E}
\frac{d}{dt}R_\infty G\left(r\right)=
\frac{R_\infty}{\tau}R_\infty G'\left(r\right)\left(\eta-r\right).
\end{align}
To obtain the exponential rate of convergence stated in the theorem, 
we consider $G(x)=(x-1)^2$ in  \eqref{derivada}.
Its first term  is negative and  will provide the strongest 
control when combined with the   Poincar\'e's  inequality.
After some algebraical computations, the second term can be written as
\begin{align*}
 -N_\infty[G(\eta(t))-G(p(V_R,t)) & -(r(t)-p(V_R,t))G'(p(V_R,t))-
(\eta(t)-r(t))G'(r(t)]
\\
=& -N_\infty[(r(t)-\eta(t))^2 + (r(t)-p(V_R,t))^2].
\end{align*}
Applying the inequality $(a+b)^2\ge\epsilon(a^2-2b^2)$, 
for $a,b \in \R$ and $0<\epsilon<\frac{1}{2}$, we obtain
\begin{equation}\label{sumando1}
 -N_\infty(r(t)-\eta(t))^2 \leq -\epsilon N_\infty (\eta(t)-1)^2+2\epsilon N_\infty(r(t)-1)^2.
\end{equation}
Recalling the Poincar\'e's inequality of \cite{CP}[Proposition 5.3],
and in a similar way as in \cite{carrillo2014qualitative},
for small connectivity parameters,
there exists $\gamma>0$ such that:
\begin{equation}
\label{poincare}
\int_{-\infty}^{V_F}\frac{(\rho-\rho_\infty)^2}{\rho_\infty} dv+
\frac{(R-R_\infty)^2}{R_\infty} \leq 
\frac{1}{\gamma}\left[
\int_{-\infty}^{V_F} 
\rho_\infty(v)\left[\frac{\partial p}{\partial v}\right] ^2 (v,t) \ dv
+N_\infty(r(t)-p(V_R,t))^2
\right],
\end{equation}
thus
\begin{align}\label{sobolev-poincare}
(r(t)-1)^2 \leq \frac{1}{\gamma R_\infty}\int_{-\infty}^{V_F} \rho_\infty(v)\left[\frac{\partial p}{\partial v}\right] ^2 (v,t) \ dv
+\frac{N_\infty}{\gamma R_\infty}(r(t)-p(V_R,t))^2,
\end{align}
and therefore
\begin{equation}\label{sumando2}
2\epsilon N_\infty(r(t)-1)^2 \leq  \frac{2 \epsilon N_\infty}{\gamma R_\infty}\int_{-\infty}^{V_F} \rho_\infty(v)\left[\frac{\partial p}{\partial v}\right] ^2 (v,t) \ dv
+\frac{2 \epsilon N_\infty}{\gamma R_\infty}N_\infty(r(t)-p(V_R,t))^2.
\end{equation}
Joining now estimates \eqref{sumando1} and \eqref{sumando2}, choosing $0<\epsilon<\frac{1}{2}$ such that $\frac{2 \epsilon N_\infty}{\gamma R_\infty}<\min(\frac{a}{2},\frac{1}{2})$ and denoting $C_0:=\epsilon N_\infty$ yields
\begin{align}\label{segundo_sumando}
 & -N_\infty[G(\eta(t))-G(p(V_R,t))-(r(t)-p(V_R,t))G'(p(V_R,t))-(\eta(t)-r(t))
G'(r(t)]\nonumber
\\
\leq & -C_0G(\eta(t))+\frac{a}{2}\int_{-\infty}^{V_F} \rho_\infty(v)
\left[\frac{\partial p}{\partial v}\right] ^2 (v,t) \ dv-\frac{1}{2}N_\infty(r(t)-p(V_R,t))^2.
\end{align}
The third term can be bounded in the same way as in \cite{carrillo2014qualitative}. Thus, for some $C>0$ we have
\begin{align} \label{tercer_sumando}
 & b(N(t)-N_\infty)\int_{-\infty}^{V_F}\frac{\partial \rho_\infty}{\partial v}(v) [G(p(v,t))-p(v,t)G'(p(v,t))] \ dv \\
\leq & C(2b^2+|b|)(\eta(t)-1)^2+a \int_{-\infty}^{V_F} \rho_\infty 
\left[\frac{\partial p}{\partial v}\right]^2(v,t) \ dv
\left(\frac{1}{2}+|b|\int_{-\infty}^{V_F} \rho_\infty(v)(p(v,t)-1)^2 \ dv 
\right). \nonumber
\end{align}
Combining estimates \eqref{segundo_sumando} and \eqref{tercer_sumando} gives the bound
\begin{align*}
\frac{d}{dt}E(t) \leq & -C_0(\eta(t)-1)^2 +C(2b^2+|b|)(\eta(t)-1)^2 
-\frac{1}{2}N_\infty(r(t)-p(V_R,t))^2
\\& -a\int_{-\infty}^{V_F}\rho_ \infty(v) 
\left[ \frac{\partial p}{\partial v} \right] ^2 (v,t) \ dv 
\left(1-|b|\int_{-\infty}^{V_F}\rho_\infty(v)(p(v,t)-1)^2 \ dv \right).
\end{align*}
Taking now $b$ small enough such that $C(2b^2+|b|)\leq C_0$ 
we obtain
\begin{align*}
\frac{d}{dt}E(t) \leq & - \tilde{C} \left[\int_{-\infty}^{V_F}
\rho_ \infty(v) \left[ \frac{\partial p}{\partial v} \right] ^2 (v,t) \ dv  
+N_\infty(r(t)-p(V_R,t))^2\right]
\\& -\frac{a}{2}\int_{-\infty}^{V_F}\rho_ \infty(v) 
\left[ \frac{\partial p}{\partial v} \right] ^2 (v,t) \ dv 
\left(1-2|b|\int_{-\infty}^{V_F}\rho_\infty(v)(p(v,t)-1)^2 \ dv \right)
\\&\leq  -\mu E(t)
 -\frac{a}{2}\left(1-2|b|E(t) \right)
\int_{-\infty}^{V_F}\rho_ \infty(v) 
\left[ \frac{\partial p}{\partial v} \right] ^2 (v,t) \ dv,
\end{align*}
where  Poincar\'e's inequality \eqref{poincare} was used, 
with $\tilde{C}=\min(\frac{a}{2}, \frac{1}{2})$,  $\mu=\tilde{C}\gamma$.
Finally, thanks to the choice of the initial datum \eqref{dato_inicial}
and Gronwall's inequality, the relative entropy decreases 
for all times so that, $E(t)\leq\frac{1}{2|b|}$,  $\forall t\ge0$,
and the result is proved:
\begin{equation*}
E(t)\leq e^{-\mu t} E(0)\leq e^{-\mu t}\frac{1}{2|b|}.
\end{equation*}
\qed
\end{proof}
For two populations with refractory states (as given in model \cite{CP}),
this exponential rate of convergence to the unique steady can
also be proved.
The proof is achieved by considering the full
entropy for  both populations: 
\begin{align}
\mathcal{E}[t]   :=\int_{-\infty}^{V_F}\rho_E^\infty (v) \left( \frac{\rho_E(v) -
\rho_E^\infty (v) }{\rho_E^\infty (v) } \right)^2 \  dv 
& +
\int_{-\infty}^{V_F}\rho_I^\infty (v) \left( \frac{\rho_I(v) -
\rho_I^\infty (v)} {\rho_I^\infty (v) }\right)^2  \ dv 
\nonumber
\\
& +
\frac{(R_E(t)-R_E^\infty)^2}{R_E^\infty} 
+
\frac{(R_I(t)-R_I^\infty)^2}{R_I^\infty},
\nonumber
\end{align}
and proceeding in the same way as in \cite{CS17}[Theorem 4.2], 
 taking into account that now there are some terms with refractory states
which have to be handled, as in Theorem \ref{th: long1}.
\begin{theorem}
Consider system \eqref{modelo} for two populations, with
$M_\alpha(t)=\frac{R_\alpha(t)}{\tau_\alpha}$, $\alpha=I,E$. 
Assume that the connectivity parameters $b_i^\alpha$ are
small enough, the diffusion terms $a_\alpha>0$ are constant,
the transmission delays  $D_i^\alpha$ vanish ($\alpha=I,E$, $i=I,E$), 
 and that the initial data ($\rho_E^0, \rho_I^0$) are
close enough to the unique steady state 
($\rho_E^\infty, \rho_I^\infty$):
\begin{equation}
\nonumber 
\mathcal{E}[0]<\frac{1}{2\max\left(b_E^E+b_I^E,b_E^I+b_I^I \right)}.
\end{equation}
Then, for fast decaying solutions to \eqref{modelo}, 
there is a constant $\mu>0$ such that for all $t\ge0$
\begin{equation*}
\mathcal{E}[t]\le e^{-\mu t} \mathcal{E}[0].
\end{equation*}
Consequently, for $\alpha=E,I$
\begin{equation*}
\int_{-\infty}^{V_F}\rho_\alpha^\infty(v)  
\left( \frac{\rho_\alpha(v) -\rho_\alpha^\infty(v) }
{\rho_\alpha^\infty(v) }\right)^2 
dv + \frac{(R_\alpha(t)-R_\alpha^\infty)^2}{R_\alpha^\infty} 
\leq e^{-\mu t} \mathcal{E}[0].
\end{equation*}
\label{th:long2}
\end{theorem}

To conclude the study about the long time behavior we have to
remember that solutions to  \eqref{modelo} may blow-up in finite time if
there are no delays.
Specifically, following  similar
 steps as those developed in \cite{CP}[Theorem 3.1]
and \cite{CS17}[Theorem 3.1], we can prove an analogous result for
the general system \eqref{modelo} without delay between 
excitatory neurons, this is $D_E^E=0$:
\begin{theorem}
Assume that
 \begin{equation}
h^E(v,N_E,N_I)+v \ge b_E^EN_E-b_I^EN_I, 
\label{hip1}
\end{equation}
\begin{equation}
 a_E(N_E,N_I)\ge a_m>0,  \label{hip2}
\end {equation}
$\forall \, v \, \in (-\infty,V_F]$, 
 and $\forall \ N_I,N_E\ge0$.
Assume also that  $D_E^E=0$  and that there exists some $C>0$ such that
\begin{equation}
\label{NI}
\int_0^tN_I(s-D_I^E) \ ds \le C\, t, \quad \quad \forall \,  t\ge 0.
\end{equation}
Then,
a weak solution to the system (\ref{modelo})
cannot be global in time because one
of
 the following reasons:
\begin{itemize}
\item  $b_E^E>0$ is large enough, for $\rho_E^0$ fixed.
\item $\rho_E^0$ is `concentrated enough' around $V_F$:
\begin{eqnarray}
\int_{-\infty}^{V_F} e^{\mu v} \rho_E^0(v) \ dv \ge
\frac{e^{\mu V_F}}{b_E^E \mu},
 \qquad \mbox{ for a certain } \mu>0 \label{cicritica}
\end{eqnarray}
and for $b_E^E>0$ fixed.
\end{itemize}
\label{th_blowup}
\end{theorem}
Therefore, thanks to Theorem \ref{th:long2} and Theorem \ref{th_blowup}, 
we may conclude that, even with a unique steady state, if system
\eqref{modelo}  has inmediate spike transmissions between 
excitatory neurons, (that is 
$D_E^E=0$) then solutions can blow-up, whether
initially they are close enough to the threshold potential
or whether the excitatory neurons are highly connected (that is
$b_E^E$ is large enough).
In the following numerical experiments we will show that the 
transmission delay between excitatory neurons prevent the blow-up phenomenon,
but the remaining  transmission delays cannot avoid it.

\section{Numerical experiments}
\label{sec: numerics}
\subsection{Numerical Scheme}

The numerical scheme used to simulate equation \eqref{nota} approximates the advection term by a fifth order 
 finite difference flux-splitting Weighted Essentially Non-Oscillatory
(WENO) scheme. 
The flux-splitting considered is the Lax-Friedrich splitting  \cite{S}

\begin{equation*}
f^\pm(\rho)=\frac{1}{2}(f(\rho) \pm \alpha \rho) \quad \textrm{where} \quad \alpha=\max_\rho|f'(\rho)|.
\end{equation*}
 In our case $f(\rho)=h(v,N)\rho$, and thus
 $ \alpha=\max_{v \in (-\infty, V_F)} |h(v,N)|$.
The diffusion term is estimated by standard second order finite 
differences and the time evolution is calculated by an explicit third order 
Total Variation Diminishing (TVD)  Runge-Kutta method.

 Due to the delay, during the time evolution of the solution we have to recover the value of $N$ at time $t-D$, for every time $t$.
 To implement this, we fix a time step $\overline{dt}$ 
and define an array of $M=\frac{D}{\overline{ dt}}$ positions.
Therefore, this array will save only M 
values of $N(t)$ for a time interval  $[kD,(k+1)D)$, $k=0,1,2, ...$ 
In the time interval  $[(k+1)D,(k+2)D)$ these values of the array 
will be used to obtain the delayed values $N(t-D)$ by linear interpolation
between the corresponding positions of the array. 
We assume that $N(t)=0$ $\forall t<0$, so initially 
all the values of  the array are zero, and the recovered values 
for the first time interval ($k=0$) are all zero. 
Notice that we use linear interpolation since the time step $dt$ 
 for the time evolution  is taken according to the 
Courant-Friedrich-Levy (CFL) condition.
Furthermore, once a position of the array is no longer necessary for the 
interpolation, it is overwritten.

The refractory state used in \cite{brunel} is based on considering a delayed firing rate, $N(t-\tau)$, on the right hand side of the PDE for $\rho$. This value is recovered in the same manner as the delayed $N$ that appears due to the transmission delay. The refractory period $\tau$ and the delay $D$ do not usually coincide, and thus the firing rates have to be saved in two different arrays. The refractory state for which $M(t)=\frac{R(t)}{\tau}$ was implemented using a finite difference approximation of its ODE.

\

The numerical approximation of the solution for the two-populations model  was  implemented using the same numerical scheme as that described above for one population. The main difference here is that the code runs over two cores using parallel computational techniques, following the ideas in \cite{CS17}. Each core handles the equations of one of the populations. At the end of every time step the cores communicate via Message Passing Interface  (MPI) to exchange the values of the firing rates. Also the transmission delays were handled as for one population, taking  into account that now each processor has to save two arrays of firing rates, one for each population, since there are four different delays.  The approximation of the different refractory states was done as for 
one population.

\subsection{Numerical results}

 For the following simulations we will consider a uniform mesh for $v \in [-V_{left}, V_F]$, where 
$-V_{left}$ is chosen so that $\rho_\alpha(-V_{left},t) \sim 0$. Moreover, unless otherwise specified,
 $V_F=2$, $V_R=1$, $ \nu_{E,ext}=0$ and 
$a_\alpha(N_E,N_I)= 1$.
We will consider two different types of initial condition:
\begin{equation}
\rho^0_\alpha(v)=
\frac{k}{\sqrt{2 \pi}}e^{-\frac{(v-v_0^\alpha)^2}{2 {\sigma_0^\alpha}^2}},
\label{ci_maxwel}
\end{equation}
where $k$ is a constant such that
$\displaystyle\int_{-V_{left}}^{V_F}\rho^0_\alpha(v) \ dv \approx 1$ numerically,
and
\begin{equation}
\rho_\alpha^0(v) =
\frac{N_\alpha}{a_\alpha(N_E,N_I)}
e^{-\frac{(v-V_0^\alpha(N_E,N_I))^2}{2a_\alpha(N_E,N_I)}}
\int_{\max(v,V_R)}^{V_F}
e^{\frac{(w-V_0^\alpha(N_E,N_I))^2}{2a_\alpha(N_E,N_I)}} \ dw,
\quad \alpha=E,I, \label{soleq}
\end{equation}
with $V_0^\alpha (N_E,N_I)=b_E^\alpha N_E - b_I^\alpha N_I +
(b_E^\alpha-b_E^E)\nu_{E,ext}$  and where $N_\alpha$ is an approximated value
 of the stationary firing rate.
The second kind of initial data is an approximation of the steady states
of the system and allows us to study their local stability.

Notice that we will also refer to \eqref{ci_maxwel}  as the initial condition for the one-population model by just considering $\rho_\alpha=\rho$, $v_0^\alpha=v_0$ and $\sigma_0^{\alpha 2}=\sigma_0^2$.


\subsubsection{Analysis of the number of steady states}

As a first step in our numerical analysis we illustrate numerically some
of  the 
results of Theorem \ref{th: steady states}.
Fig. \ref{bif_EI} shows the behaviour of 
$\mathcal{F}(N_E):= N_E[I_1(N_E,N_I(N_E)+\tau_E]$ for different parameter 
values, which produces bifurcation diagrams.
In the  figure on the left we observe the influence of 
the excitatory refractory period $\tau_E$, considering
fixed the rest of parameters; a large $\tau_E$ gives rise to
the uniqueness of the steady state. In figure on the right one,
the impact of the connectivity parameter $b_E^E$ is described.
In this case, a small $b_E^E$  guarantees a unique stationary solution.
Moreover,  as noted in Remark \ref{remark}, we observe the uniqueness of the
steady state
if the system is highly
connected between excitatory and inhibitory neurons, or if the excitatory
neurons have enough refractory period.

As happens in the case of only one population \cite{CP}, for two populations
(excitatory and inhibitory), neurons in a refractory state guarantee the
existence of stationary states.
(However, the refractory state itself does not prevent the blow-up phenomenon, as
we will show later).

\subsubsection{Blow-up}

In \cite{CP}, the blow-up phenomenon for one population
of neurons with refractory states was shown. 
Theo\-rem \ref{th_blowup} extends this result to two populations of neurons,
 one excitatory and  the other one inhibitory.
The refractory period is not enough to deter the
blow-up of the network; if the membrane potentials of the excitatory population 
are close to the threshold potential, or if the connectivity 
parameter $b_E^E$ is large enough, 
then the network blows-up in finite time.
To achieve the global-in-time existence, it seems necessary some 
transmission delay between excitatory neurons, 
as we observe in our simulations and as it was proved
at the microscopic level for one population \cite{delarue2015particle}.

We start the analysis of the blow-up phenomenon 
 by considering only 
one  average-excitatory population  (we recall that there is global existence
for one average-inhibitory population, see \cite{CGGS}).
In \cite{CCP,CP} it was proved that some solutions blow-up.
In Fig. \ref{blowup_1pob}, we show how the transmission
delay of the spikes between neurons prevents the network from
 blowing-up in finite
time. Have the networks refractory states or not, we observe that
 the blow-up phenomenon appears in absence of a transmission delay.

In \cite{CS17}, the excitatory-inhibitory system without
refractory states was studied. In the current paper, we extend this analysis
to the presence of refractory states.
Figs. \ref{blowup_EIR_bEE} and  \ref{blowup_EIR_ci}
illustrate the results of Theorem \ref{th_blowup}; if there is
no transmission delay between excitatory neurons,  the solution blows-up
because  most of the excitatory neurons have a membrane potential close
to the threshold potential, or because excitatory neurons are highly connected, 
that is, $b_E^E$ is large enough.
We observe in Fig. \ref{blowup_EIR_ci_delay} that the remaining delays 
do not avoid the blow-up phenomenon, since in this figure 
all the delays are 0.1, except $D_E^E=0$.
The importance of  $D_E^E$ is discerned in Fig. \ref{noblowup_EIRD_ci}.
We show the evolution in time
of the solution of \eqref{modelo}, with the same initial data
as considered in Fig. \ref{blowup_EIR_ci}
and with  $D_E^E=0.1$; in this case, 
the solution exists for every time, thus avoiding the blow-up.

\subsubsection{Steady states and periodic solutions}

In Fig. \ref{bif_EI} we examined several choices of the model
parameters, for which the system \eqref{modelo} presents
three steady states. For one of these cases,
the  analysis of their stability
is numerically investigated in Fig. \ref{est_eq_EI1}.
For $\alpha=E,I$, the initial conditions $\rho_\alpha^0-1,2,3$ 
are given by the profiles \eqref{soleq}, 
where $N_\alpha$ are approximations of the stationary firing rates.
The evolution in time of the probability densities, the firing
rates and the refractory states show that the lower steady state
seems to be stable, while the two others are unstable. Moreover,
considering as initial data \eqref{soleq}
with $N_\alpha$ approximations of the higher stationary 
firing rates the solution blows-up
in finite time, while with the intermediate firing rate the solution
tends to the lower steady state.
Fig. \ref{est_eq_EI2} also describes the stability when there are
three steady states. In this case the intermediate state
is very close to the highest one. Here,
the lower steady state  also appears to be stable.
The two others are unstable, but the higher one does not blow-up
in finite time. 

The transmission delay not only prevents the blow-up phenomenon, but
also should produce periodic solutions.
In Fig. \ref{osci_R_MJ-1}, we analyze the influence of the transmission delay
for one  average-excitatory population; if the initial datum is concentrated
around $V_F$,   periodic solutions appear; on the contrary,
 if it is far from $V_F$, 
the solution reaches a steady state. 
In Figs. \ref{osci_R_MJ-2} and \ref{osci_R_MJ-3},
for one average-inhibitory population with
transmission delay, we show that 
periodic solutions emerge if the initial condition is
concentrated around the
threshold potential, and even if the initial datum is far from
the threshold and  $v_{ext}$ is large.
A  comparison between $R(t)$ and $N(t)$ for 
$M(t)=\frac{R(t)}{\tau}$ and $M(t)=N(t-\tau)$ is presented 
in Fig. \ref{blowup_ERD_ci_MJ}.
In both cases the steady state is the same and  the solutions tend to it. If the system tends to a synchronous
state, these states are also almost the same for both possible
choices of  $M$.

Synchronous states appear also in the case of two populations
(excitatory and inhibitory), as it is described in Fig. \ref{oscilaciones_EI}.
In this particular case, they seem to appear due to the inhibitory population, 
which tends to a periodic solution. What is more, the excitatory population 
presents a solution that oscillates close around the equilibrium.

\begin{figure}[H]
\begin{center}
\begin{minipage}[c]{0.33\linewidth}
\begin{center}
\includegraphics[width=\textwidth]{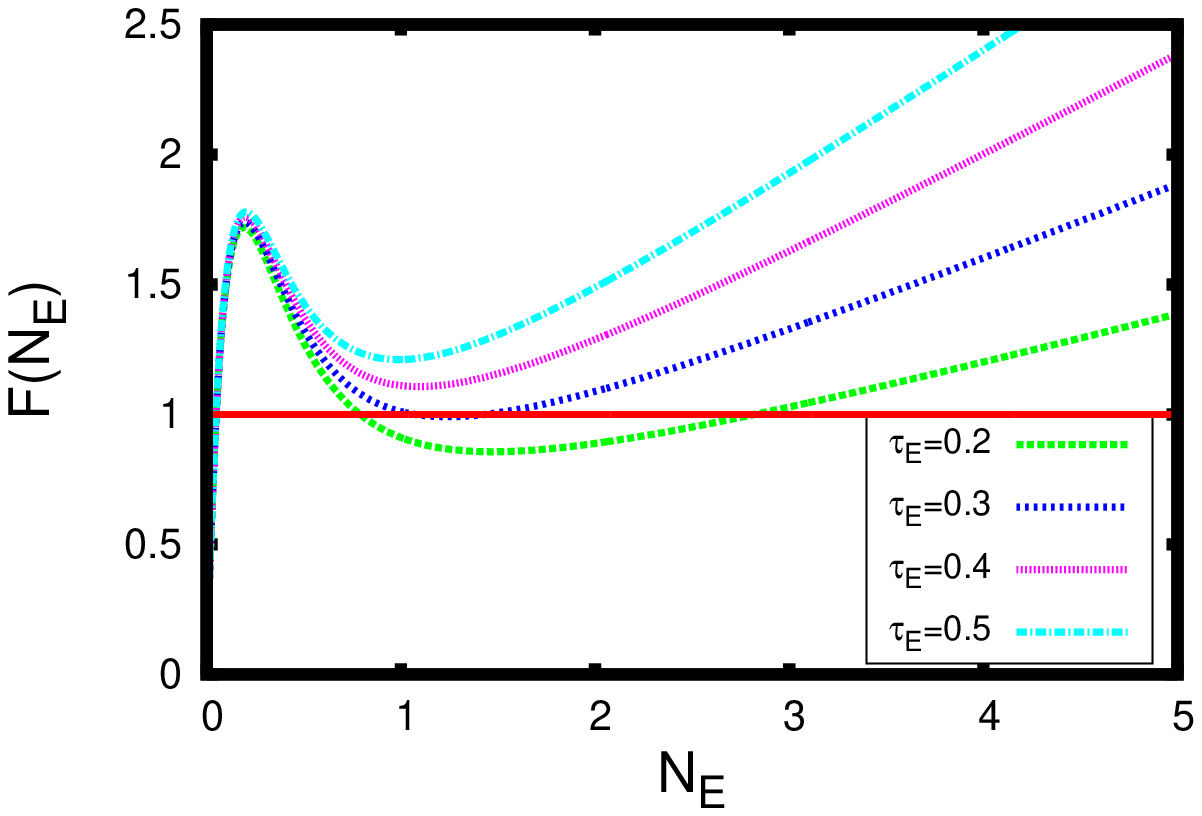}
\end{center}
\end{minipage}
\begin{minipage}[c]{0.33\linewidth}
\begin{center}
\includegraphics[width=\textwidth]{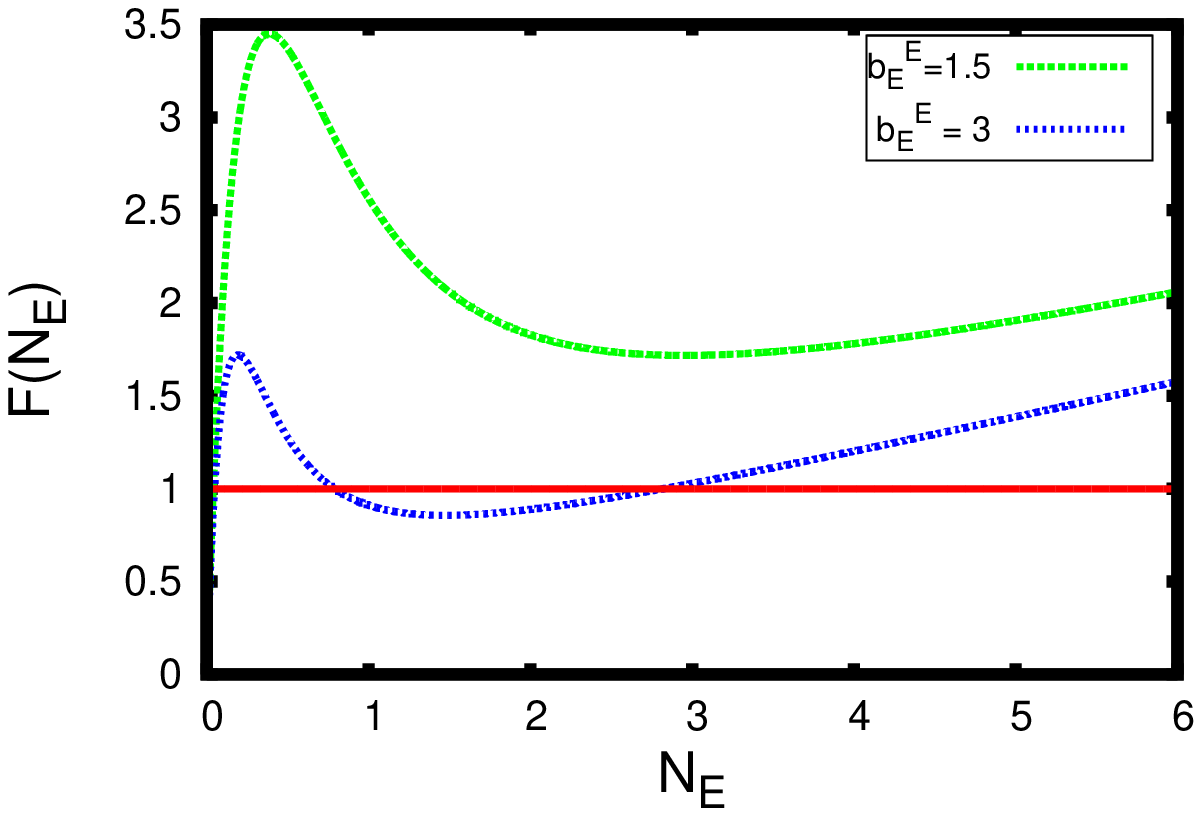}
\end{center}
\end{minipage}
\end{center}

\caption{
{\bf Number of steady states for system \eqref{modelo} described by Theorem
\ref{th: steady states}.-} 
Left: For fixed $b_I^E=7$, $b_I^I=2$, $b_E^I=0.01$, $b_E^E=3$ and
$\tau_I=0.2$, we observe the influence of the excitatory
 refractory period $\tau_E$.
Right: For fixed $b_I^E=7$, $b_I^I=2$, $b_E^I=0.01$ and $\tau_E=\tau_I=0.2$,
we observe the influence of the connectivity parameter $b_E^E$.}
\label{bif_EI}
\end{figure}
\begin{figure}[H]
\begin{center}
\begin{minipage}[c]{0.33\linewidth}
\begin{center}
\includegraphics[width=\textwidth]{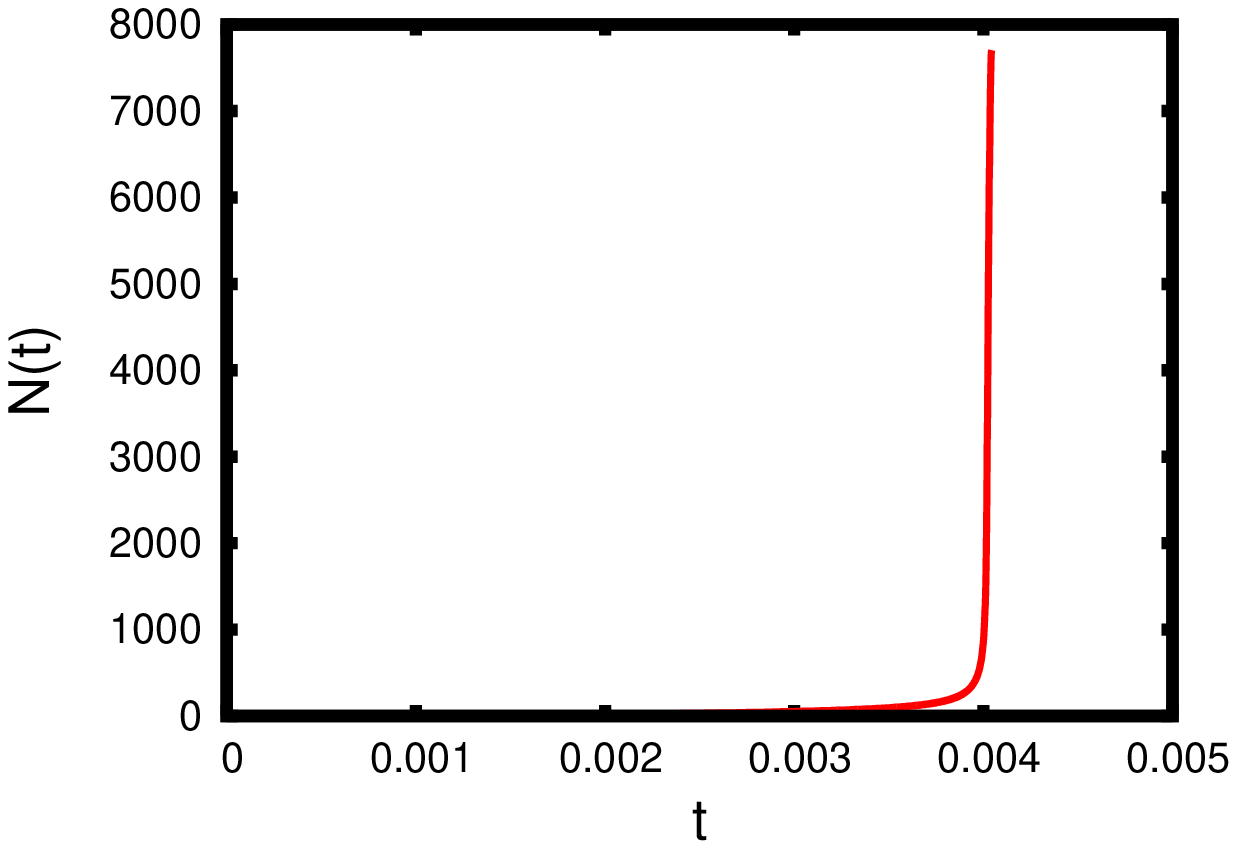}
\end{center}
\end{minipage}
\begin{minipage}[c]{0.33\linewidth}
\begin{center}
\includegraphics[width=\textwidth]{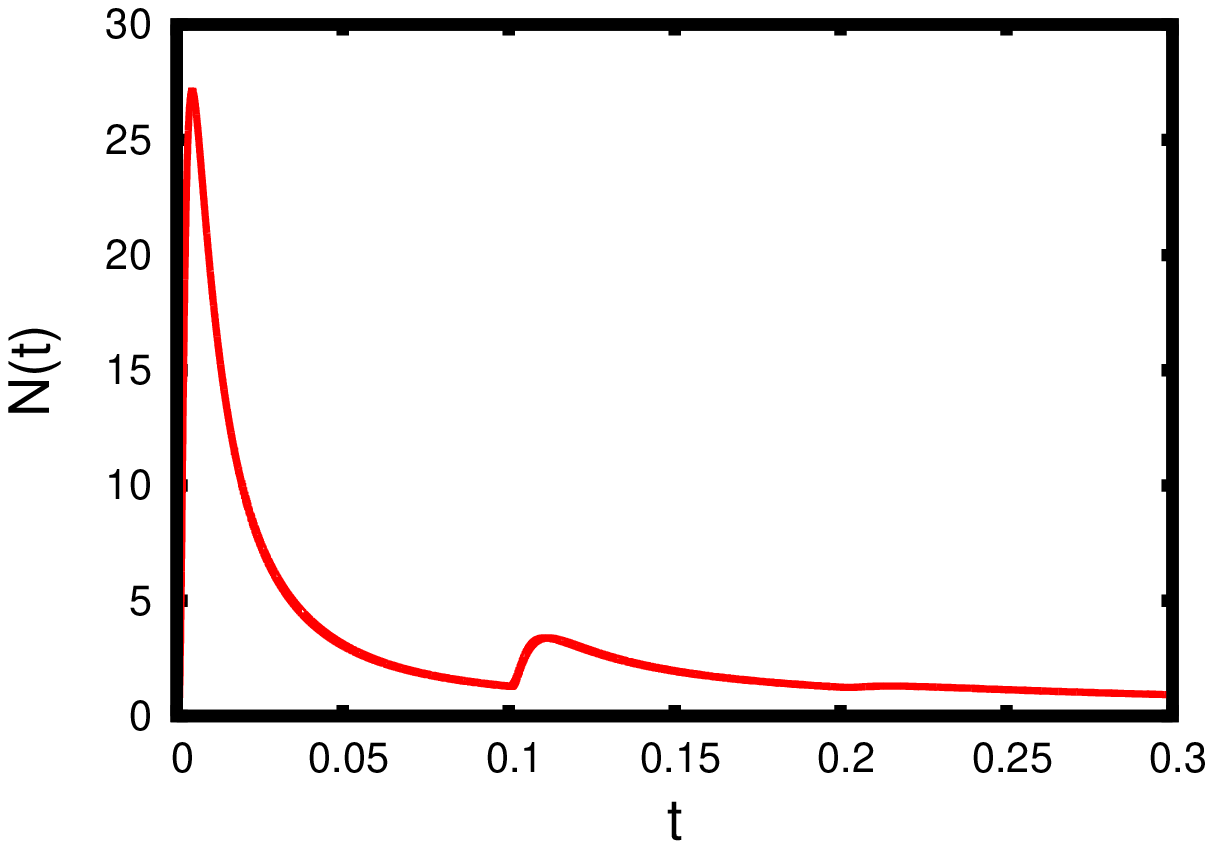}
\end{center}
\end{minipage}
\end{center}
\begin{center}
\begin{minipage}[c]{0.33\linewidth}
\begin{center}
\includegraphics[width=\textwidth]{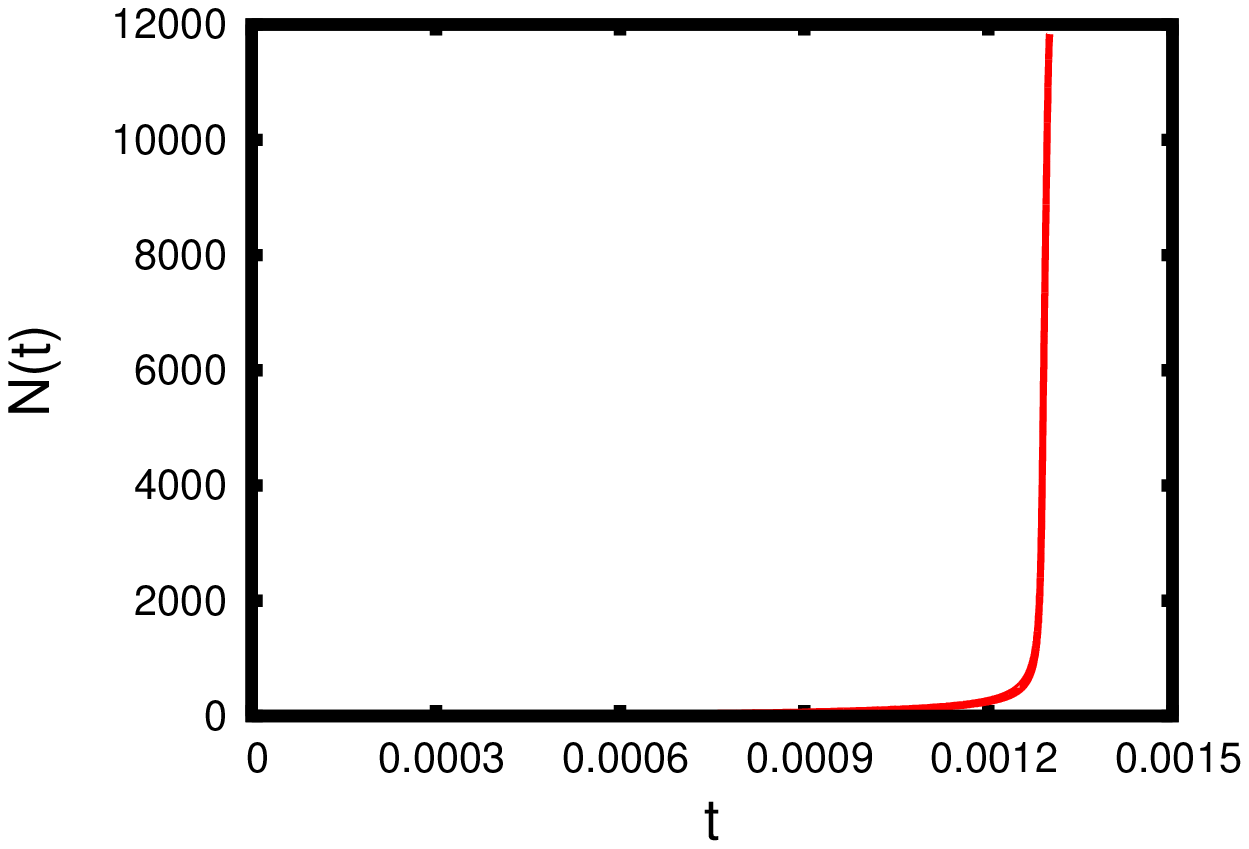}
\end{center}
\end{minipage}
\begin{minipage}[c]{0.33\linewidth}
\begin{center}
\includegraphics[width=\textwidth]{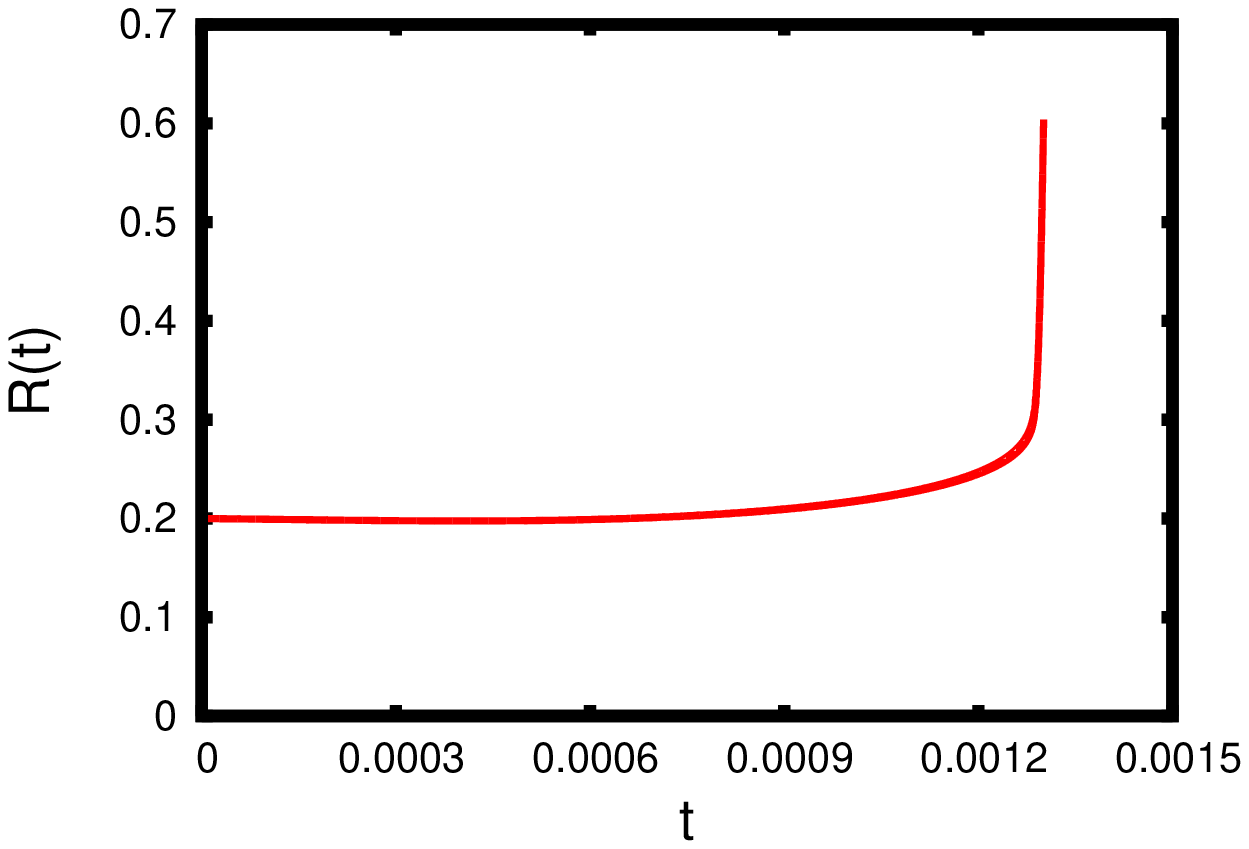}
\end{center}
\end{minipage}
\end{center}
\begin{center}
\begin{minipage}[c]{0.33\linewidth}
\begin{center}
\includegraphics[width=\textwidth]{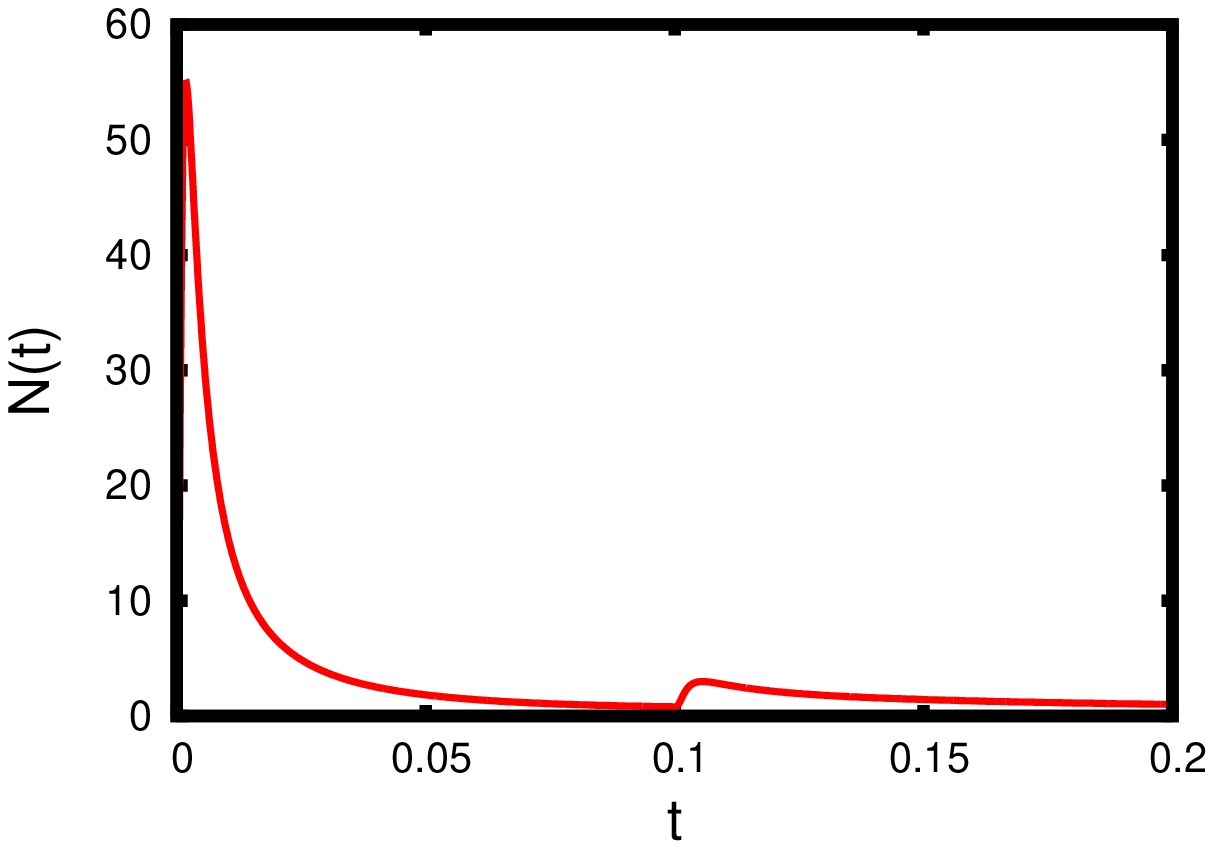}
\end{center}
\end{minipage}
\begin{minipage}[c]{0.33\linewidth}
\begin{center}
\includegraphics[width=\textwidth]{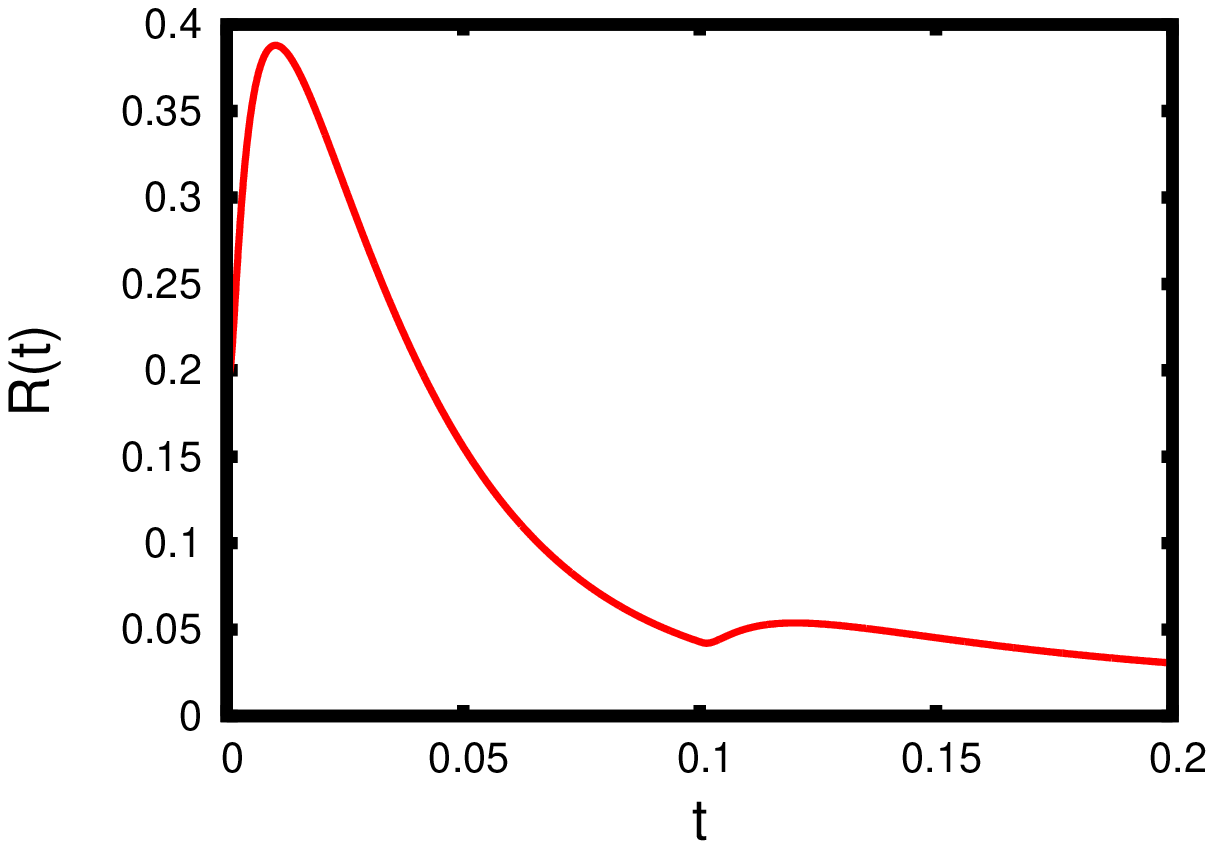}
\end{center}
\end{minipage}
\end{center}
\caption{{\bf System \eqref{nota} (only one population) presents blow-up,
 if there is no transmission delay.-}
We consider the initial data \eqref{ci_maxwel} with $v_0=1.83$, 
and $\sigma_0=0.0003$,  and the connectivity parameter $b=0.5$. 
Top:  Without refractory state;
Left: $N$ blows-up in finite time, if  there is no delay, $D=0$.
Right: $N$ does not blow-up if there is delay,  $D=0.1$.
\newline
Middle: 
With refractory state ($M(t)=\frac{R(t)}{\tau}$), 
$R(0)=0.2$, $\tau=0.025$ and $D=0$,
since there is no transmission delay  $N$ and $R$ blow-up in finite
time.
\newline
Bottom: 
With refractory state ($M(t)=\frac{R(t)}{\tau}$), 
$R(0)=0.2$, $\tau=0.025$ and $D=0.07$, 
the solution tends to the steady state, due to the transmission delay. 
}
\label{blowup_1pob}
\end{figure}
\begin{figure}[H]
\begin{minipage}[c]{0.33\linewidth}
\begin{center}
\includegraphics[width=\textwidth]{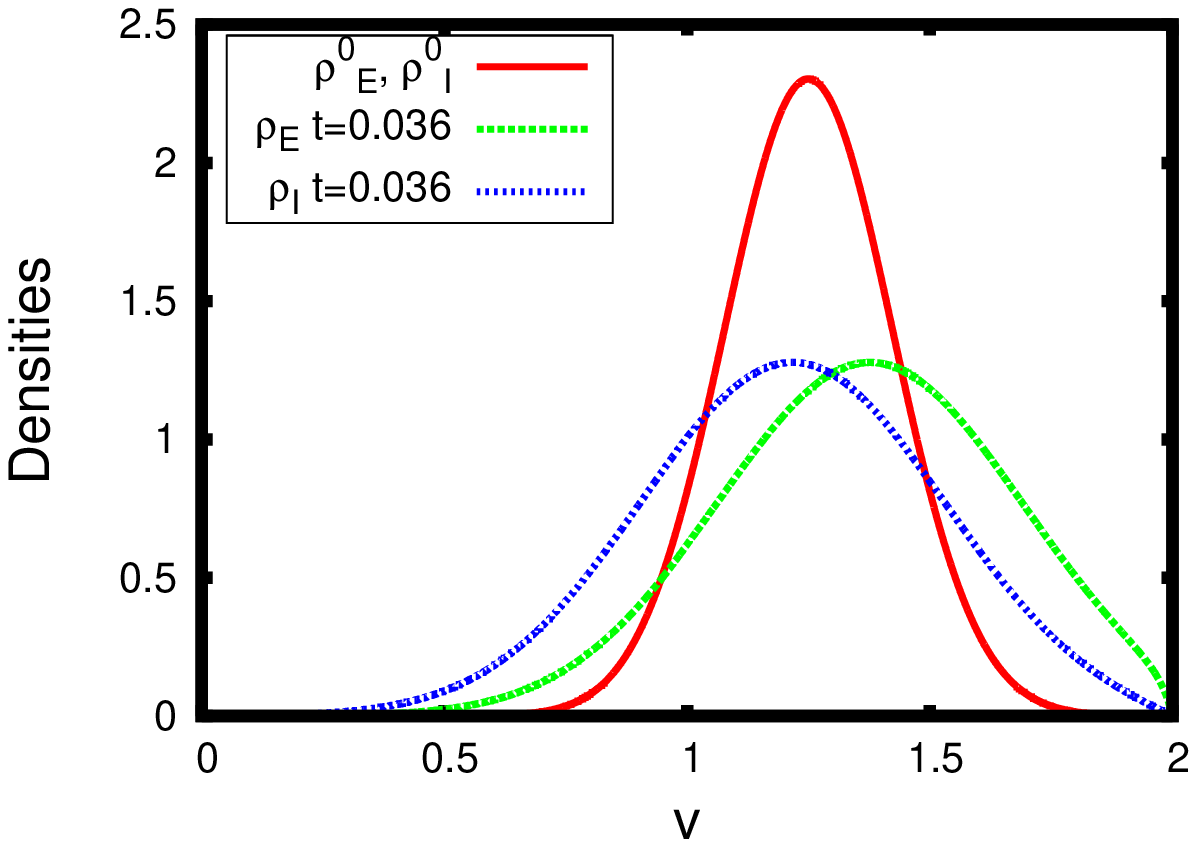}
\end{center}
\end{minipage}
\begin{minipage}[c]{0.33\linewidth}
\begin{center}
\includegraphics[width=\textwidth]{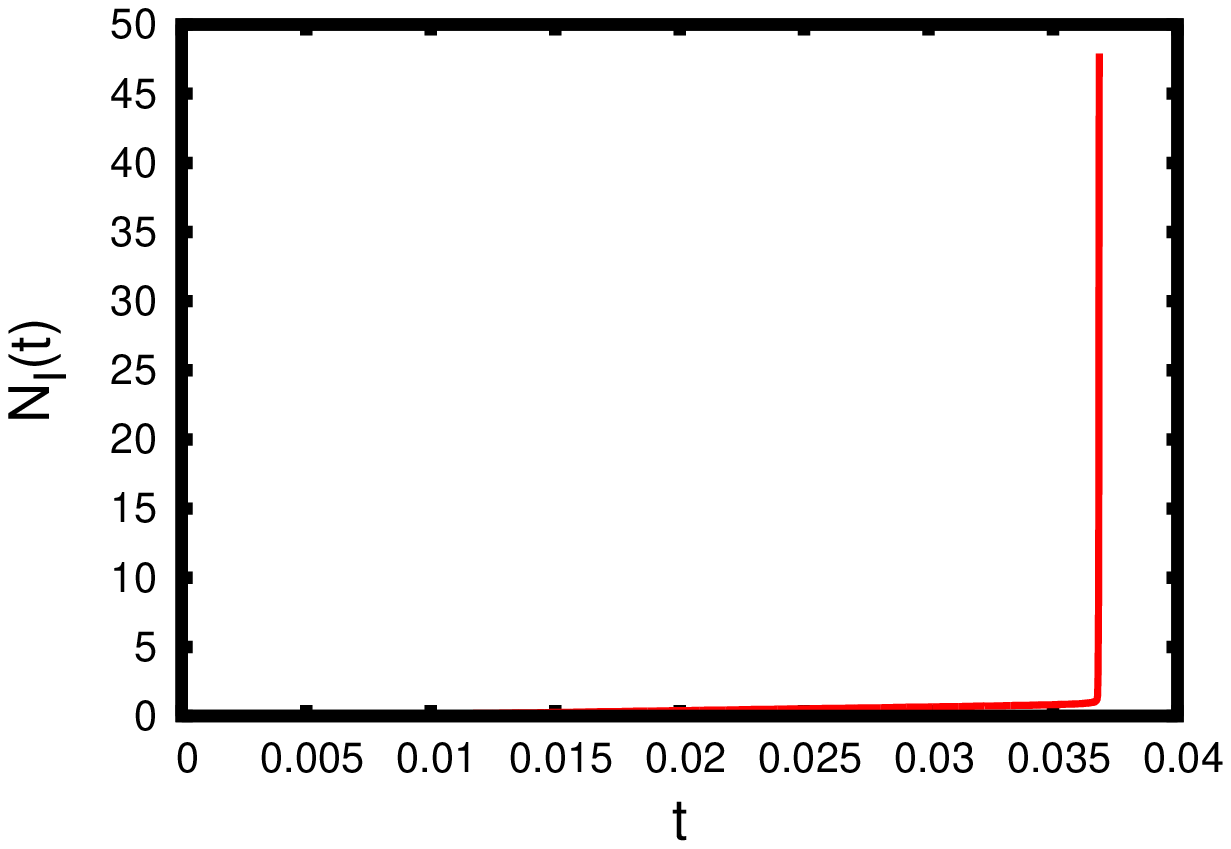}
\end{center}
\end{minipage}
\begin{minipage}[c]{0.33\linewidth}
\begin{center}
\includegraphics[width=\textwidth]{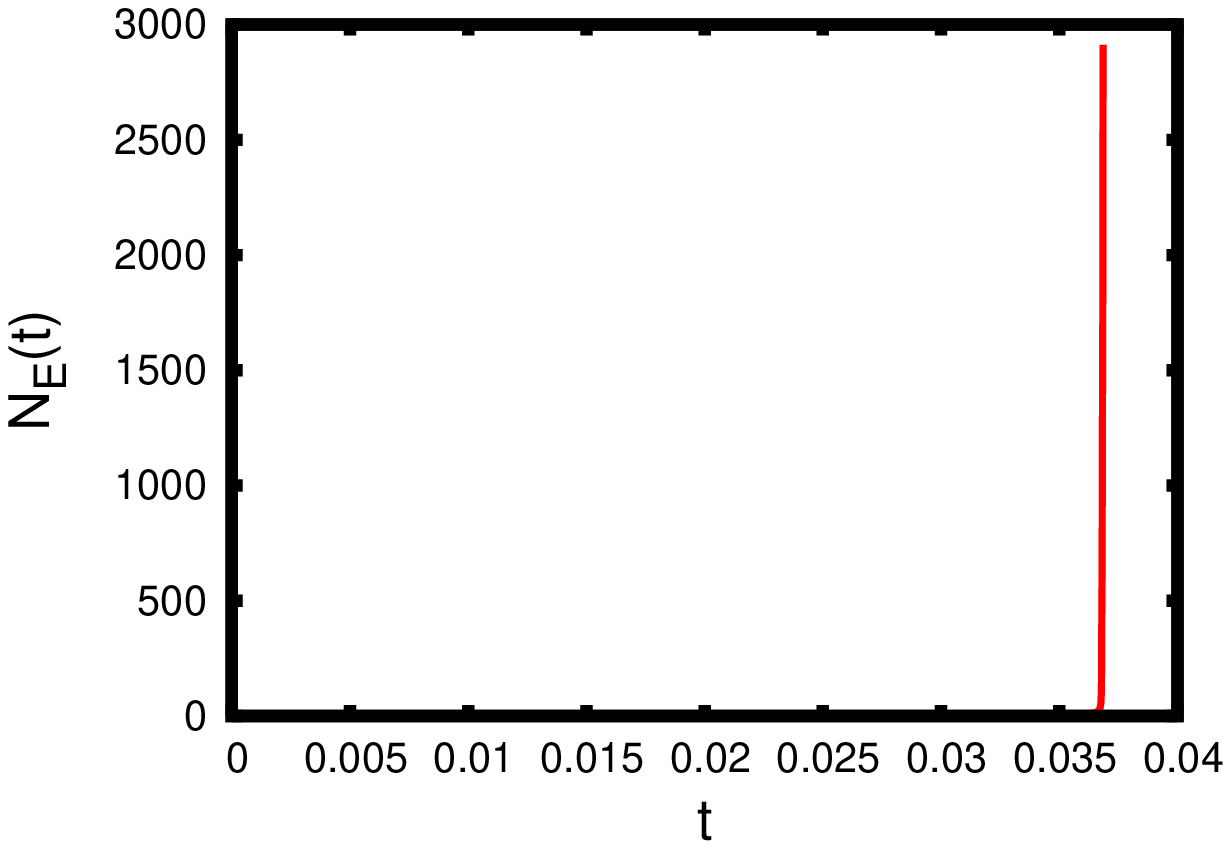}
\end{center}
\end{minipage}

\begin{center}
\begin{minipage}[c]{0.33\linewidth}
\begin{center}
\includegraphics[width=\textwidth]{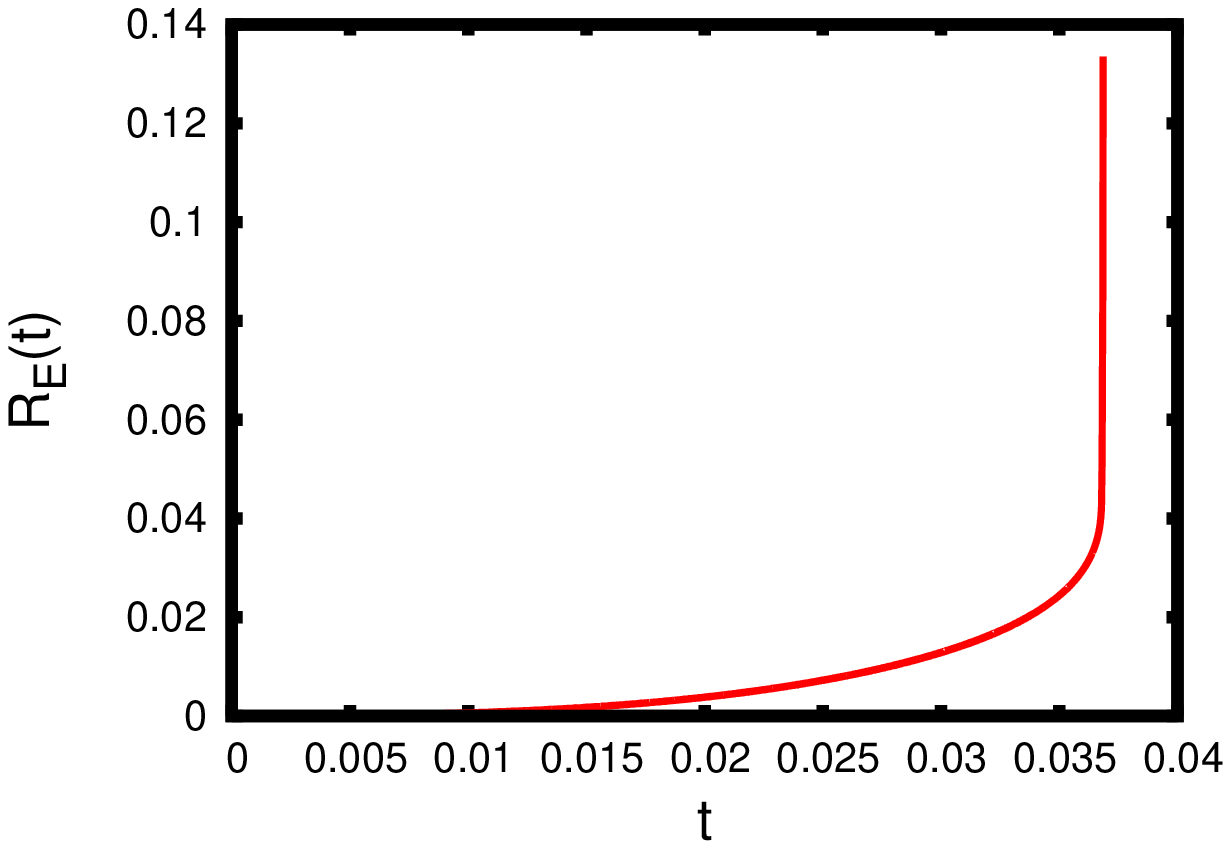}
\end{center}
\end{minipage}
\begin{minipage}[c]{0.33\linewidth}
\begin{center}
\includegraphics[width=\textwidth]{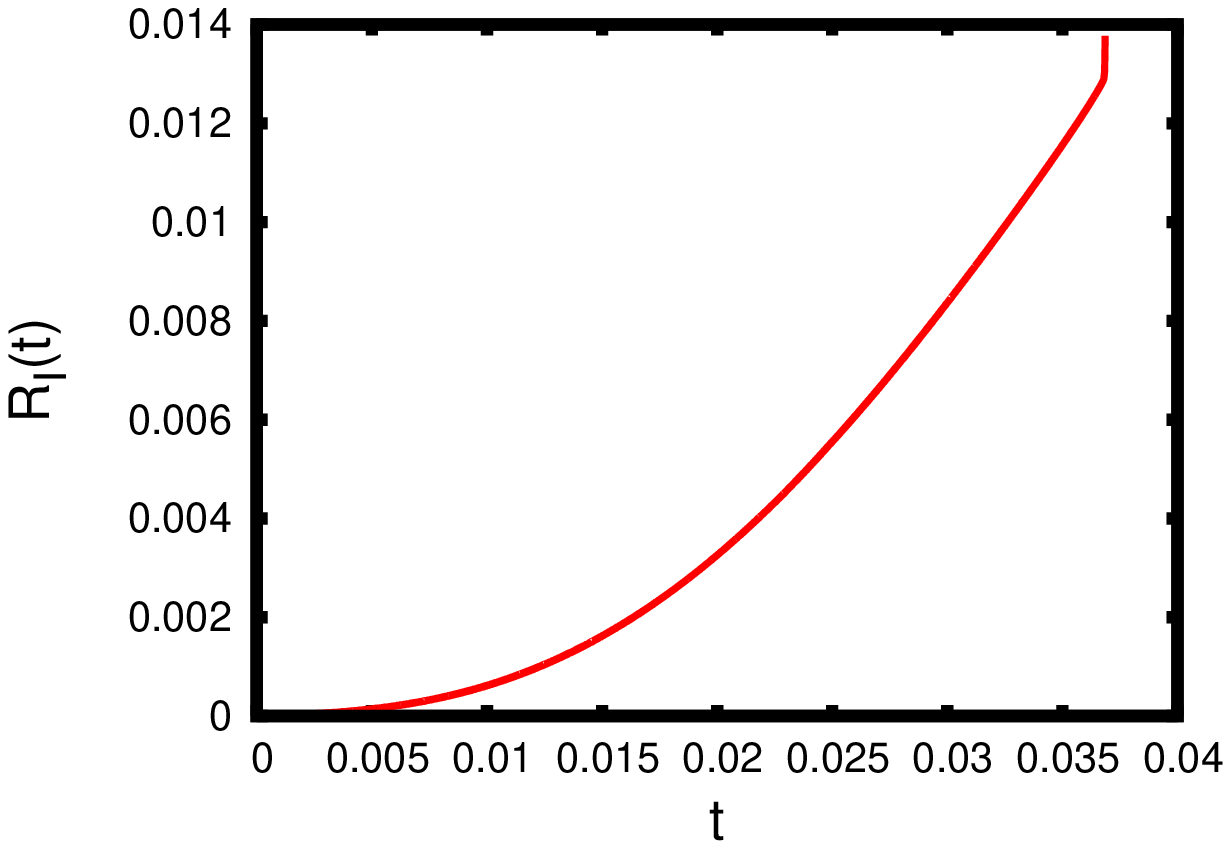}
\end{center}
\end{minipage}
\end{center}
\caption{
{\bf System \eqref{modelo} (two populations: excitatory and
inhibitory) presents blow-up, if there are no transmission delays.-}
We consider initial data \eqref{ci_maxwel} with
$v_0^E=v_0^I=1.25$ and $\sigma_0^E=\sigma_0^I=0.0003$,
the connectivity parameters 
$b_E^E=6$, $b_I^E=0.75$, $b_I^I=0.25$, $b_E^I=0.5$,
 and
with refractory states ($M_\alpha(t)=N_\alpha(t-\tau_\alpha)$)
where $\tau_\alpha=0.025$.
We observe that the initial data are not concentrated around
the threshold potential but the solution blows-up because $b_E^E=6$
is large enough 
and there are no transmission delays
(see Theorem \ref{th_blowup}).}
\label{blowup_EIR_bEE}
\end{figure}
\begin{figure}[H]
\begin{minipage}[c]{0.33\linewidth}
\begin{center}
\includegraphics[width=\textwidth]{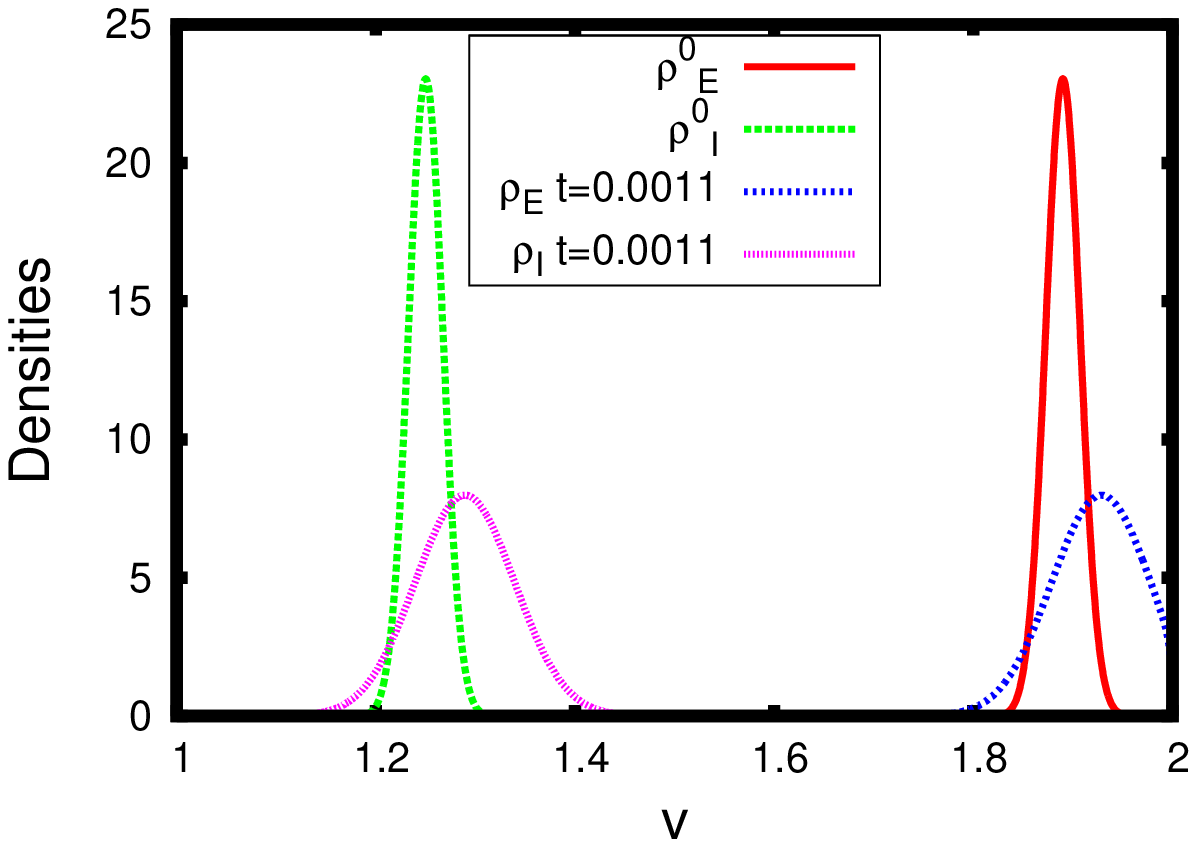}
\end{center}
\end{minipage}
\begin{minipage}[c]{0.33\linewidth}
\begin{center}
\includegraphics[width=\textwidth]{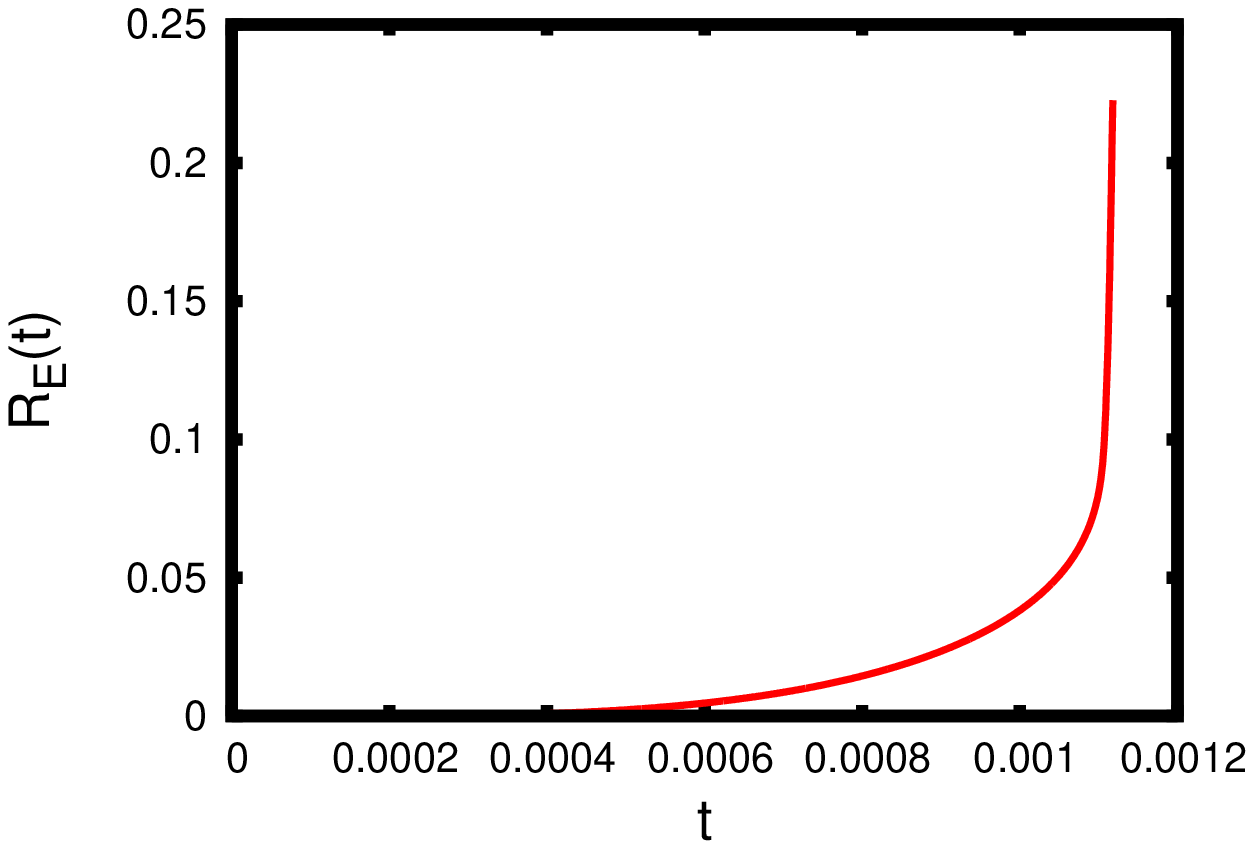}
\end{center}
\end{minipage}
\begin{minipage}[c]{0.33\linewidth}
\begin{center}
\includegraphics[width=\textwidth]{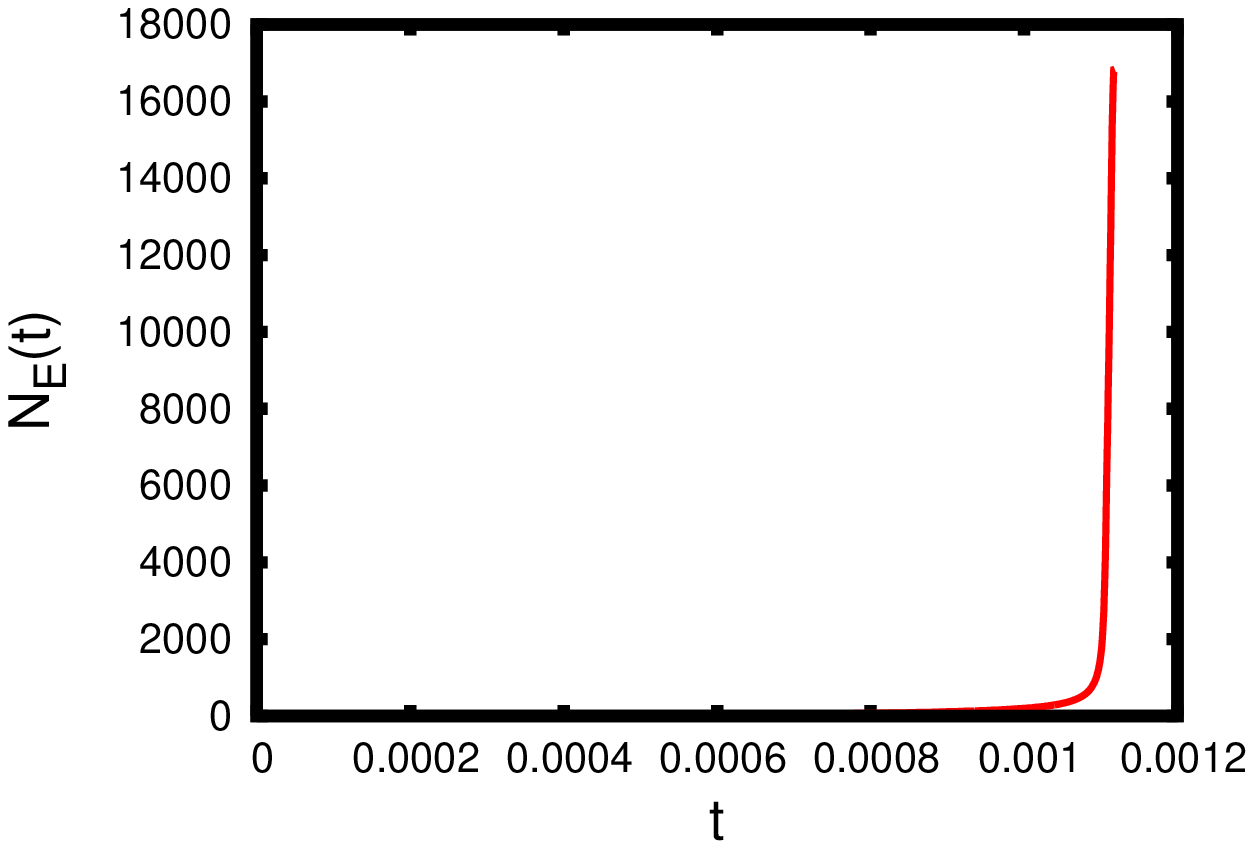}
\end{center}
\end{minipage}
\caption{
{\bf System \eqref{modelo} (two populations: excitatory and
inhibitory) presents blow-up, if there are no transmission delays.-}
We consider initial data \eqref{ci_maxwel} with
$v_0^E=1.89$,
$v_0^I=1.25$ and $\sigma_0^E=\sigma_0^I=0.0003$,
 the connectivity parameters 
 $b_E^E=0.5$, $b_I^E=0.75$, $b_I^I=0.25$, $b_E^I=0.5$,
 and
with refractory states ($M_\alpha(t)=N_\alpha(t-\tau_\alpha)$)
where $\tau=0.025$.
We observe that 
$b_E^E=0.5$ is not large enough,
but the solution blows-up because the
initial condition for the excitatory population
is concentrated around
the threshold potential
and there are no transmission delay
(see Theorem \ref{th_blowup}).
}
\label{blowup_EIR_ci}
\end{figure}
\begin{figure}[H]
\begin{minipage}[c]{0.33\linewidth}
\begin{center}
\includegraphics[width=\textwidth]{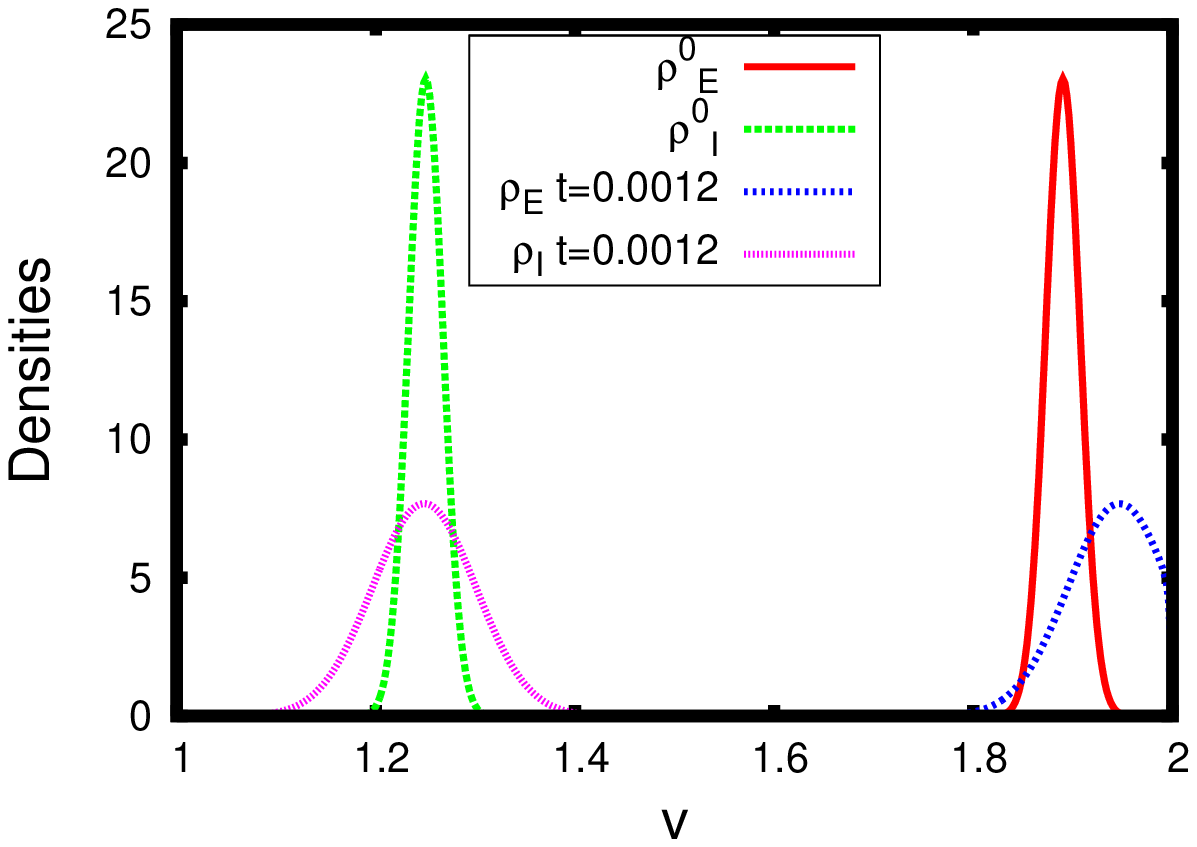}
\end{center}
\end{minipage}
\begin{minipage}[c]{0.33\linewidth}
\begin{center}
\includegraphics[width=\textwidth]{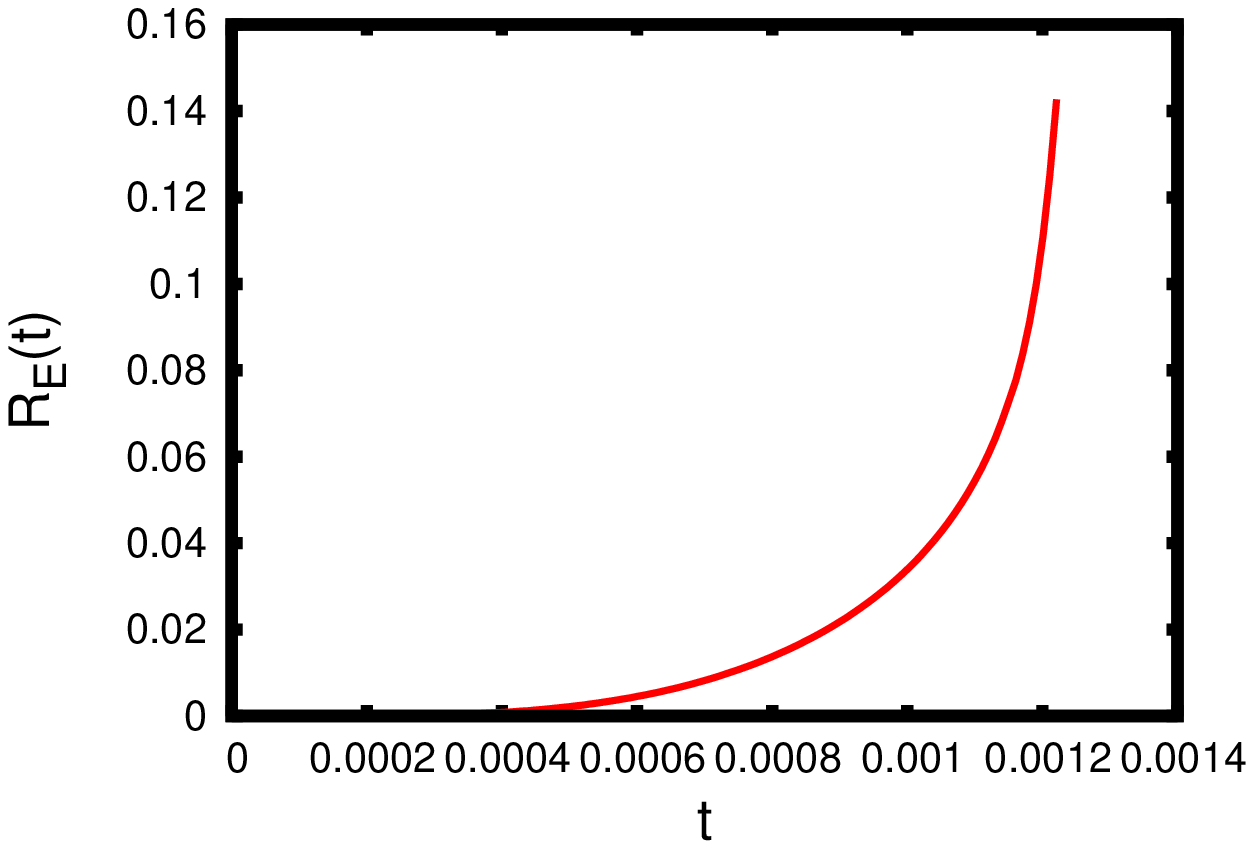}
\end{center}
\end{minipage}
\begin{minipage}[c]{0.33\linewidth}
\begin{center}
\includegraphics[width=\textwidth]{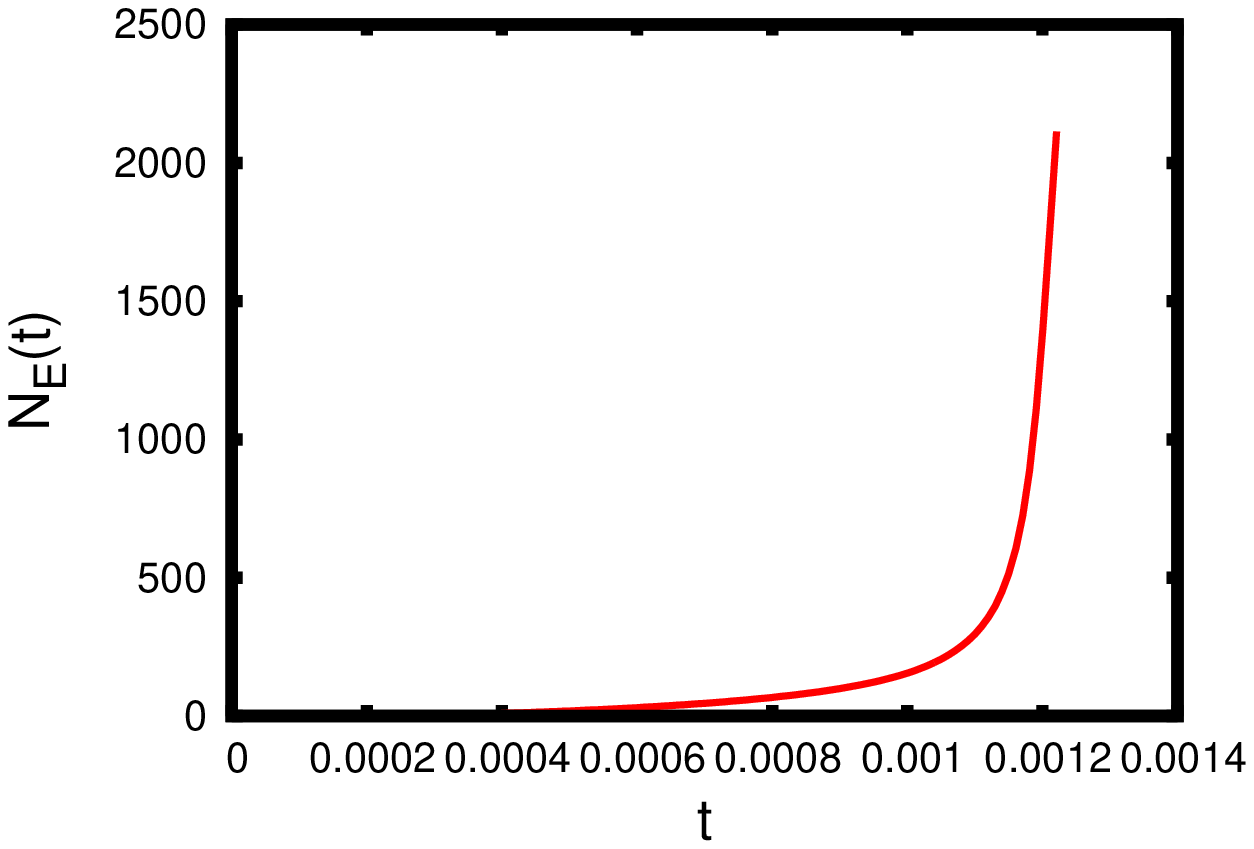}
\end{center}
\end{minipage}
\caption{
{\bf System \eqref{modelo} (two populations: excitatory and
inhibitory) presents blow-up, if there is no excitatory transmission delay.-}
We consider initial data \eqref{ci_maxwel} with
$v_0^E=1.89$,
$v_0^I=1.25$ and $\sigma_0^E=\sigma_0^I=0.0003$,
 the connectivity parameters 
$b_E^E=0.5$, $b_I^E=0.75$, $b_I^I=0.25$, $b_E^I=0.5$,
and
with refractory states ($M_\alpha(t)=N_\alpha(t-\tau_\alpha)$)
where $\tau_\alpha=0.025$. 
All the delays are 0.1, except $D_E^E=0$.
We observe that the other delays do not avoid the blow-up
due to a concentrated 
initial condition for the excitatory population.
}
\label{blowup_EIR_ci_delay}
\end{figure}
\begin{figure}[H]
\begin{center}
\begin{minipage}[c]{0.33\linewidth}
\begin{center}
\includegraphics[width=\textwidth]{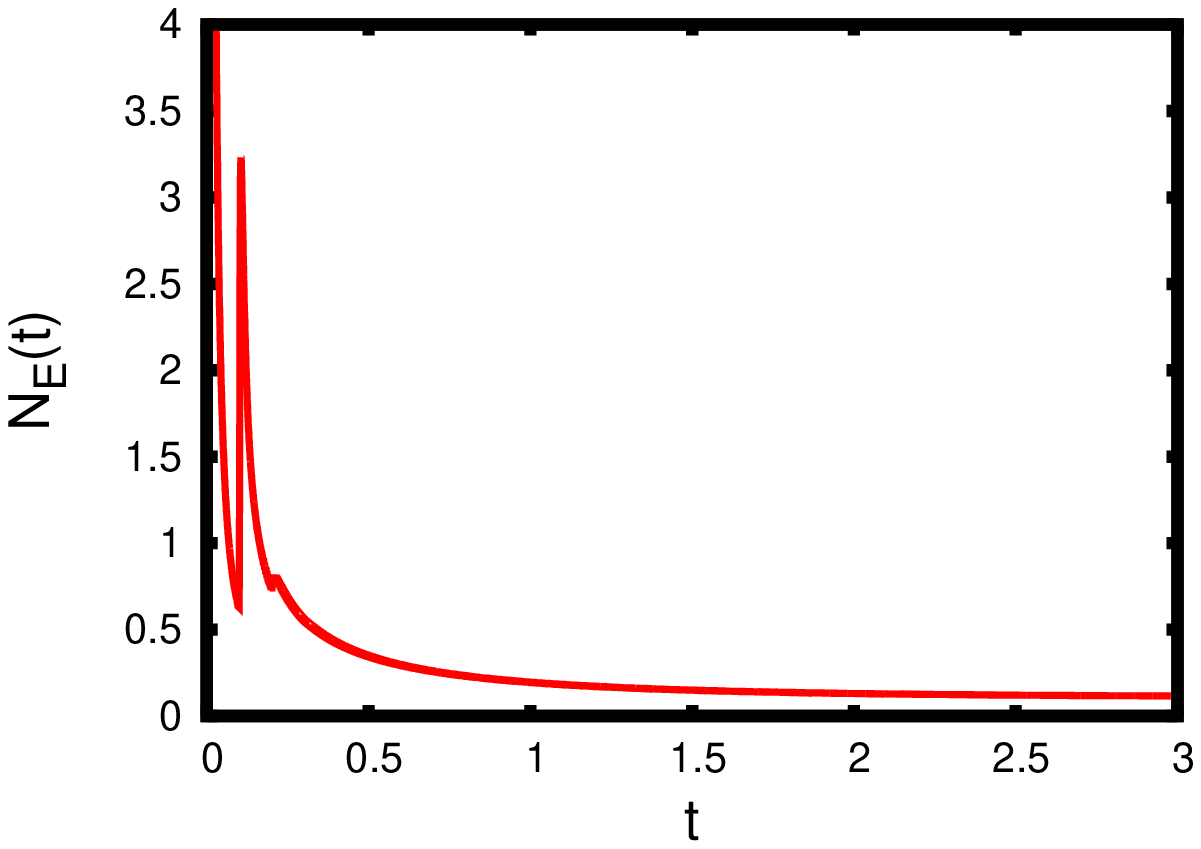}
\end{center}
\end{minipage}
\begin{minipage}[c]{0.33\linewidth}
\begin{center}
\includegraphics[width=\textwidth]{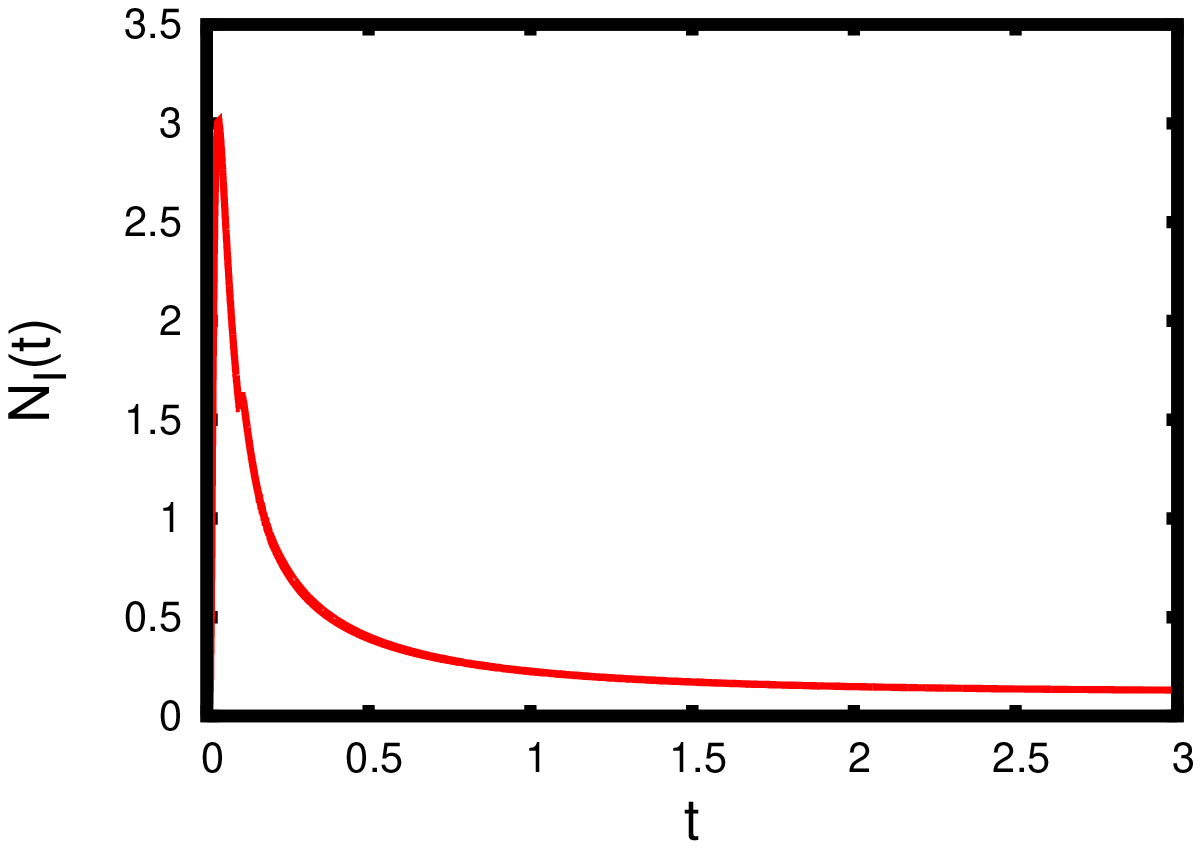}
\end{center}
\end{minipage}
\end{center}
\begin{center}
\begin{minipage}[c]{0.33\linewidth}
\begin{center}
\includegraphics[width=\textwidth]{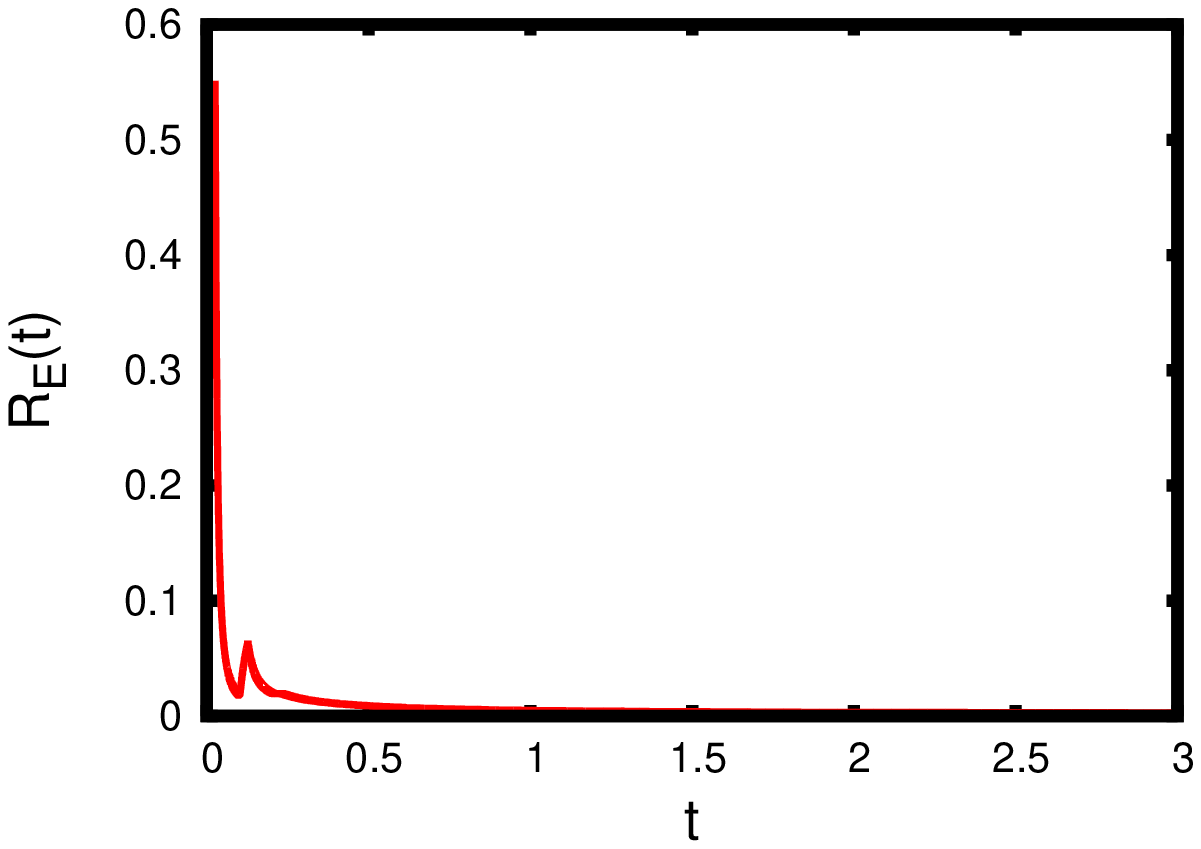}
\end{center}
\end{minipage}
\begin{minipage}[c]{0.33\linewidth}
\begin{center}
\includegraphics[width=\textwidth]{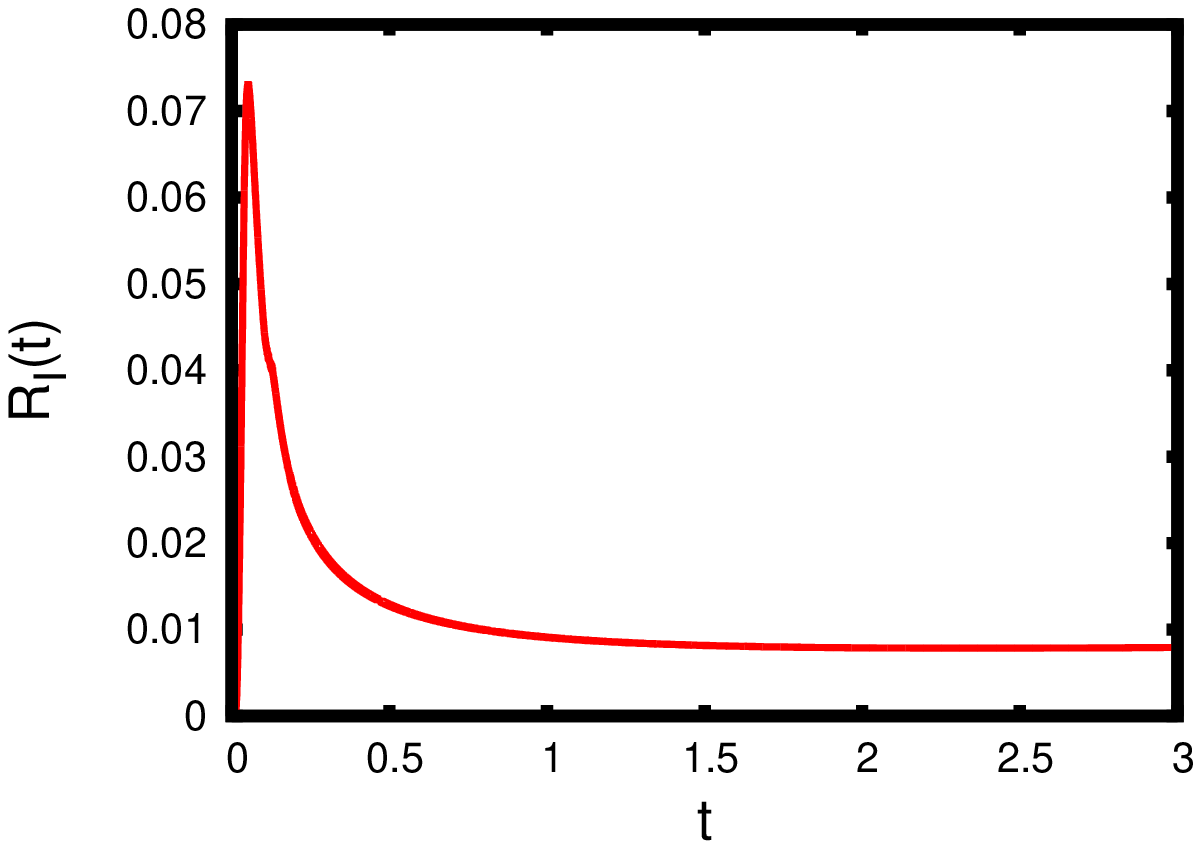}
\end{center}
\end{minipage}
\end{center}
\begin{center}
\begin{minipage}[c]{0.33\linewidth}
\begin{center}
\includegraphics[width=\textwidth]{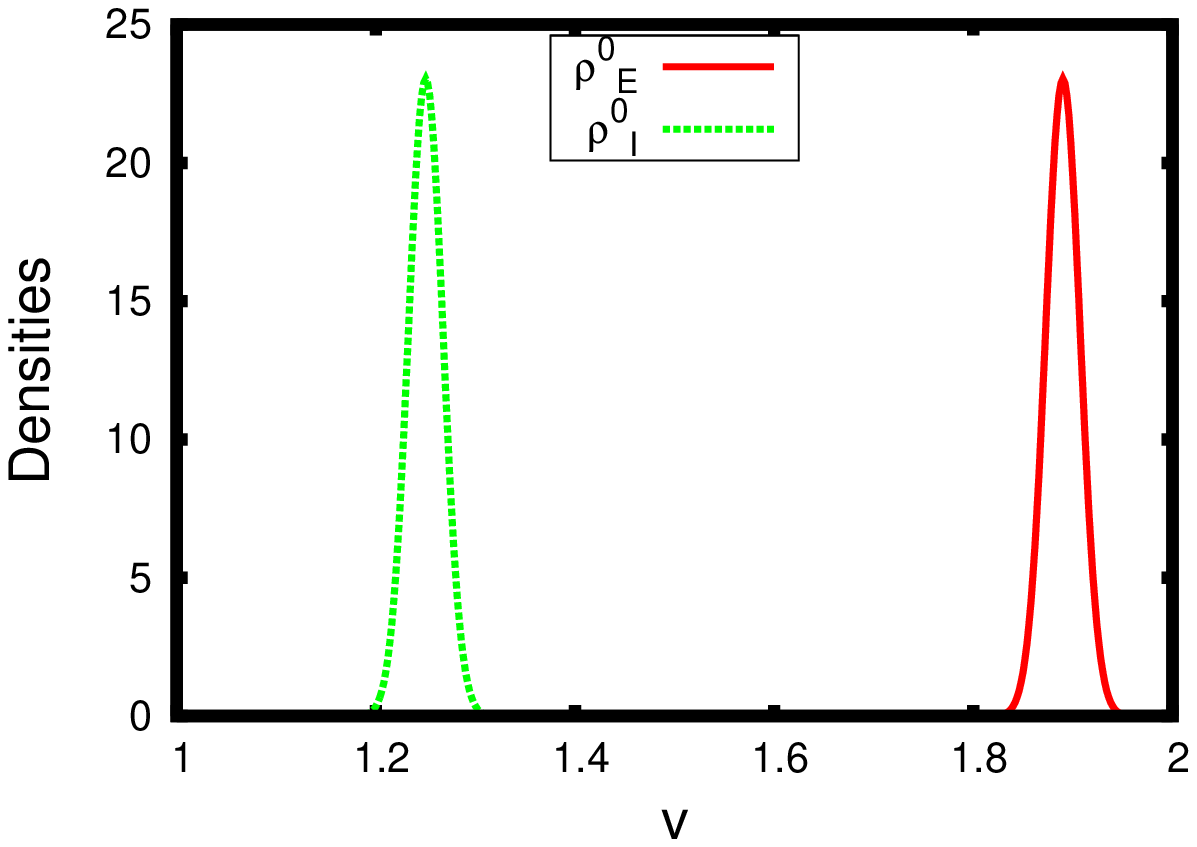}
\end{center}
\end{minipage}
\begin{minipage}[c]{0.33\linewidth}
\begin{center}
\includegraphics[width=\textwidth]{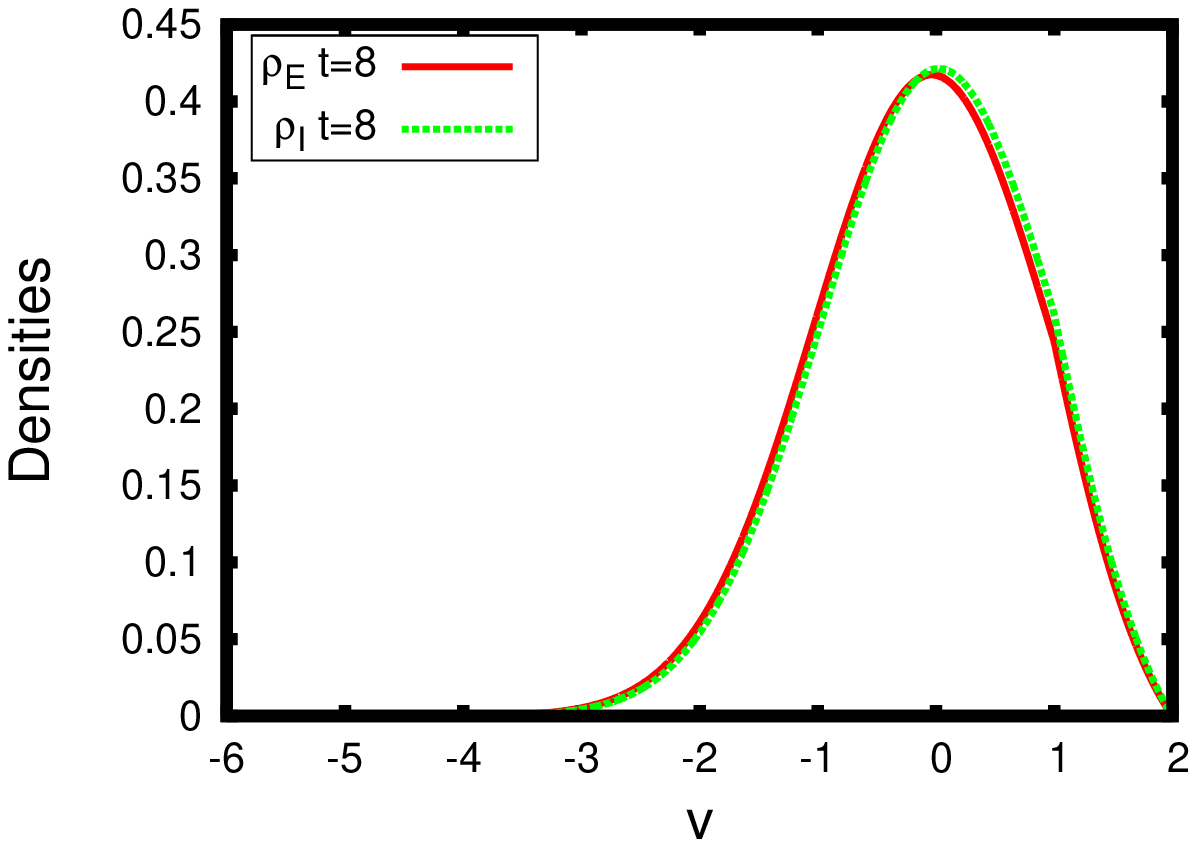}
\end{center}
\end{minipage}
\end{center}
\caption{
{\bf System \eqref{modelo} (two populations: excitatory and
inhibitory) avoids blow-up, if there is a transmission delay between
excitatory neurons.-}
We consider initial data \eqref{ci_maxwel} with
$v_0^E=1.89$,
$v_0^I=1.25$ and $\sigma_0^E=\sigma_0^I=0.0003$,
 the connectivity parameters 
 $b_E^E=0.5$, $b_I^E=0.75$, $b_I^I=0.25$, $b_E^I=0.5$,
$D_E^I=D_I^E=D_I^I=0$,
 and
with refractory states ($M_\alpha(t)=N_\alpha(t-\tau_\alpha)$)
where $\tau=0.025$.
We observe that if there is a transmission delay between excitatory
neurons 
$D_E^E=0.1$, the blow-up phenomenon
is avoided.
Top: Firing rates. Middle: Refractory states. Bottom: Probability
densities.}
\label{noblowup_EIRD_ci}
\end{figure}
\begin{figure}[H]
\begin{center}
\begin{minipage}[c]{0.33\linewidth}
\begin{center}
\includegraphics[width=\textwidth]{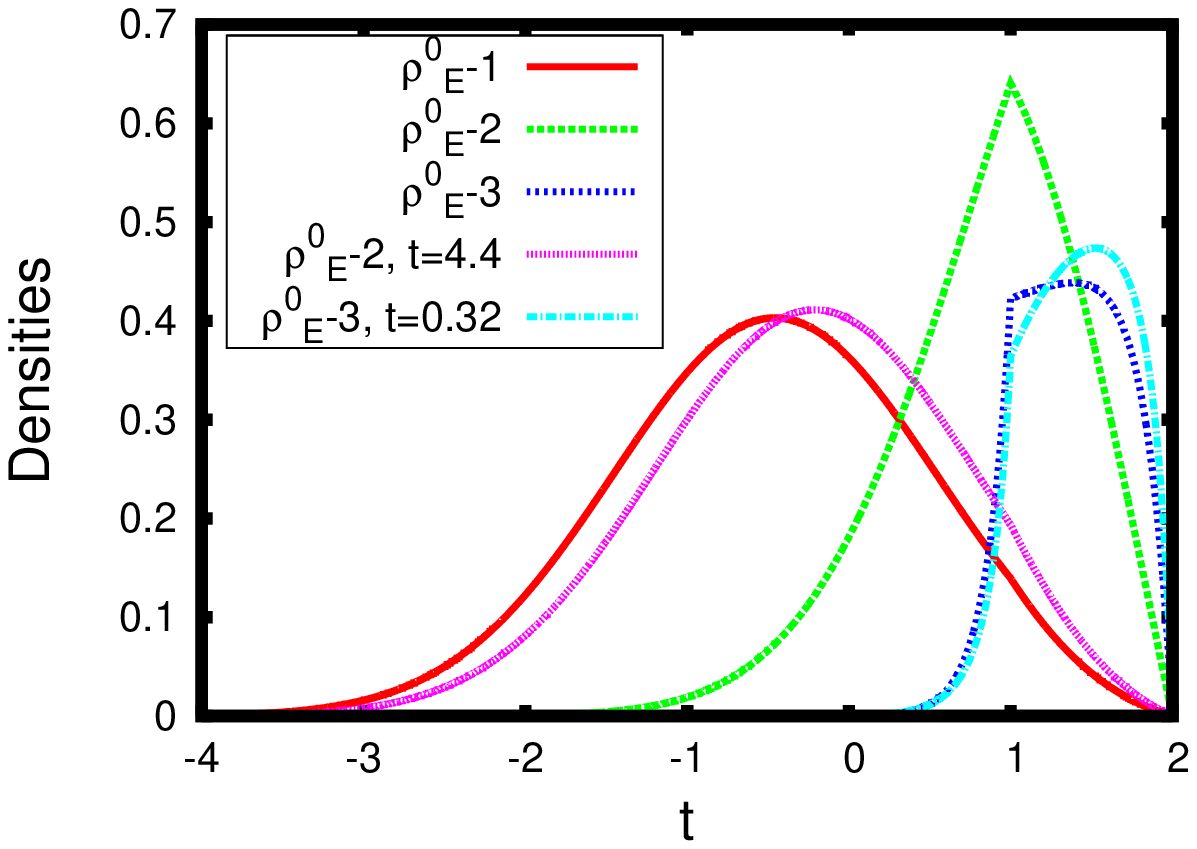}
\end{center}
\end{minipage}
\begin{minipage}[c]{0.33\linewidth}
\begin{center}
\includegraphics[width=\textwidth]{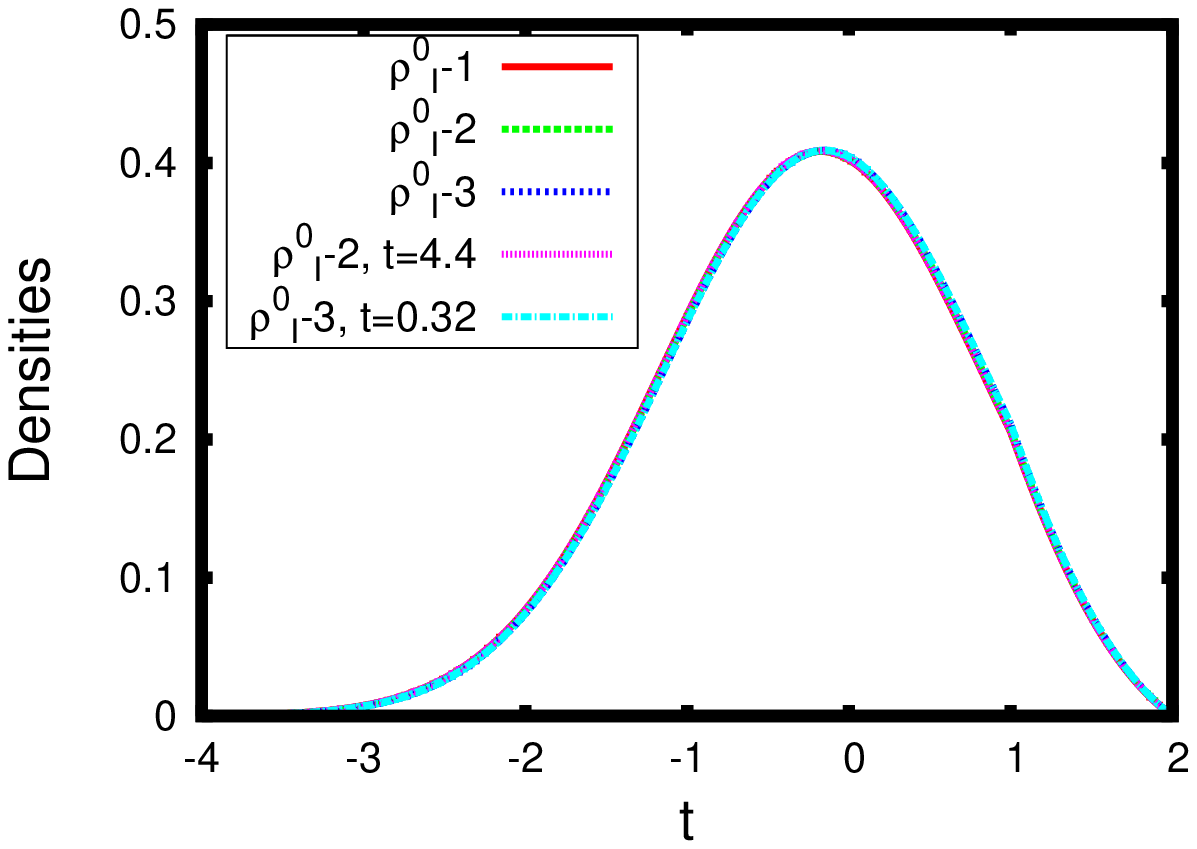}
\end{center}
\end{minipage}
\end{center}
\begin{center}
\begin{minipage}[c]{0.33\linewidth}
\begin{center}
\includegraphics[width=\textwidth]{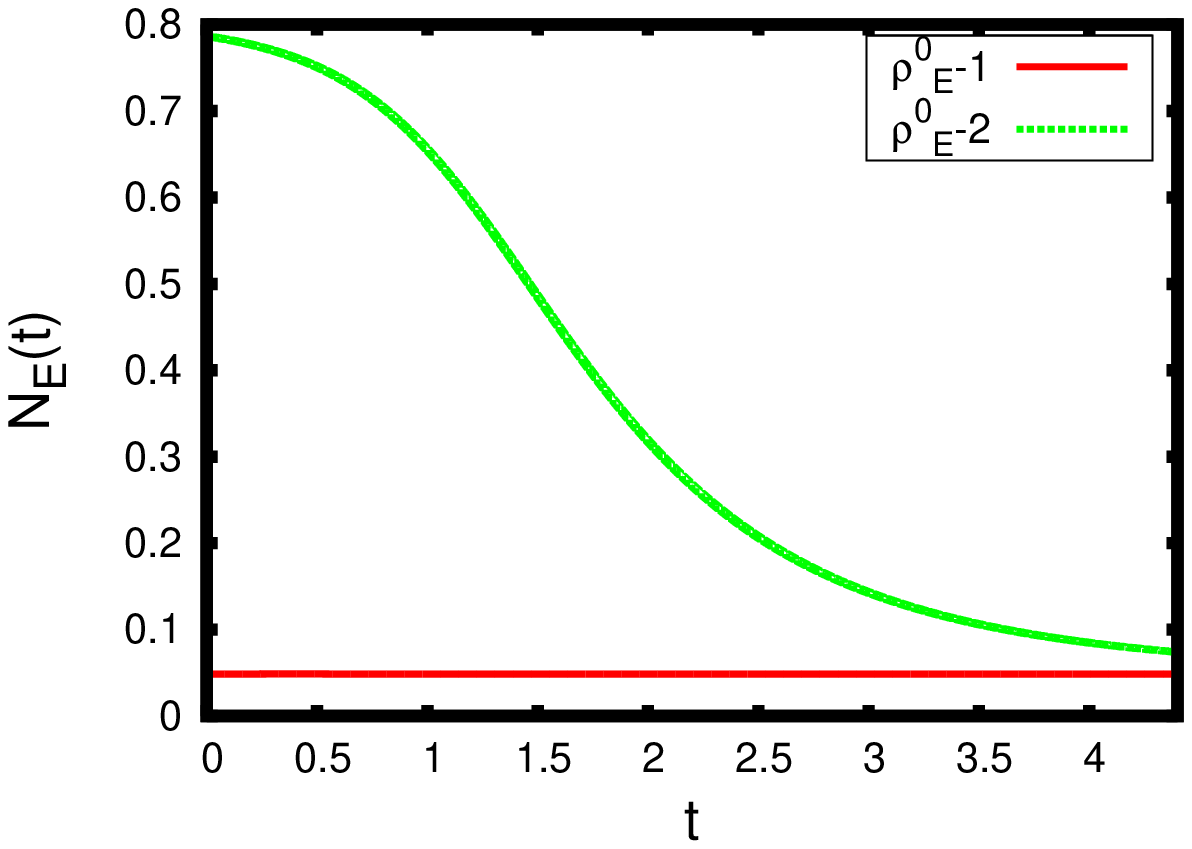}
\end{center}
\end{minipage}
\begin{minipage}[c]{0.33\linewidth}
\begin{center}
\includegraphics[width=\textwidth]{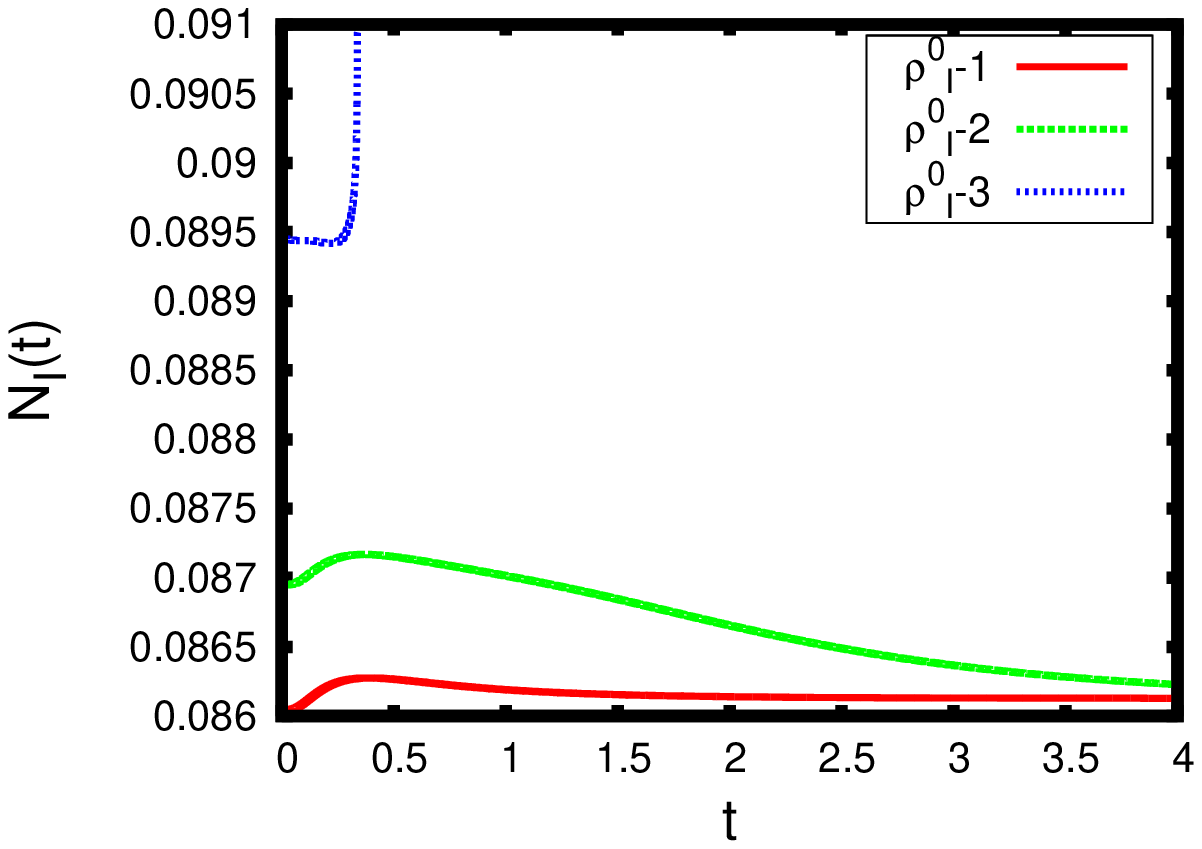}
\end{center}
\end{minipage}
\end{center}
\begin{center}
\begin{minipage}[c]{0.33\linewidth}
\begin{center}
\includegraphics[width=\textwidth]{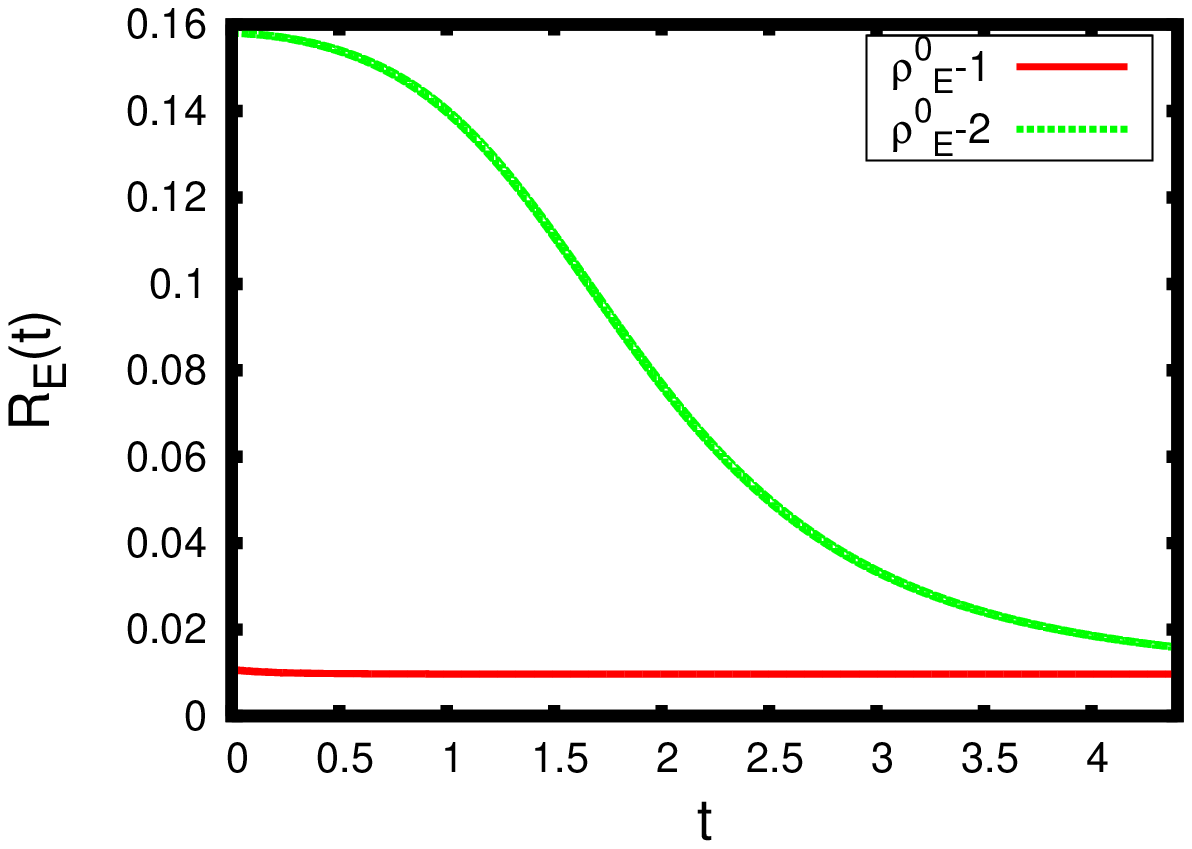}
\end{center}
\end{minipage}
\begin{minipage}[c]{0.33\linewidth}
\begin{center}
\includegraphics[width=\textwidth]{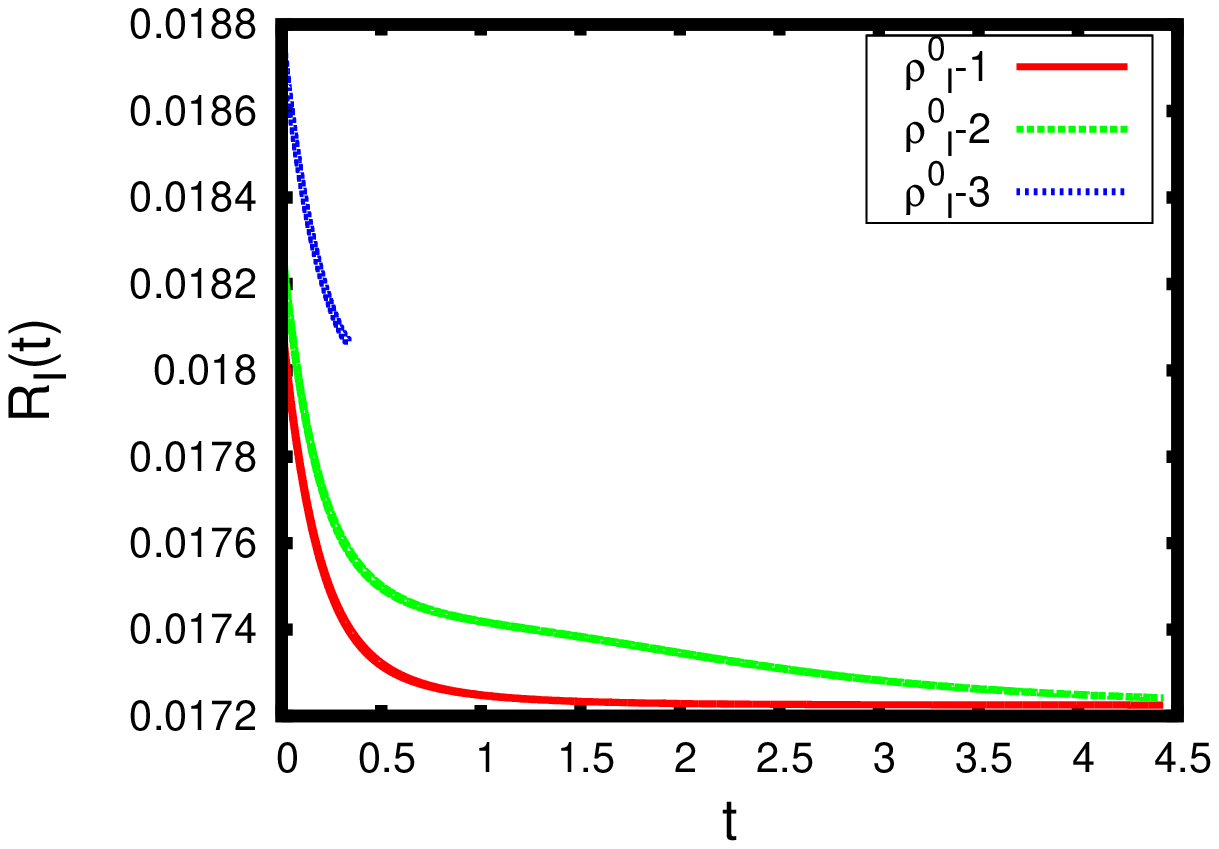}
\end{center}
\end{minipage}
\end{center}
\caption{
{\bf Numerical analysis of the stability in the case of three steady
states for the system \eqref{modelo}.-}
If $b_I^E=7$, $b_I^I=2$, $b_E^I=0.01$, $\tau_E=\tau_I=0.2$
and $b_E^E=3$, there are three steady states (see Fig. \ref{bif_EI}) . 
Top: Initial conditions $\rho_\alpha^0-1,2,3$ given by
the profile \eqref{soleq}, where $N_\alpha$ are approximations of the
stationary firing rates, and evolution of densities 2 and 3 after some time.
Middle: Evolution of the excitatory firing rates and the refractory
states.
Bottom: Evolution of the inhibitory firing rates and the refractory
states.
\newline
We observe that the lowest steady state is stable and the other two
are unstable. 
}
\label{est_eq_EI1}
\end{figure}
\begin{figure}[H]
\begin{center}
\begin{minipage}[c]{0.33\linewidth}
\begin{center}
\includegraphics[width=\textwidth]{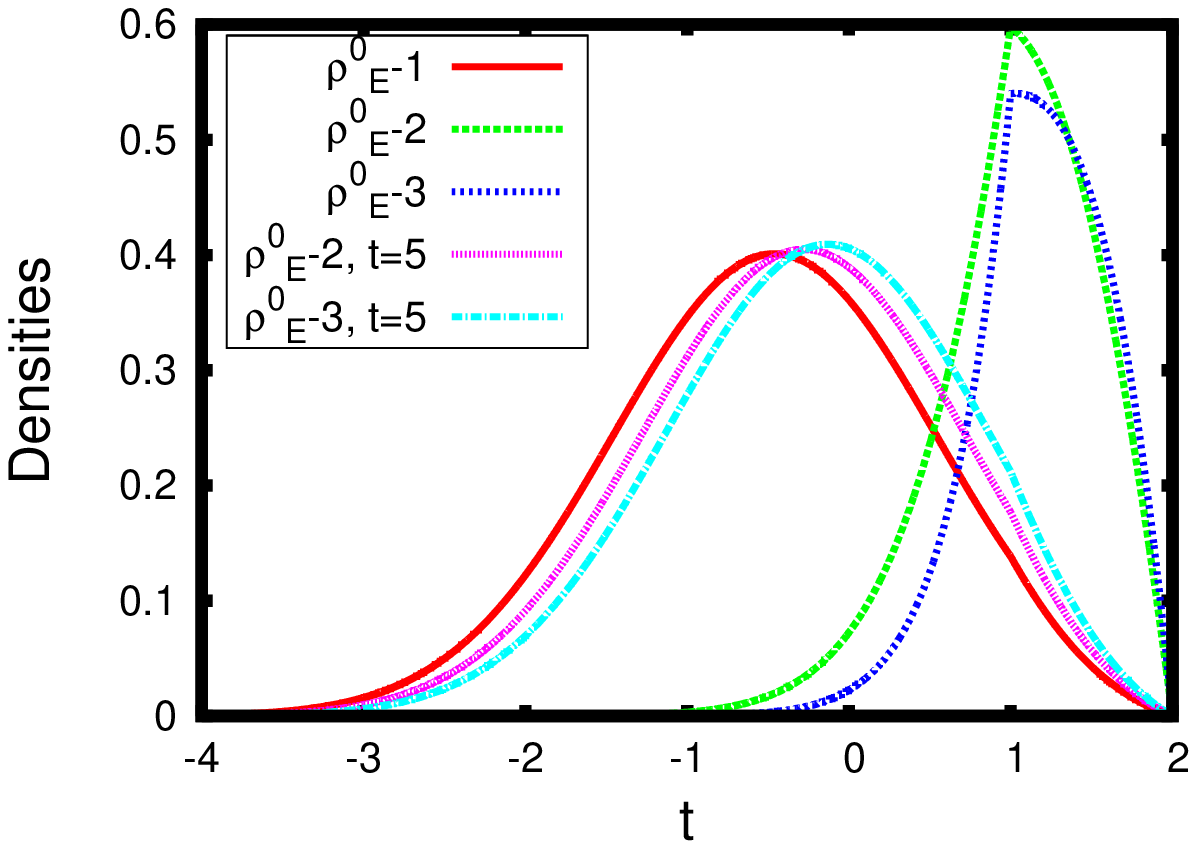}
\end{center}
\end{minipage}
\begin{minipage}[c]{0.33\linewidth}
\begin{center}
\includegraphics[width=\textwidth]{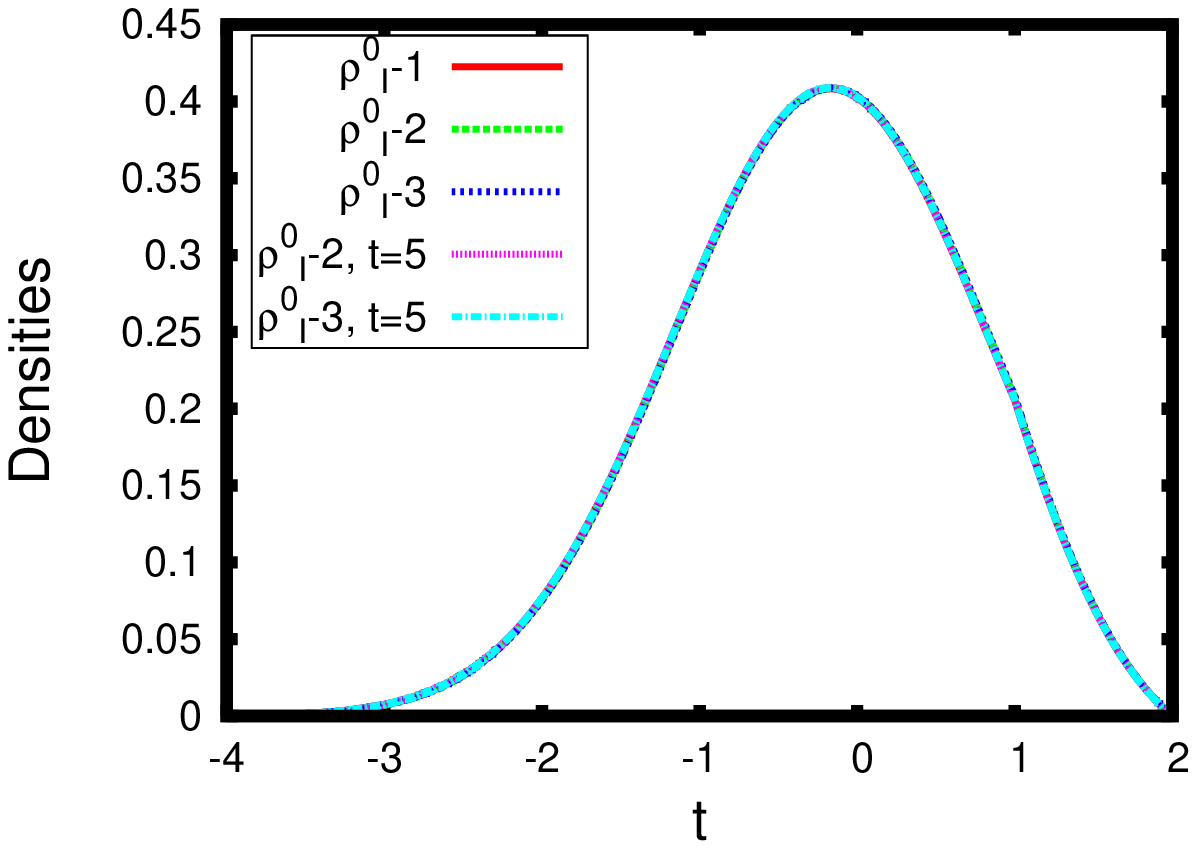}
\end{center}
\end{minipage}

\end{center}
\begin{center}
\begin{minipage}[c]{0.33\linewidth}
\begin{center}
\includegraphics[width=\textwidth]{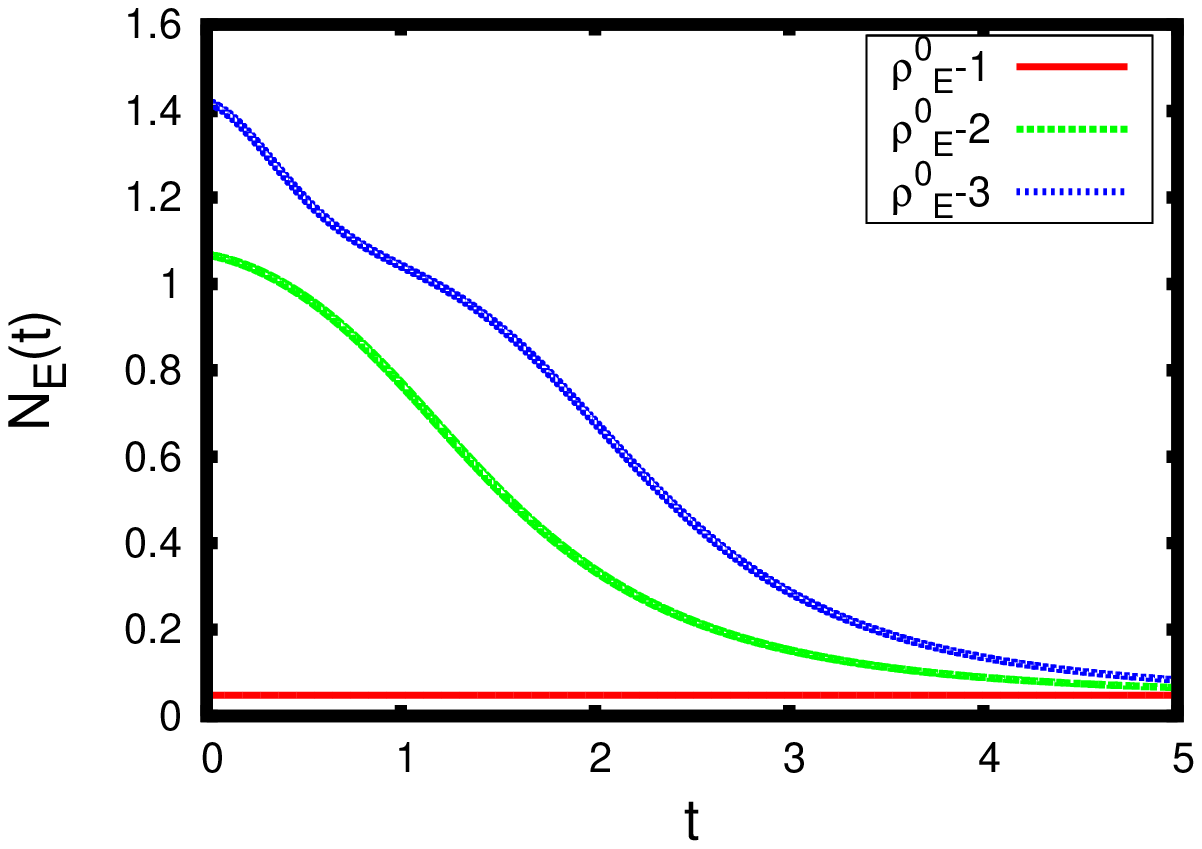}
\end{center}
\end{minipage}
\begin{minipage}[c]{0.33\linewidth}
\begin{center}
\includegraphics[width=\textwidth]{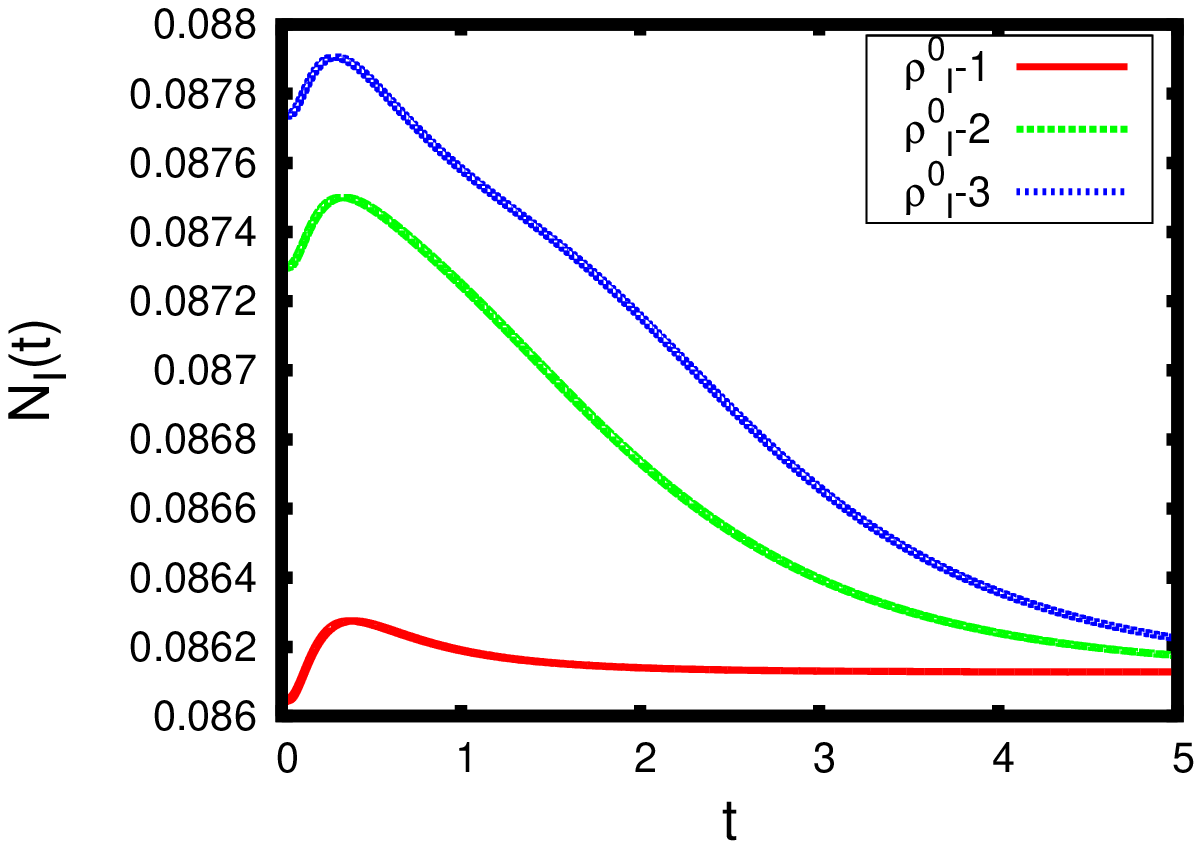}
\end{center}
\end{minipage}
\end{center}
\begin{center}
\begin{minipage}[c]{0.33\linewidth}
\begin{center}
\includegraphics[width=\textwidth]{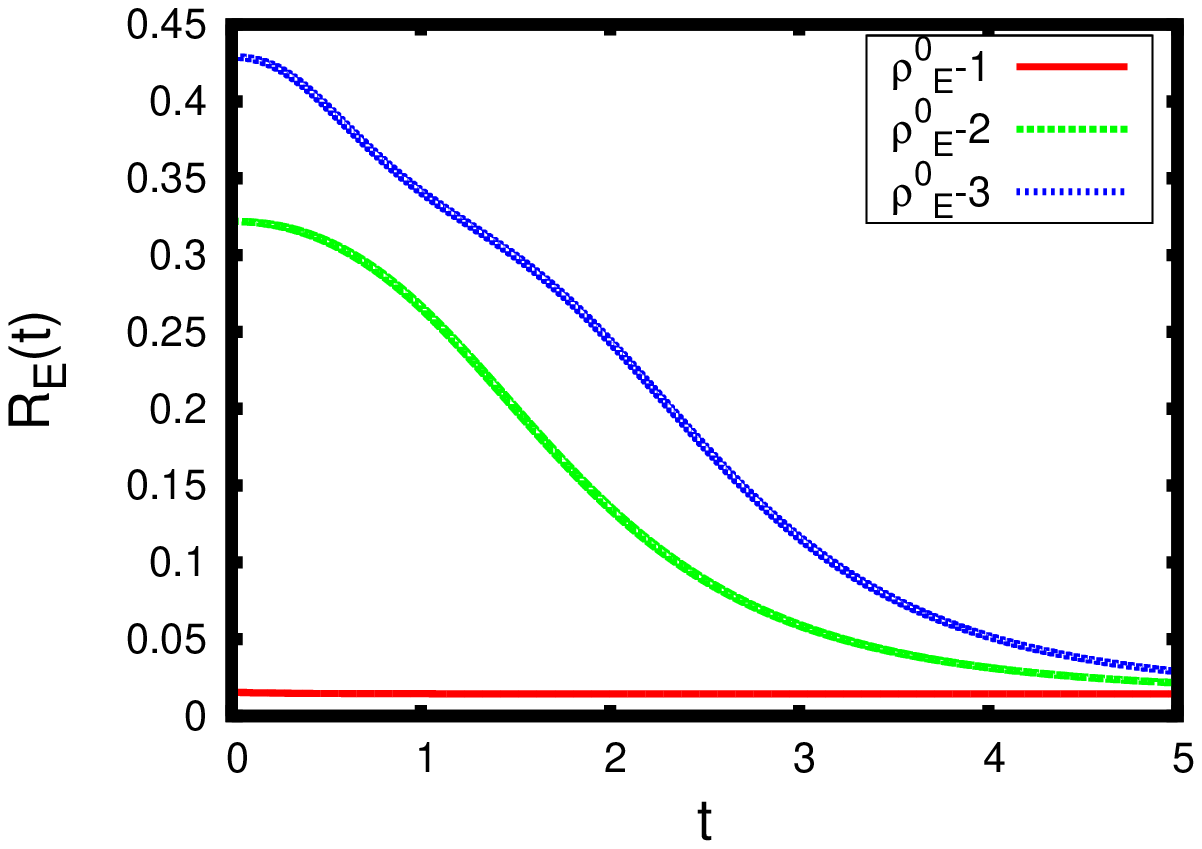}
\end{center}
\end{minipage}
\begin{minipage}[c]{0.33\linewidth}
\begin{center}
\includegraphics[width=\textwidth]{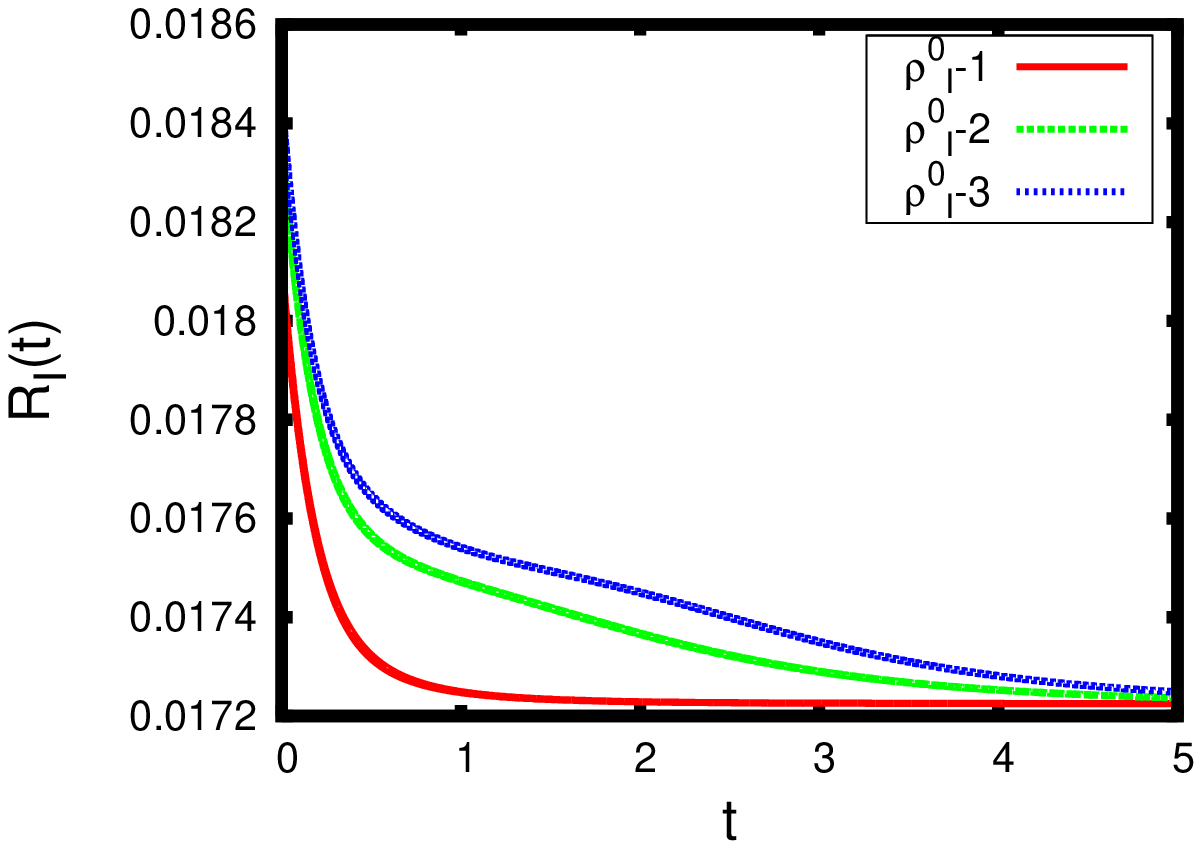}
\end{center}
\end{minipage}
\end{center}
\caption{
{\bf Numerical analysis of the stability in the case of three steady
states for the system \eqref{modelo}.-}
If $b_I^E=7$, $b_I^I=2$, $b_E^I=0.01$, $\tau_E=0.3$, $\tau_I=0.2$
and $b_E^E=3$, there are three steady states (see Fig. \ref{bif_EI}) . 
Top: Initial conditions $\rho_\alpha^0-1,2,3$ given by
the profile \eqref{soleq}, where $N_\alpha$ are approximations of the
stationary firing rates, and evolution of densities 2 and 3 after some time.
Middle: Evolution of the excitatory firing rates and the refractory
states.
Bottom: Evolution of the inhibitory firing rates and the refractory
states.
\newline
We observe that the lowest steady state is stable and the other two
are unstable. 
}
\label{est_eq_EI2}
\end{figure}
\begin{figure}[H]
\begin{center}
\begin{minipage}[c]{0.33\linewidth}
\begin{center}
\includegraphics[width=\textwidth]{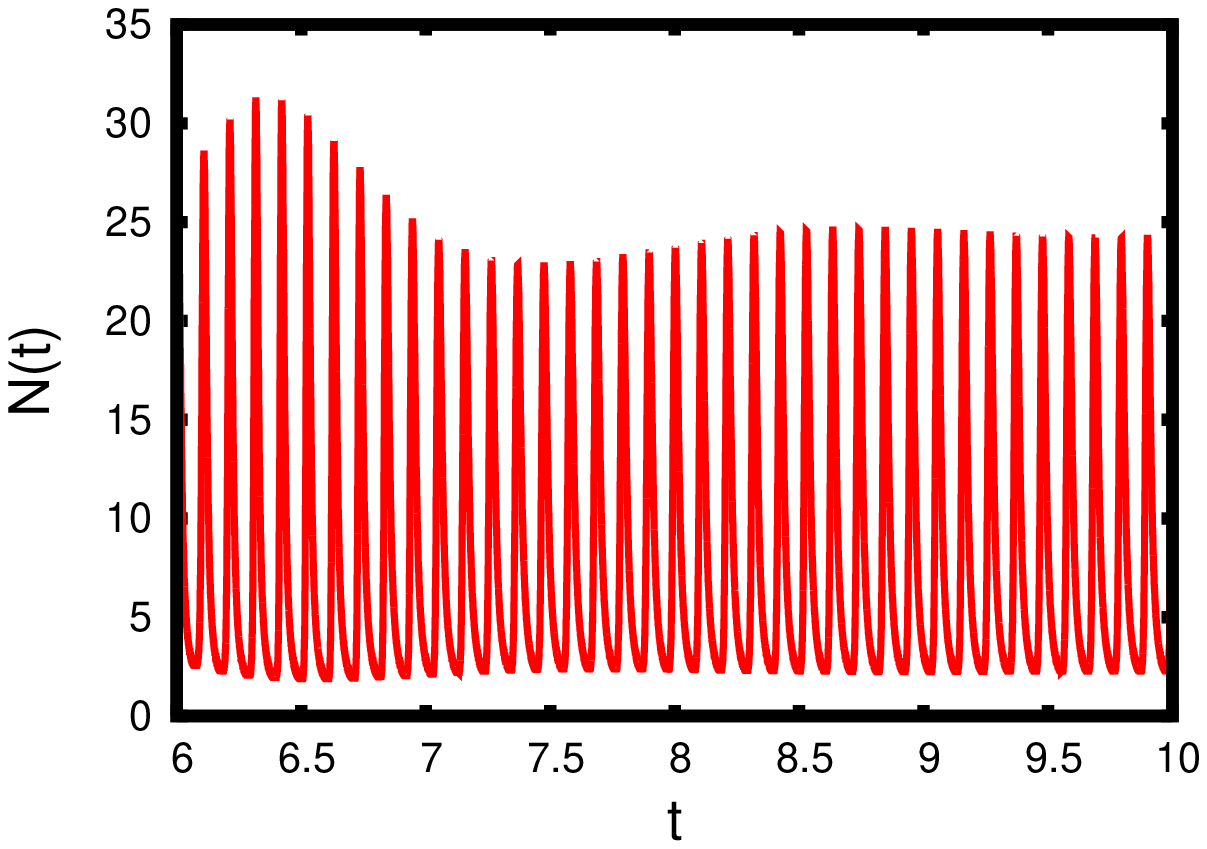}
\end{center}
\end{minipage}
\begin{minipage}[c]{0.33\linewidth}
\begin{center}
\includegraphics[width=\textwidth]{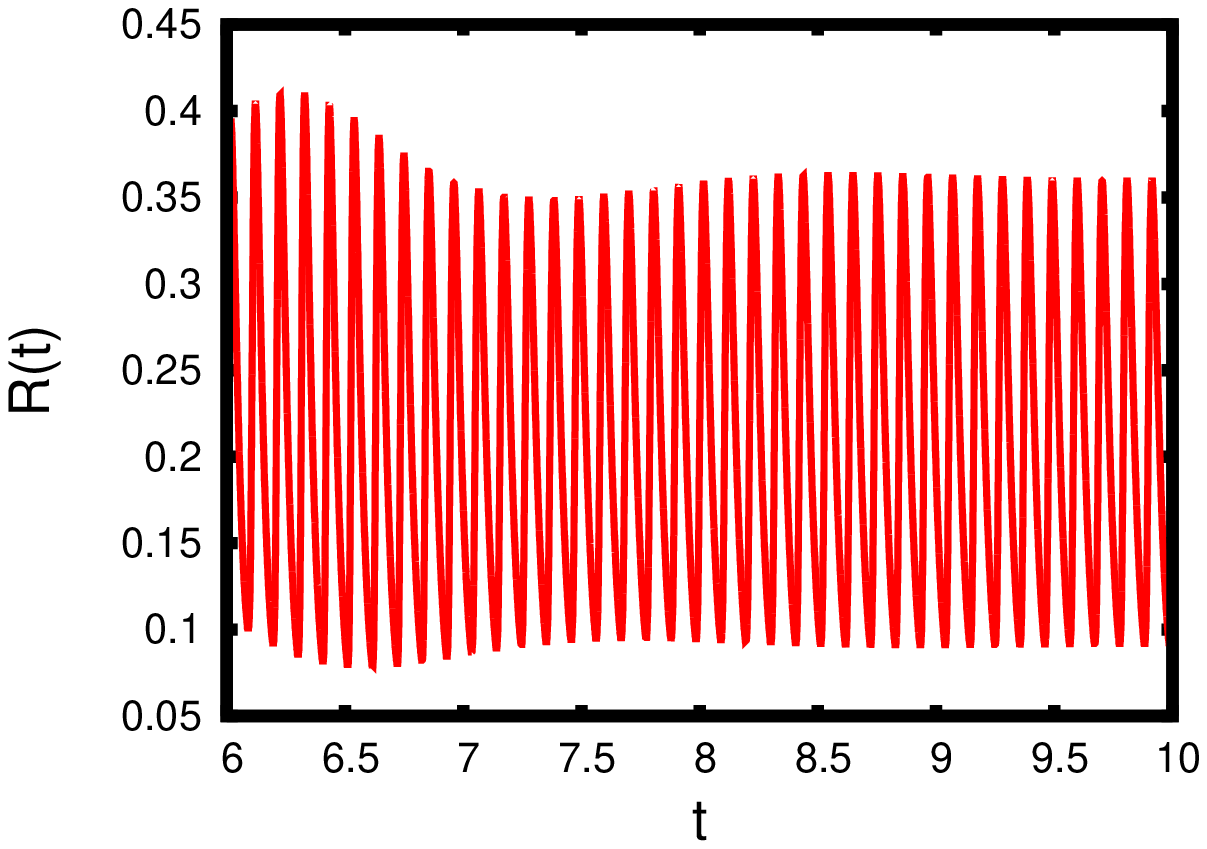}
\end{center}
\end{minipage}
\end{center}
\begin{center}
\begin{minipage}[c]{0.33\linewidth}
\begin{center}
\includegraphics[width=\textwidth]{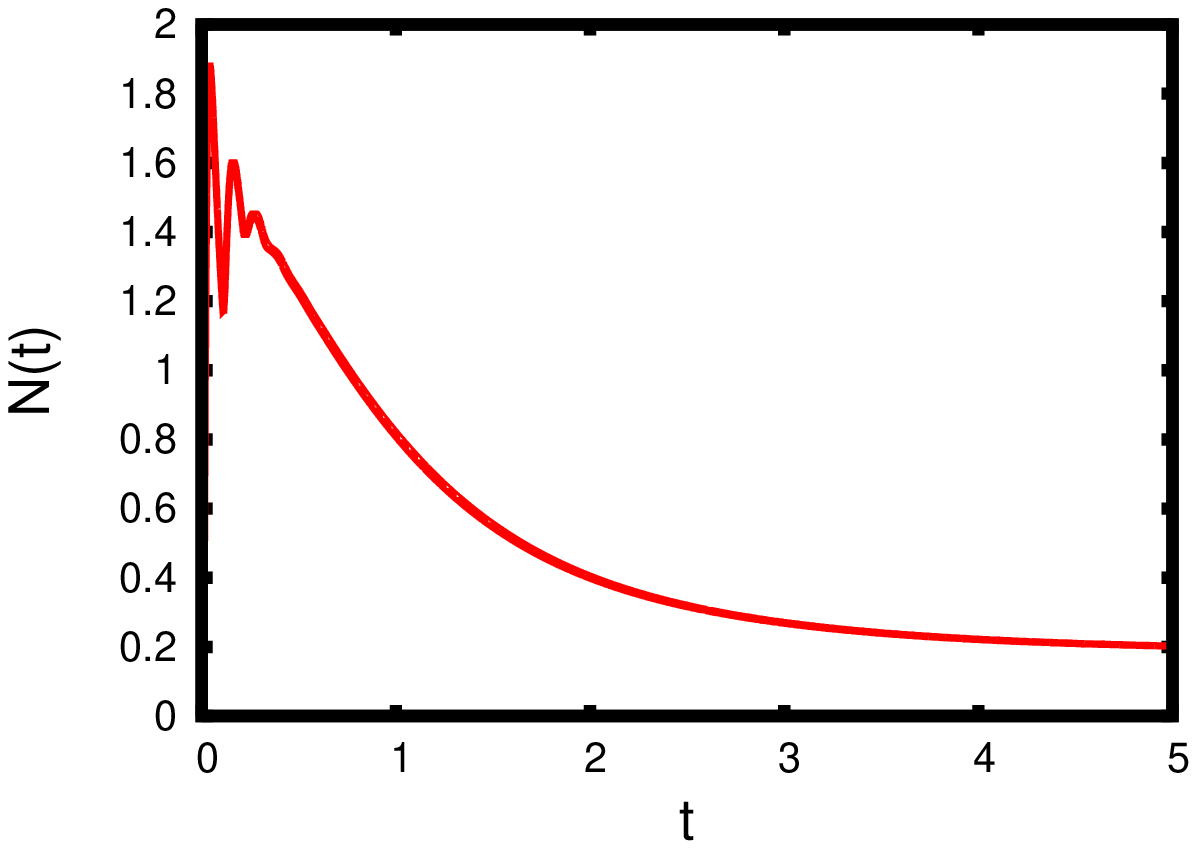}
\end{center}
\end{minipage}
\begin{minipage}[c]{0.33\linewidth}
\begin{center}
\includegraphics[width=\textwidth]{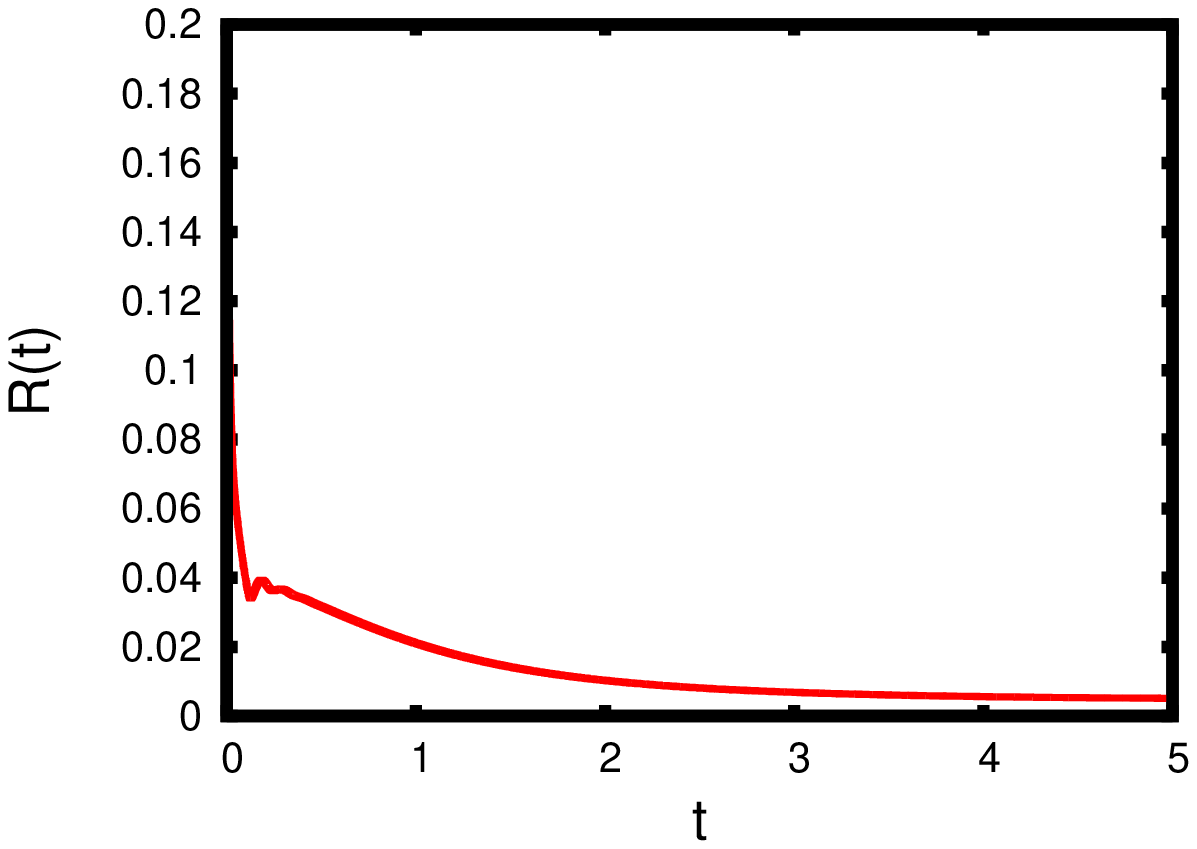}
\end{center}
\end{minipage}
\end{center}
\caption{{\bf System \eqref{nota} (only one average-excitatory
population) presents periodic 
solutions, if there is a transmission delay.-}
We consider initial data \eqref{ci_maxwel} with
 $\sigma_0=0.0003$,  the connectivity parameter  $b=1.5$,
the transmission delay  $D=0.1$,  $v_{ext}=0$ and with refractory 
states ($M(t)=\frac{R(t)}{\tau}$), where $\tau=0.025$ and $R(0)=0.2$.
\newline
Periodic solutions appear if the initial condition is
concentrated enough  around the
threshold potential
Top: $v_0=1.83$. Botton: $v_0=1.5$.}
\label{osci_R_MJ-1}
\end{figure}
\begin{figure}[H]
\begin{center}
\begin{minipage}[c]{0.33\linewidth}
\begin{center}
\includegraphics[width=\textwidth]{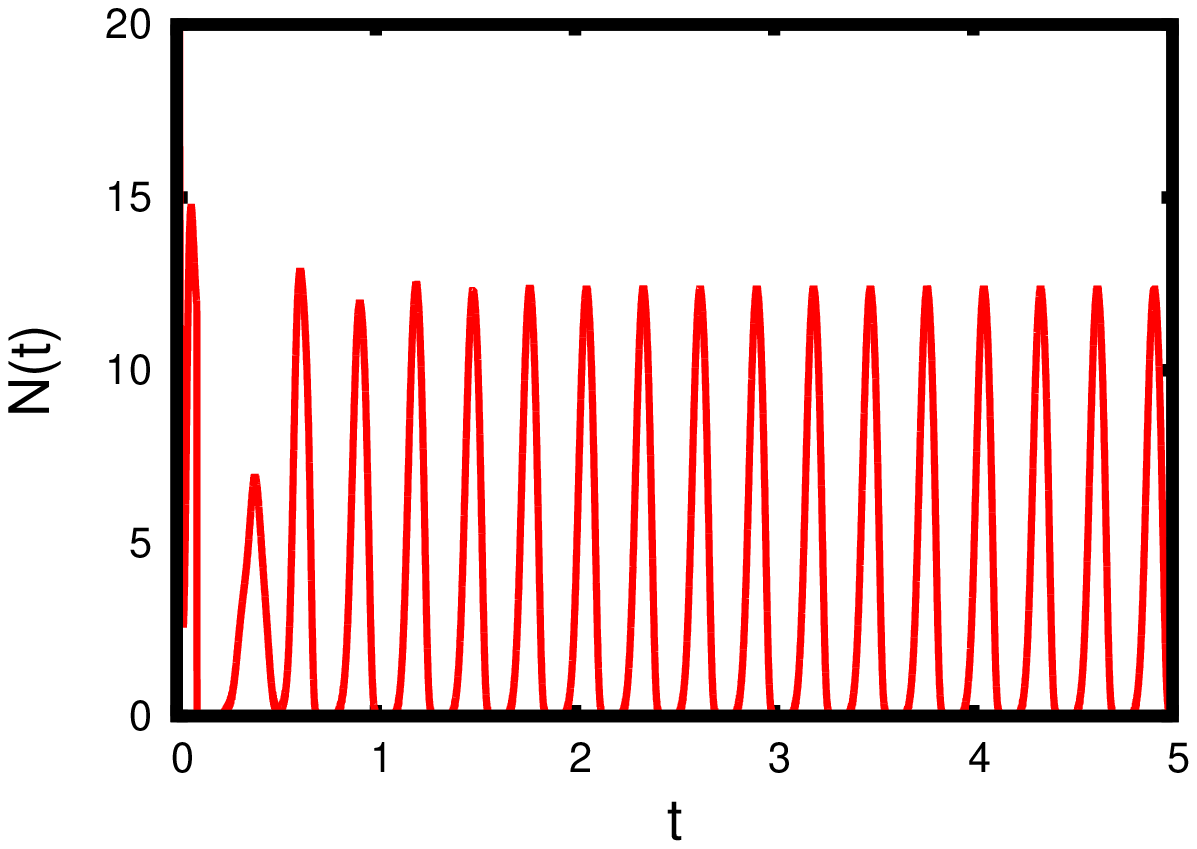}
\end{center}
\end{minipage}
\begin{minipage}[c]{0.33\linewidth}
\begin{center}
\includegraphics[width=\textwidth]{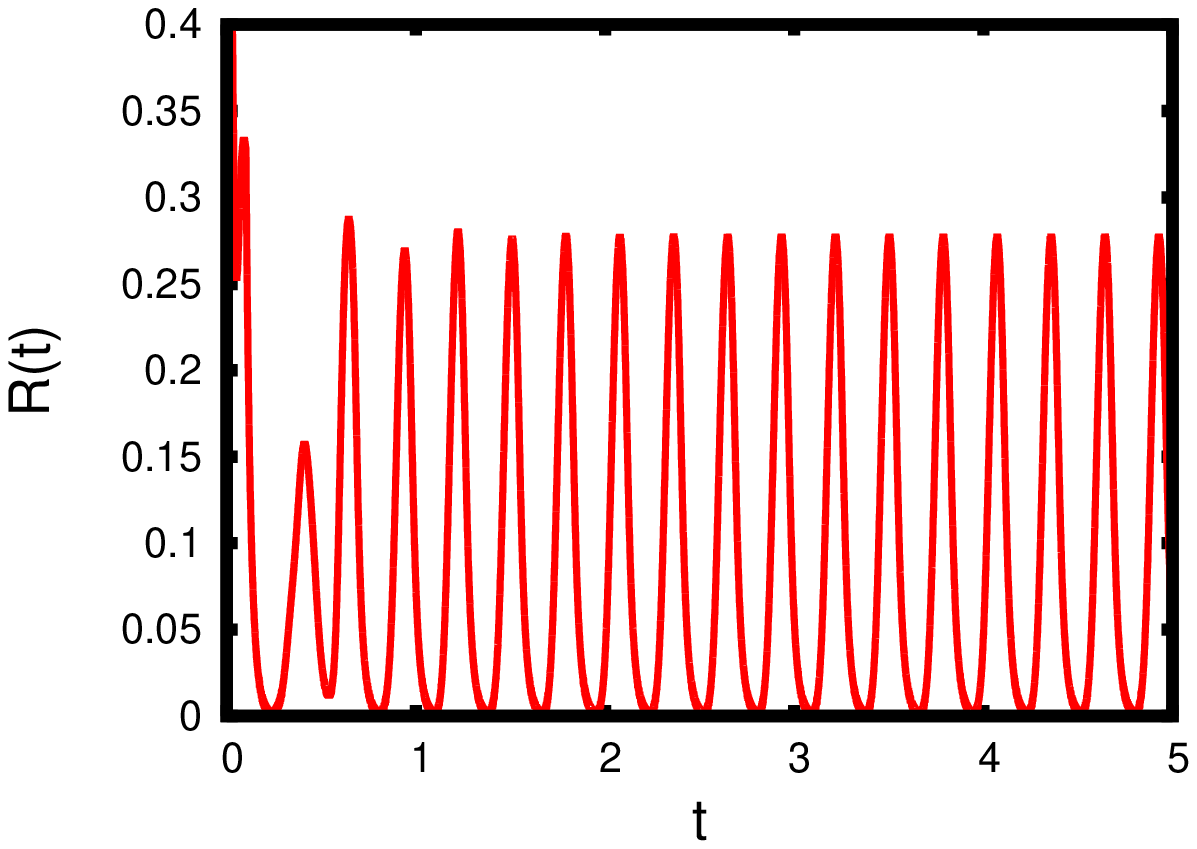}
\end{center}
\end{minipage}
\end{center}
\begin{center}
\begin{minipage}[c]{0.33\linewidth}
\begin{center}
\includegraphics[width=\textwidth]{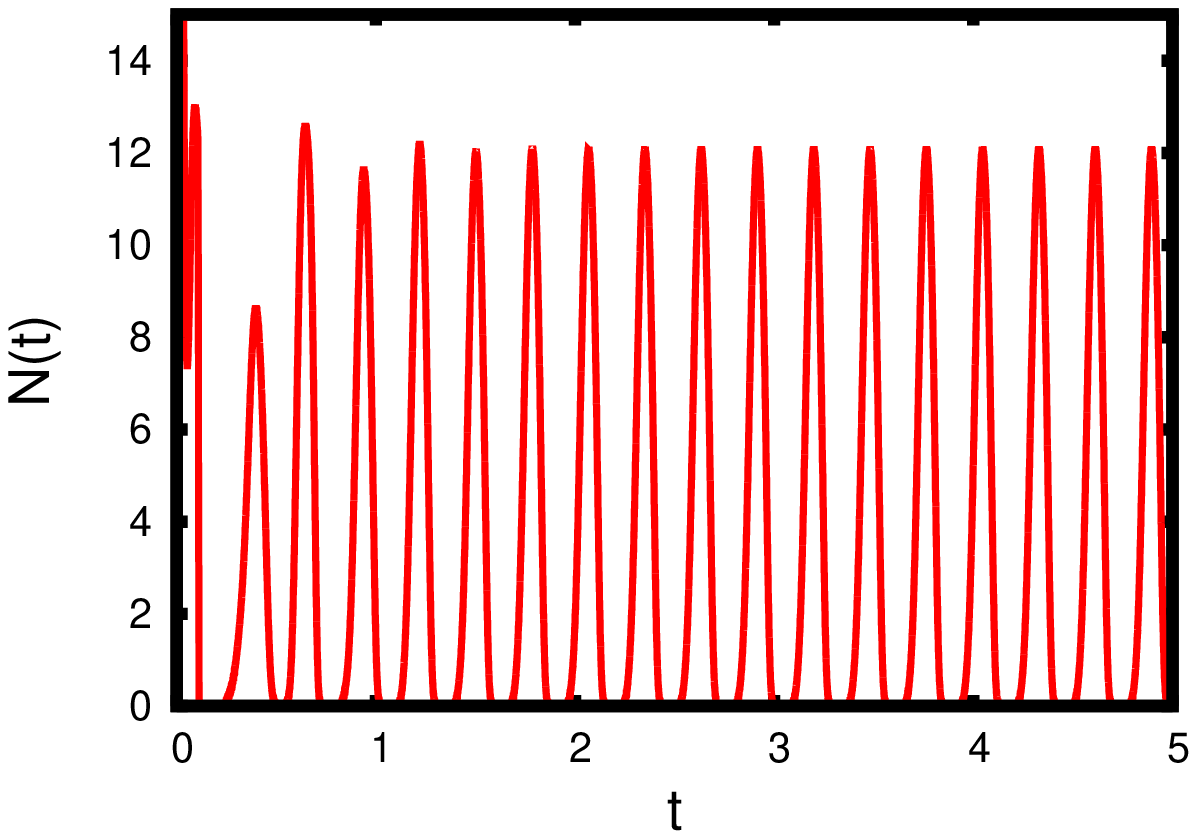}
\end{center}
\end{minipage}
\begin{minipage}[c]{0.33\linewidth}
\begin{center}
\includegraphics[width=\textwidth]{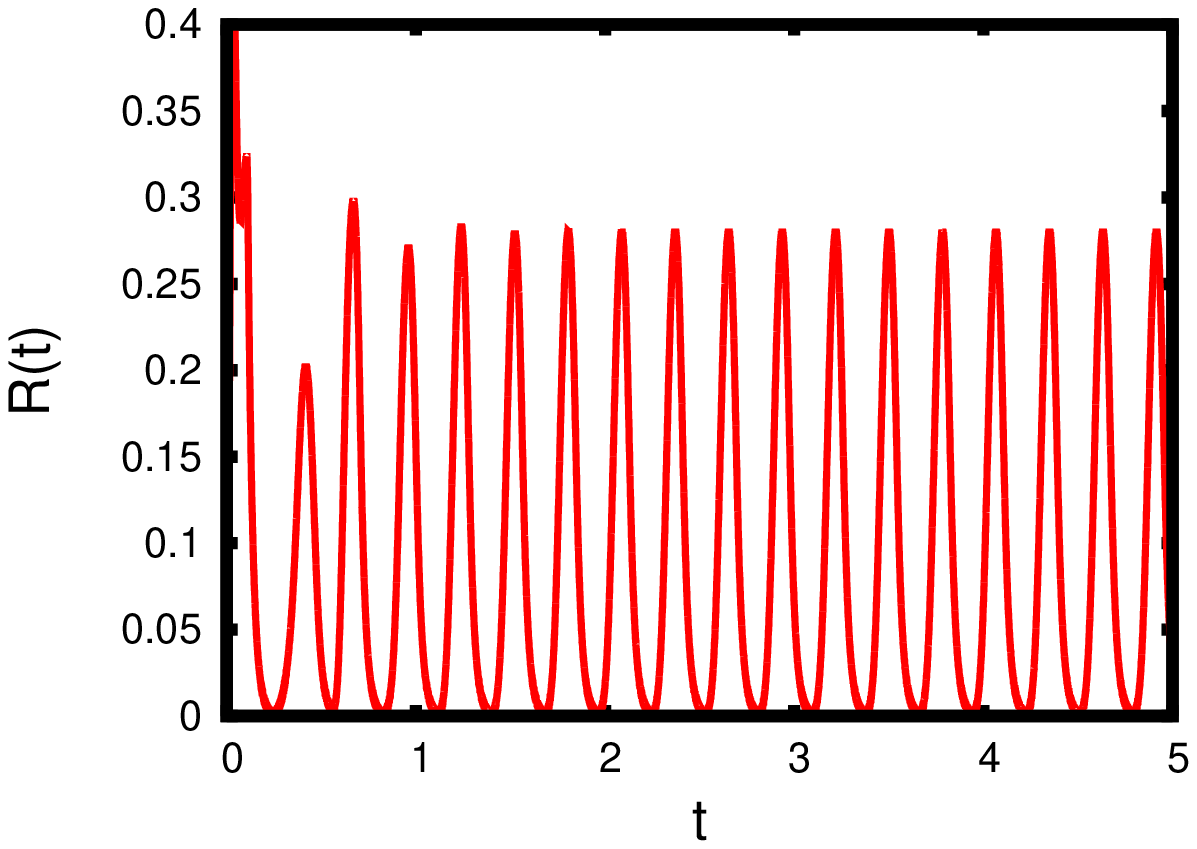}
\end{center}
\end{minipage}
\end{center}
\begin{center}
\begin{minipage}[c]{0.33\linewidth}
\begin{center}
\includegraphics[width=\textwidth]{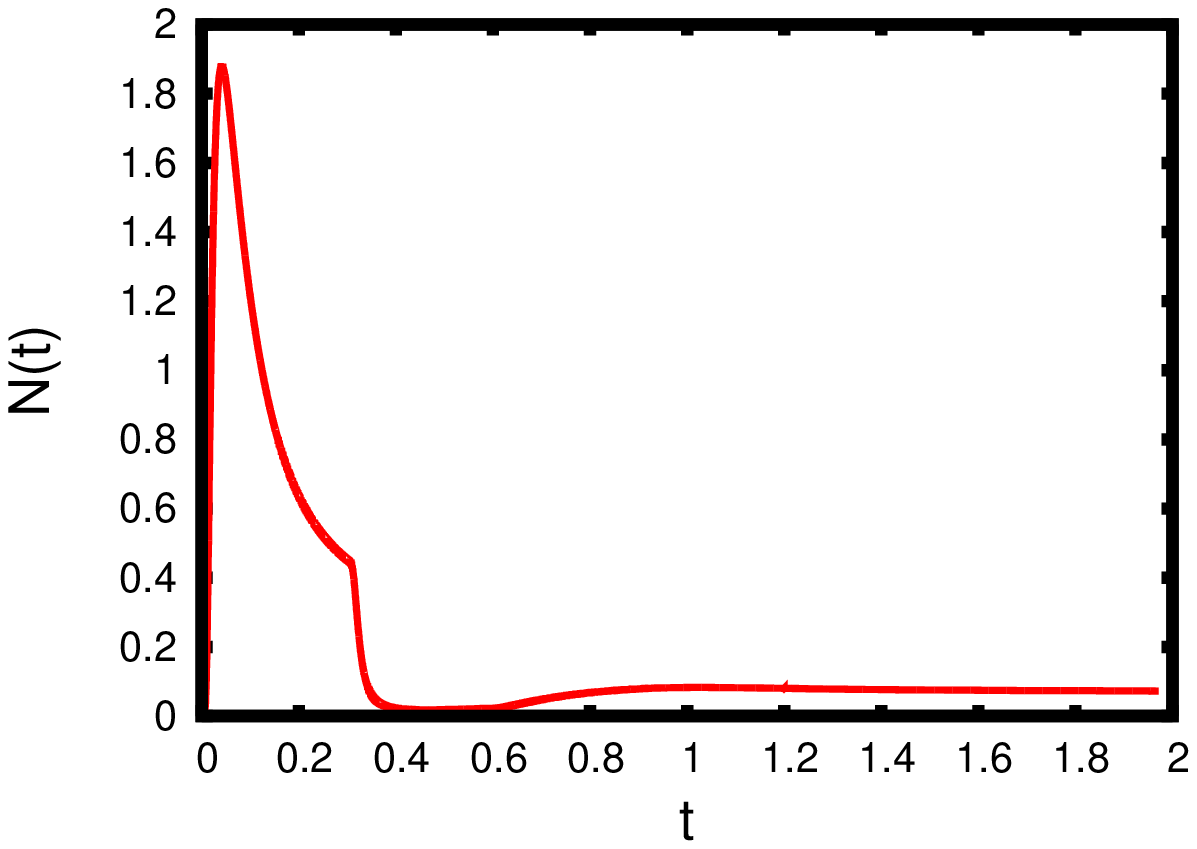}
\end{center}
\end{minipage}
\begin{minipage}[c]{0.33\linewidth}
\begin{center}
\includegraphics[width=\textwidth]{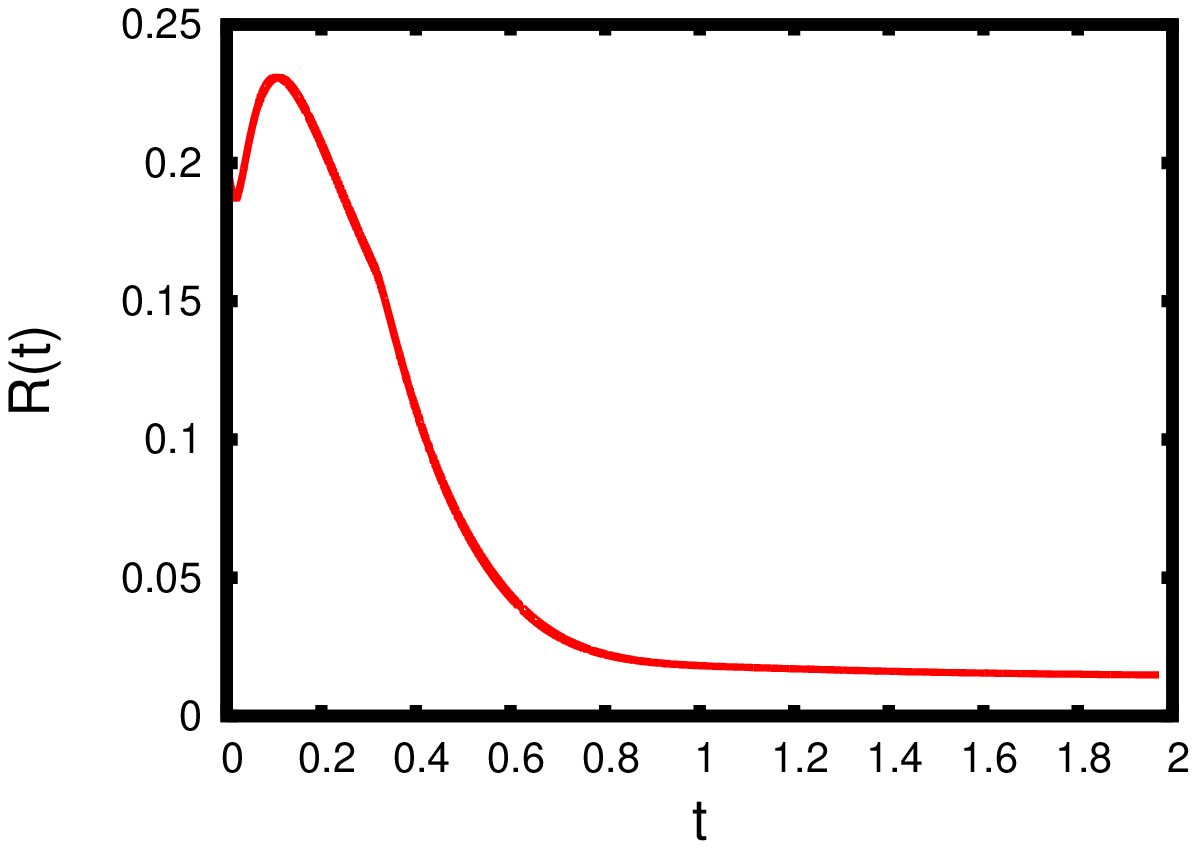}
\end{center}
\end{minipage}
\end{center}
\caption{
{\bf System \eqref{nota} (only one average-inhibitory
population) presents periodic 
solutions, if there is a transmission delay.-}
We consider initial data \eqref{ci_maxwel} with
 $\sigma_0=0.0003$,  the connectivity parameter  $b=-4$,
the transmission delay $D=0.1$,   and with refractory 
states ($M(t)=\frac{R(t)}{\tau}$), where $\tau=0.025$ and $R(0)=0.2$.
\newline
Periodic solutions appear if the initial condition is
concentrated enough  around the
threshold potential, but even if the initial datum is far from
the threshold and the $v_{ext}$ is large.
Top: $v_0=1.83$, $v_{ext}=20$. Middle: $v_0=1.5$, $v_{ext}=20$.
Bottom: $v_0=1.5$, $v_{ext}=0$.}
\label{osci_R_MJ-2}
\end{figure}
\begin{figure}[H]
\begin{center}
\begin{minipage}[c]{0.33\linewidth}
\begin{center}
\includegraphics[width=\textwidth]{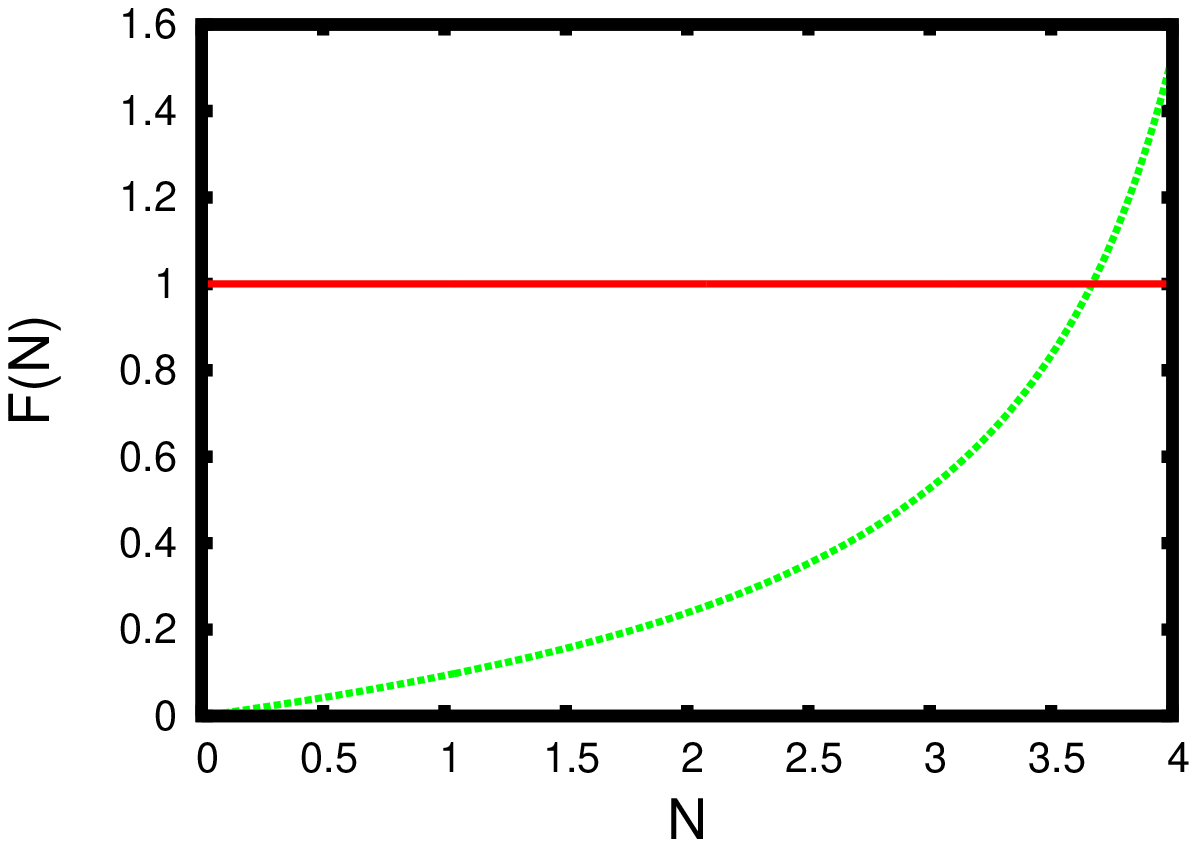}
\end{center}
\end{minipage}
\begin{minipage}[c]{0.33\linewidth}
\begin{center}
\includegraphics[width=\textwidth]{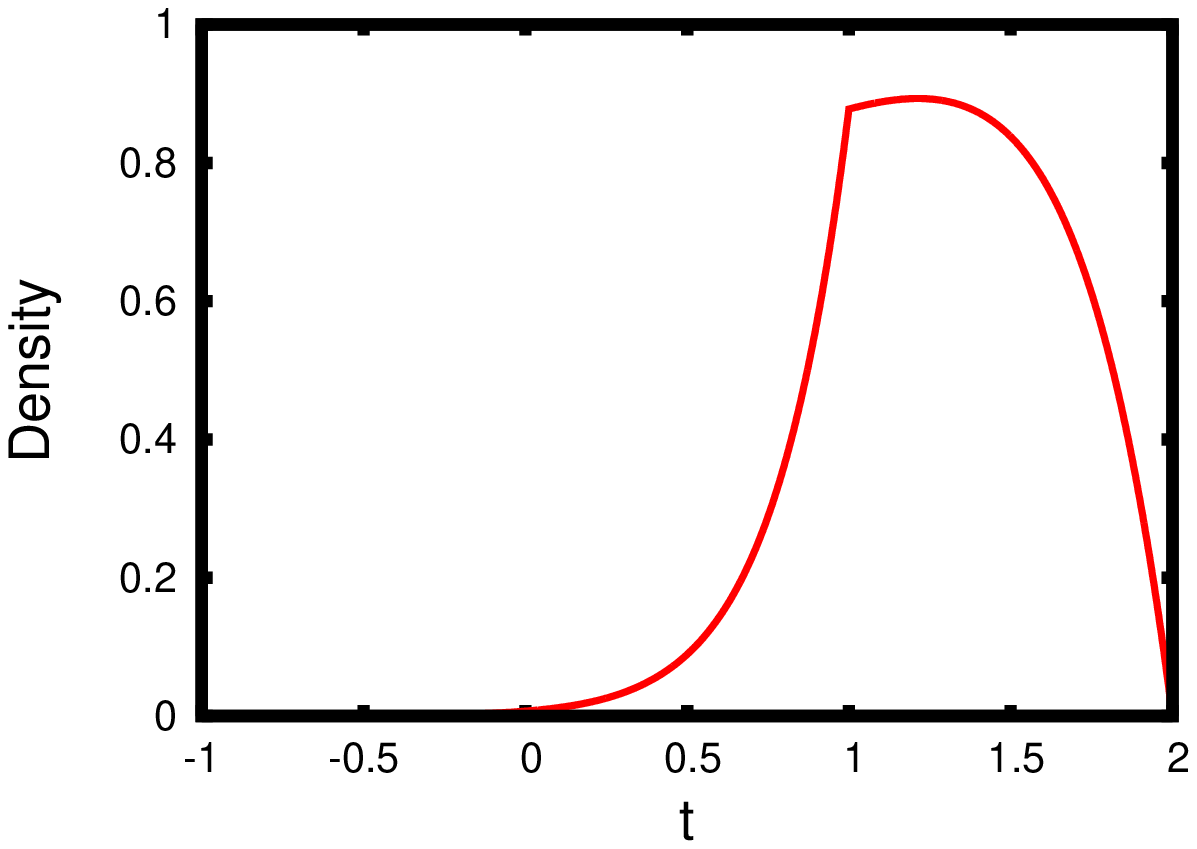}
\end{center}
\end{minipage}
\end{center}
\begin{center}
\begin{minipage}[c]{0.33\linewidth}
\begin{center}
\includegraphics[width=\textwidth]{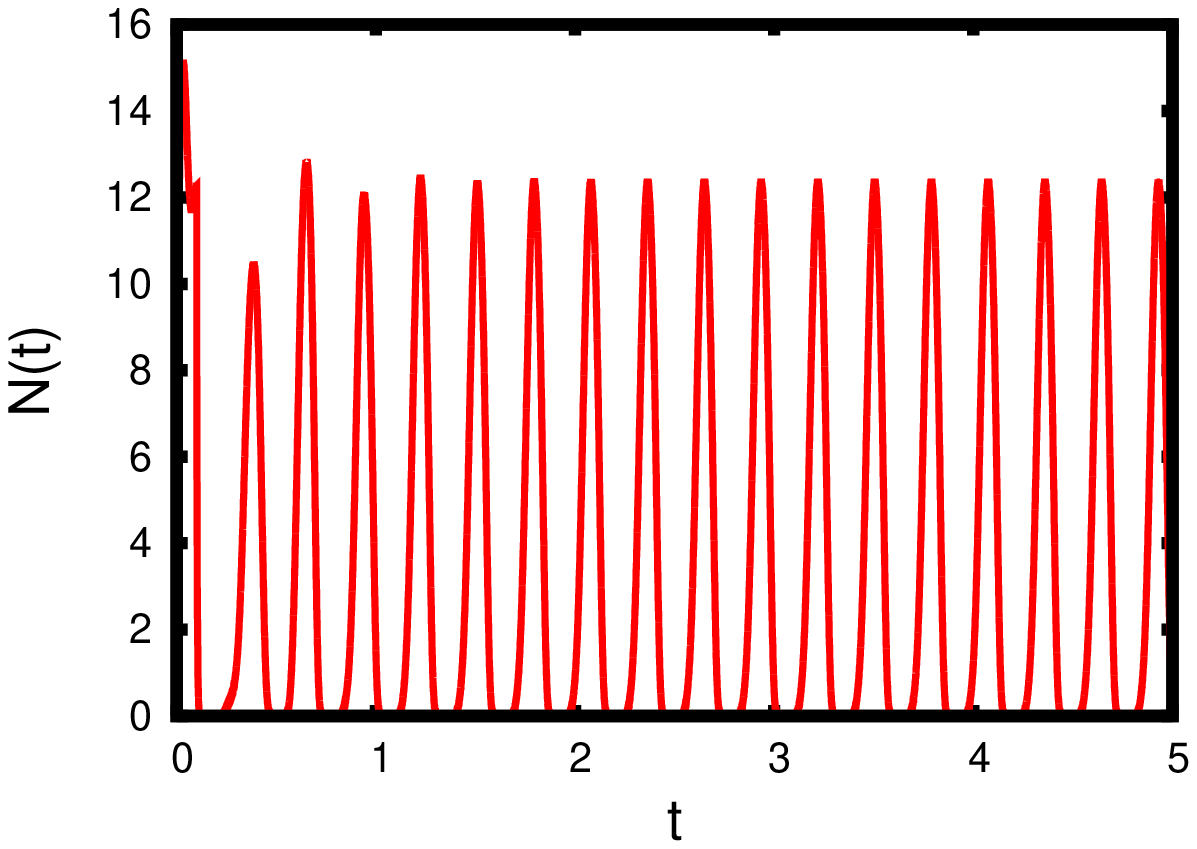}
\end{center}
\end{minipage}
\begin{minipage}[c]{0.33\linewidth}
\begin{center}
\includegraphics[width=\textwidth]{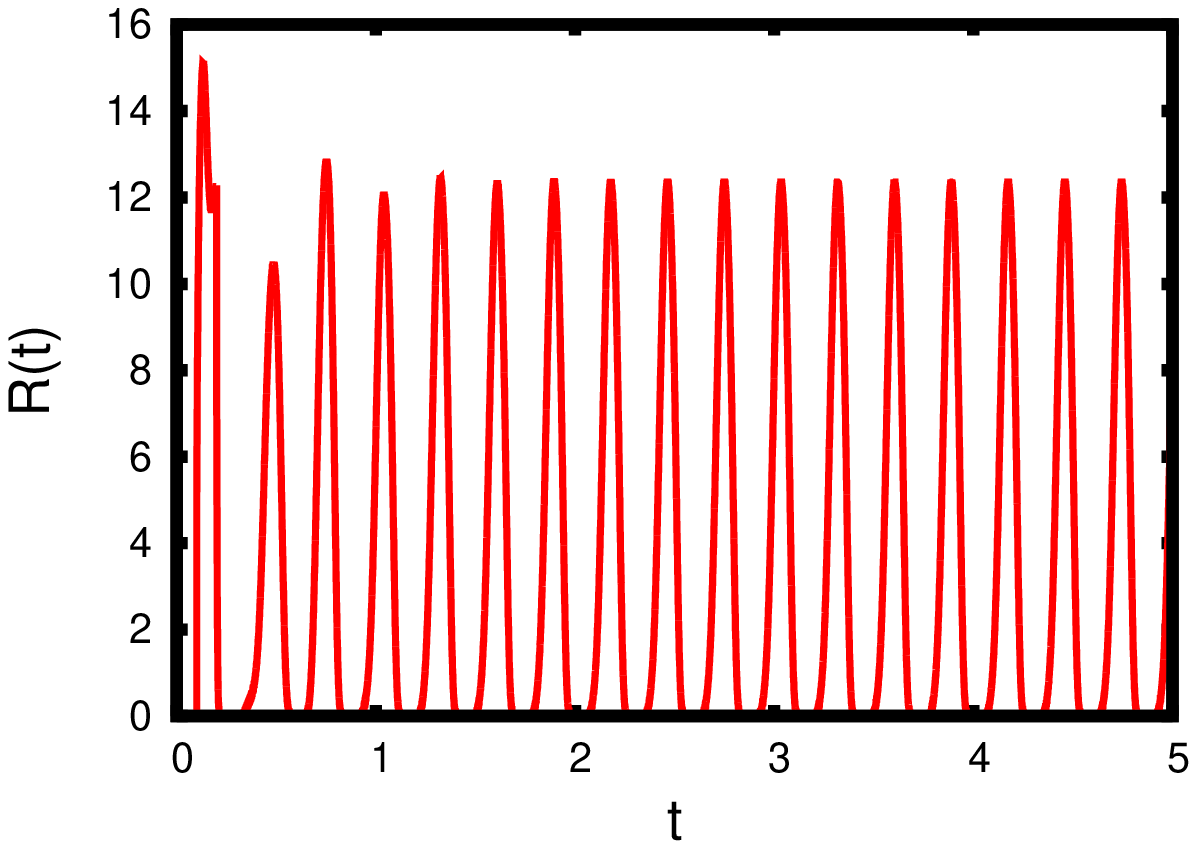}
\end{center}
\end{minipage}
\end{center}
\begin{center}
\begin{minipage}[c]{0.33\linewidth}
\begin{center}
\includegraphics[width=\textwidth]{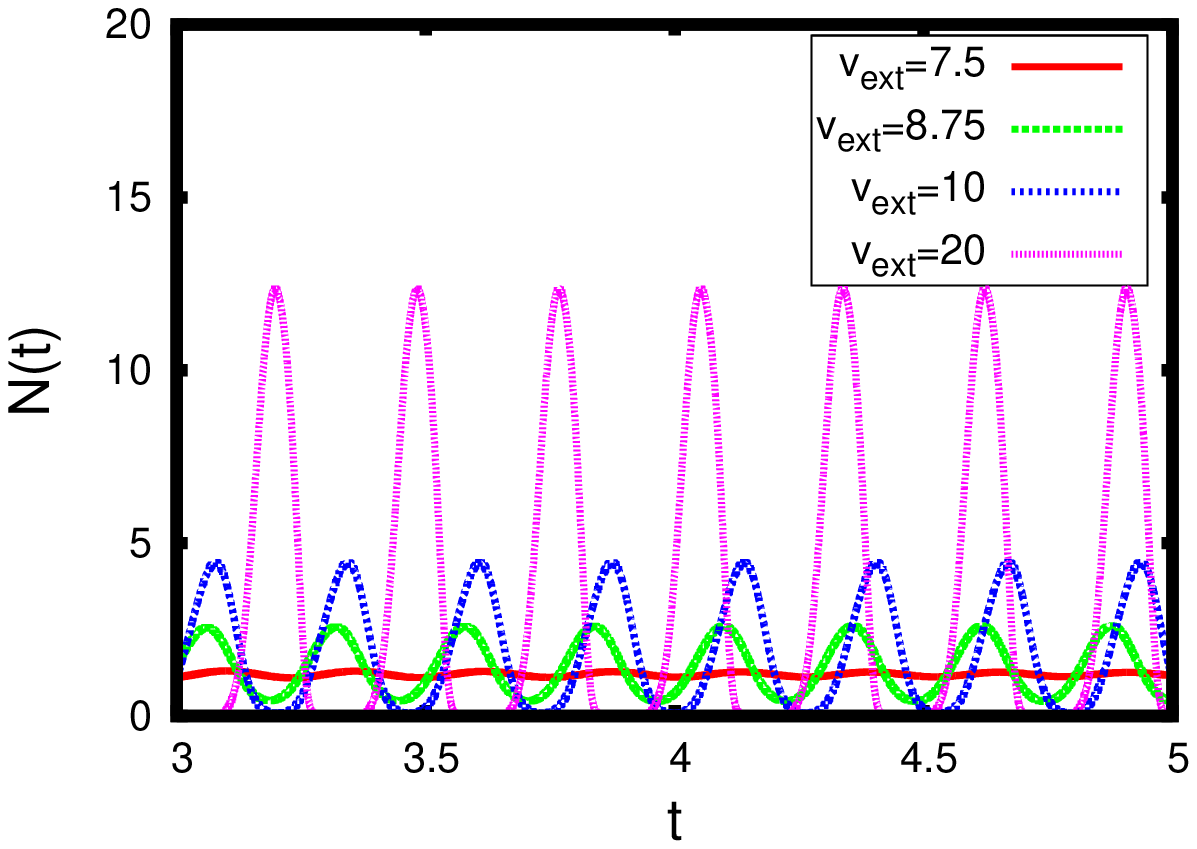}
\end{center}
\end{minipage}
\end{center}
\caption{
{\bf System \eqref{nota} (only one average-inhibitory population) presents periodic 
solutions, if there is a transmission delay.-}
We consider initial data \eqref{soleq} with
 $N=3.669$,  the connectivity parameter  $b=-4$,
the transmission delay  $D=0.1$, $v_{ext}=20$   and with refractory 
states ($M(t)=\frac{R(t)}{\tau}$), where $\tau=0.025$ and $R(0)=0.091725$.
\newline
Periodic solutions also appear if the initial condition (top right) is
very close to the unique equilibrium when $v_{ext}$ is large. Indeed, for this parameter space, solutions always converge 
 to the same periodic solution. Top:
Description of the unique steady state. Left: $F(N)=N(I(N)+\tau)$ 
 crosses with the constant function 1 giving the unique $N_\infty$.
Right: Unique steady state given by the profile \eqref{soleq} with firing rate  $N=3.669$.
Middle: Evolution of the firing rate and the refractory state 
for the solution with initial data given by \eqref{soleq} with firing rate  $N=3.669$.
Bottom: Influence of $v_{ext}$ in the behaviour of the system.}
\label{osci_R_MJ-3}
\end{figure}
\begin{figure}[H]
\begin{center}
\begin{minipage}[c]{0.33\linewidth}
\begin{center}
\includegraphics[width=\textwidth]{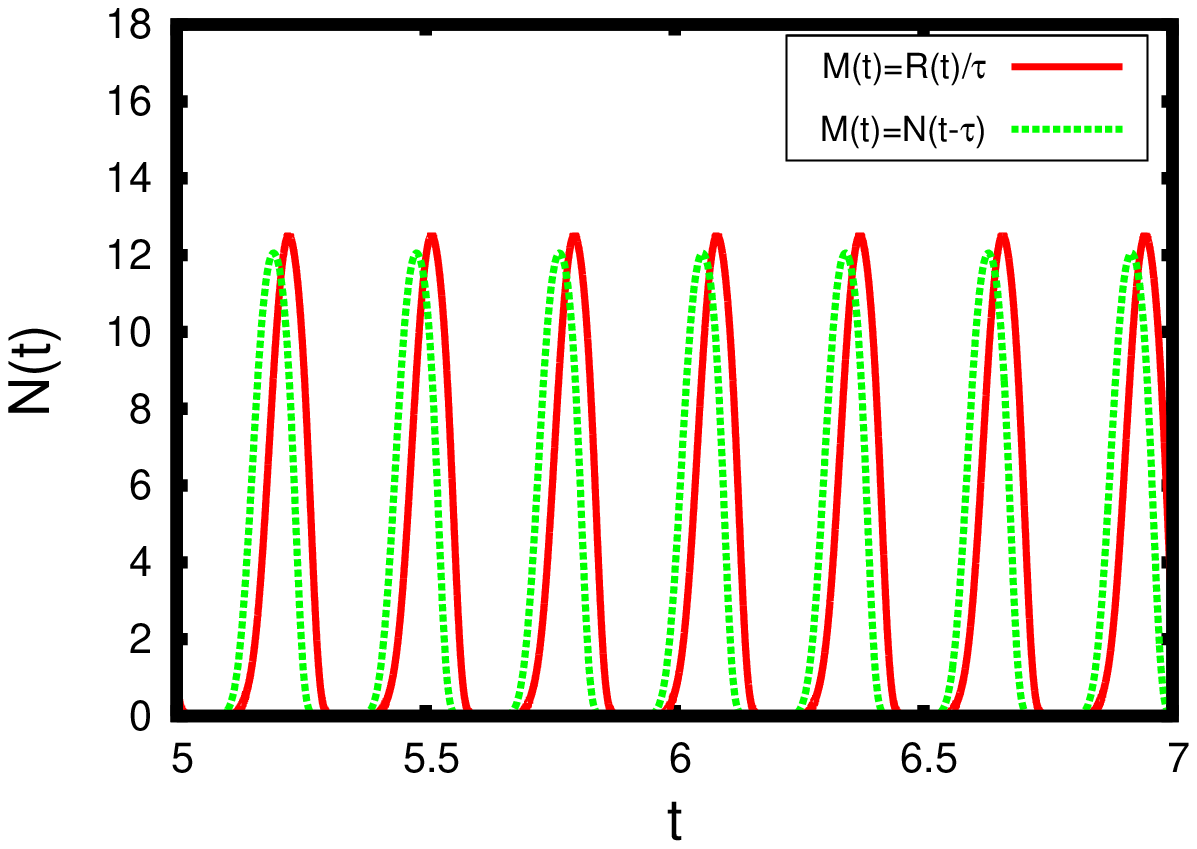}
\end{center}
\end{minipage}
\begin{minipage}[c]{0.33\linewidth}
\begin{center}
\includegraphics[width=\textwidth]{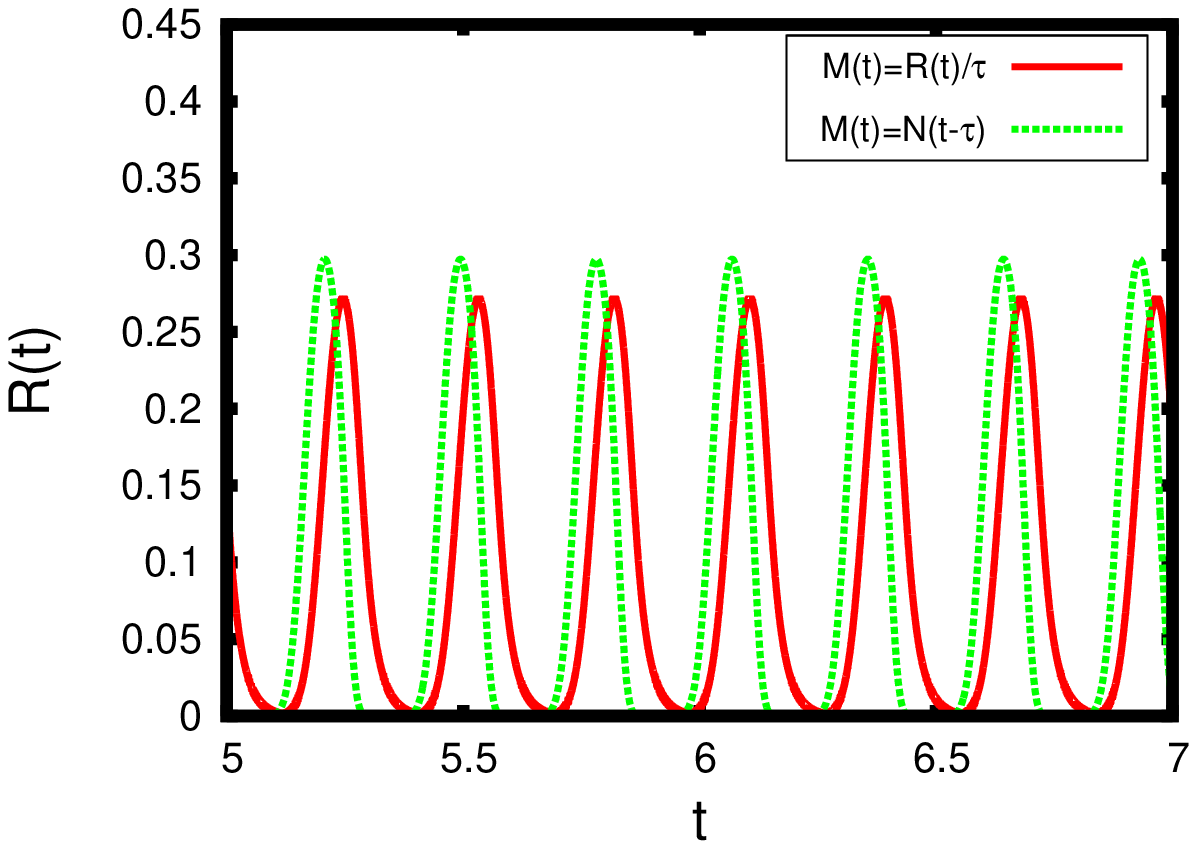}
\end{center}
\end{minipage}
\end{center}
\begin{center}
\begin{minipage}[c]{0.33\linewidth}
\begin{center}
\includegraphics[width=\textwidth]{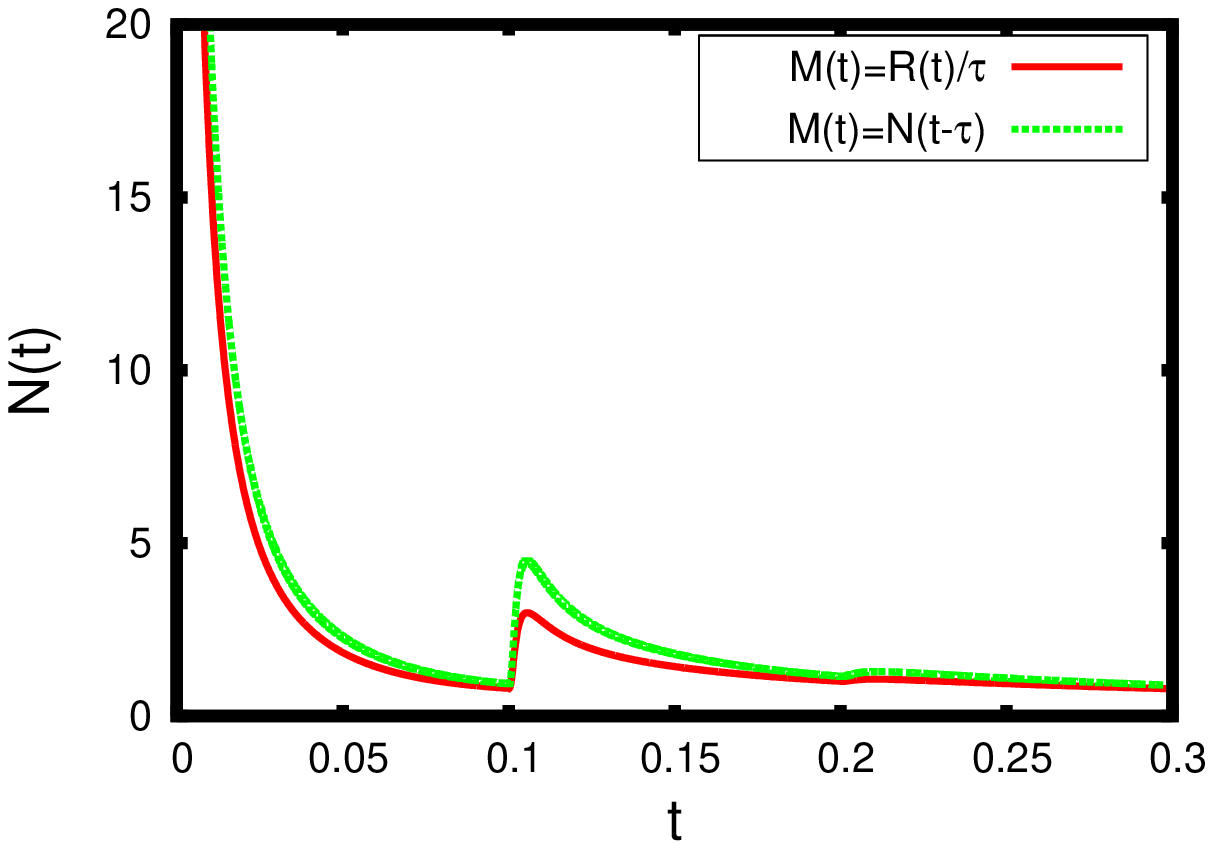}
\end{center}
\end{minipage}
\begin{minipage}[c]{0.33\linewidth}
\begin{center}
\includegraphics[width=\textwidth]{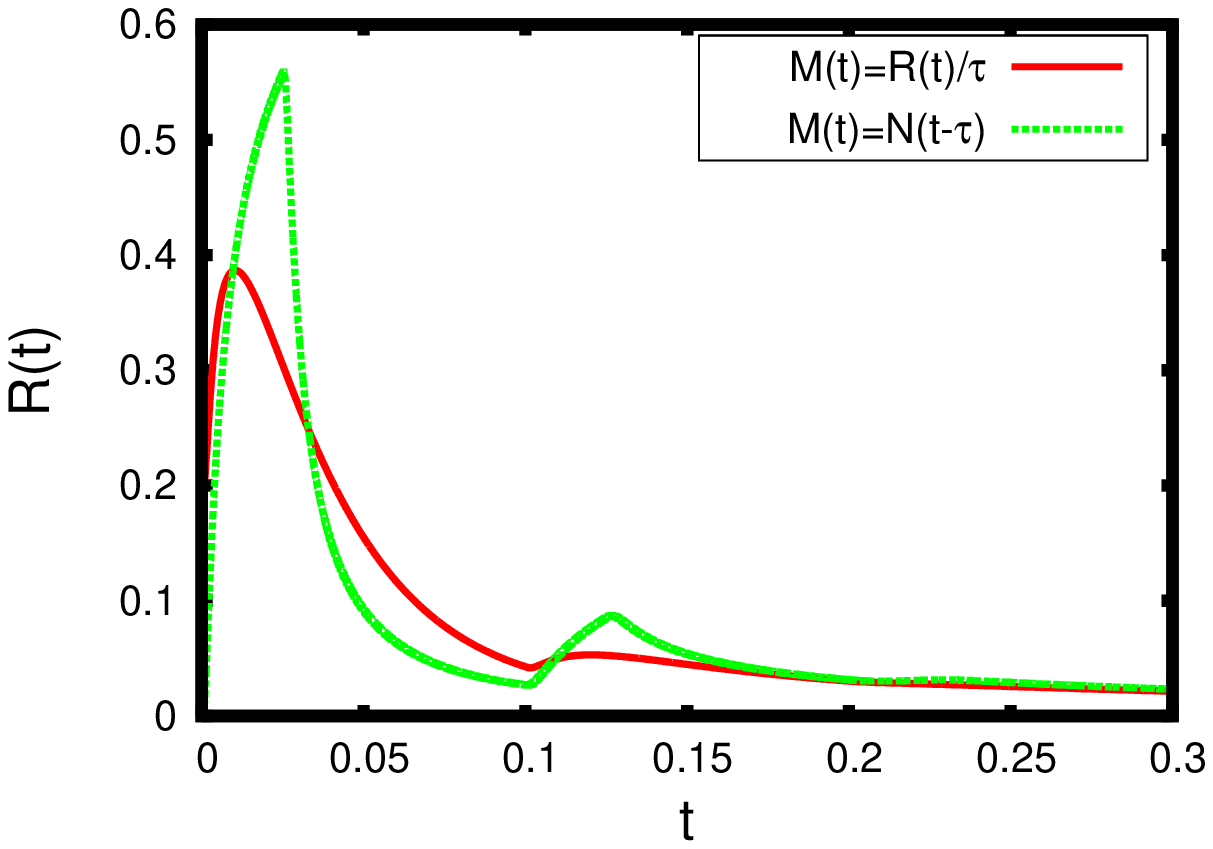}
\end{center}
\end{minipage}
\end{center}
\caption{
{\bf Comparison between $R(t)$ and $N(t)$ for $M(t)=\frac{R(t)}{\tau}$ and $M(t)=N(t-\tau)$.}
Top: initial data \eqref{ci_maxwel} with  $v_0=1.83$ and
 $\sigma_0=0.0003$,  the connectivity parameter  $b=-4$,
the transmission delay  $D=0.1$, $\tau=0.025$, $R(0)=0.2$ and $v_{ext}=20$.
Middle:  parameter space of Fig. \ref{blowup_1pob}, bottom. The qualitative behavior is the same for both models, even the
solutions seem to be hardly the same.
}
\label{blowup_ERD_ci_MJ}
\end{figure}
\begin{figure}[H]
\begin{center}
\begin{minipage}[c]{0.33\linewidth}
\begin{center}
\includegraphics[width=\textwidth]{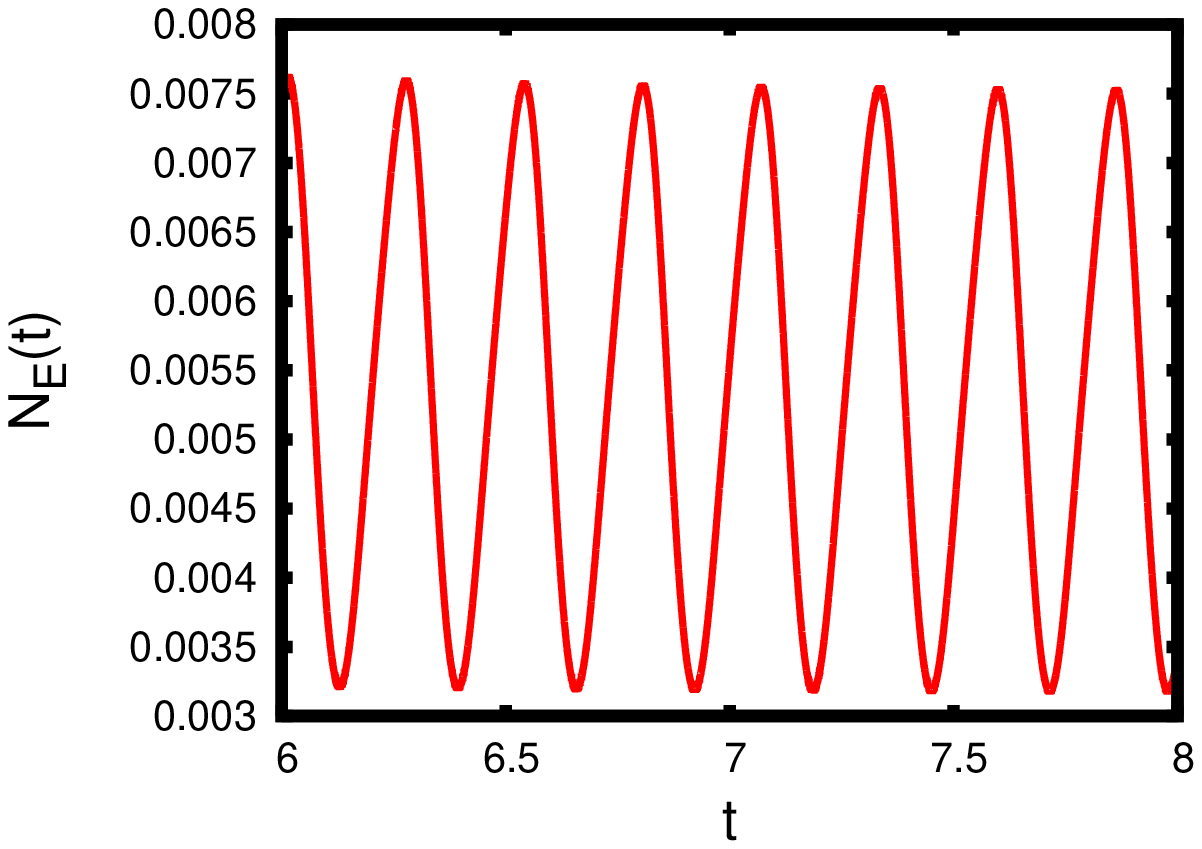}
\end{center}
\end{minipage}
\begin{minipage}[c]{0.33\linewidth}
\begin{center}
\includegraphics[width=\textwidth]{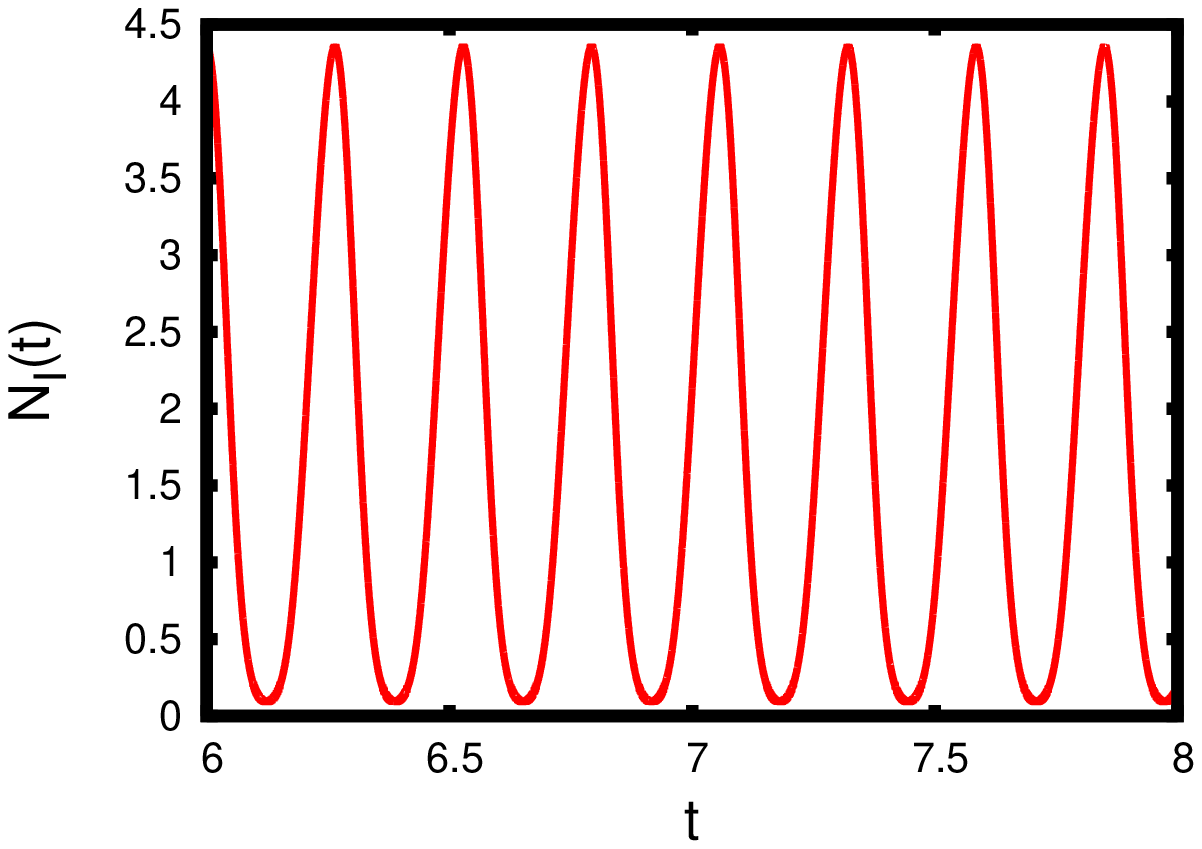}
\end{center}
\end{minipage}
\end{center}
\begin{center}
\begin{minipage}[c]{0.33\linewidth}
\begin{center}
\includegraphics[width=\textwidth]{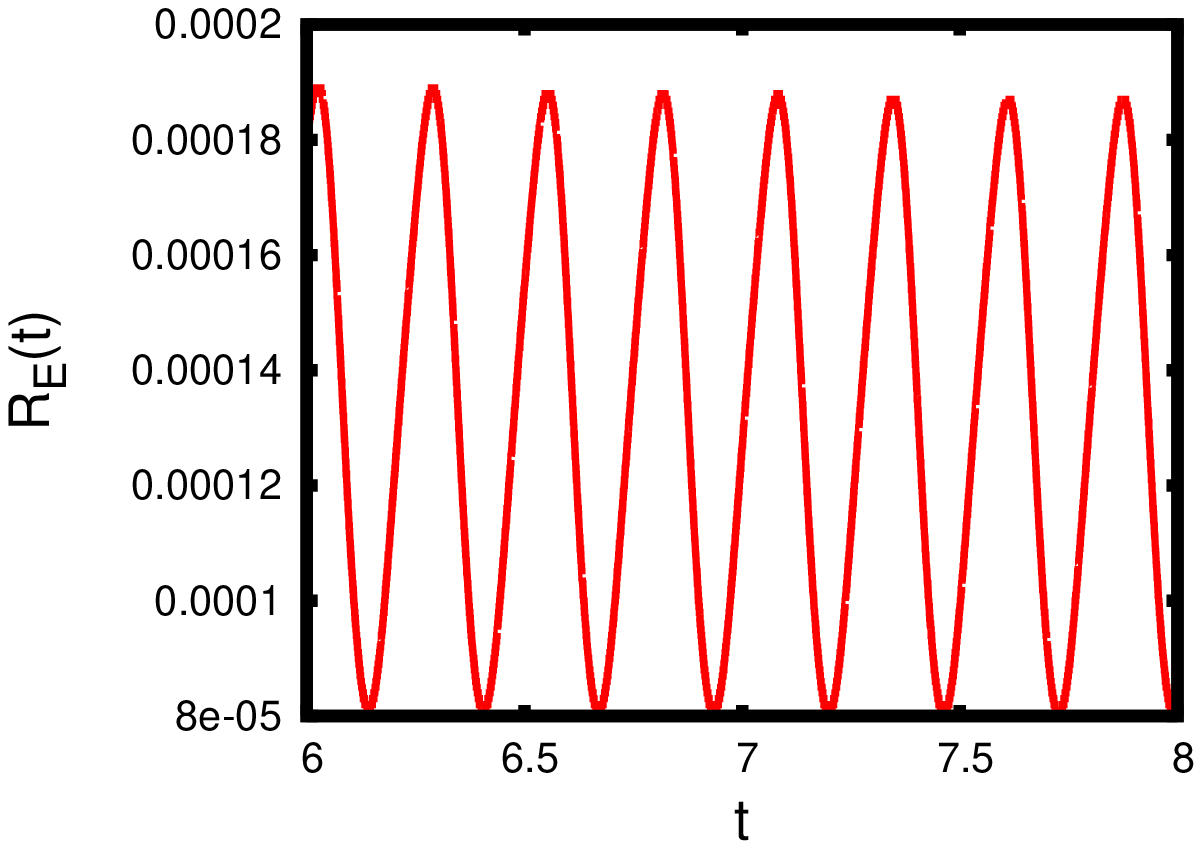}
\end{center}
\end{minipage}
\begin{minipage}[c]{0.33\linewidth}
\begin{center}
\includegraphics[width=\textwidth]{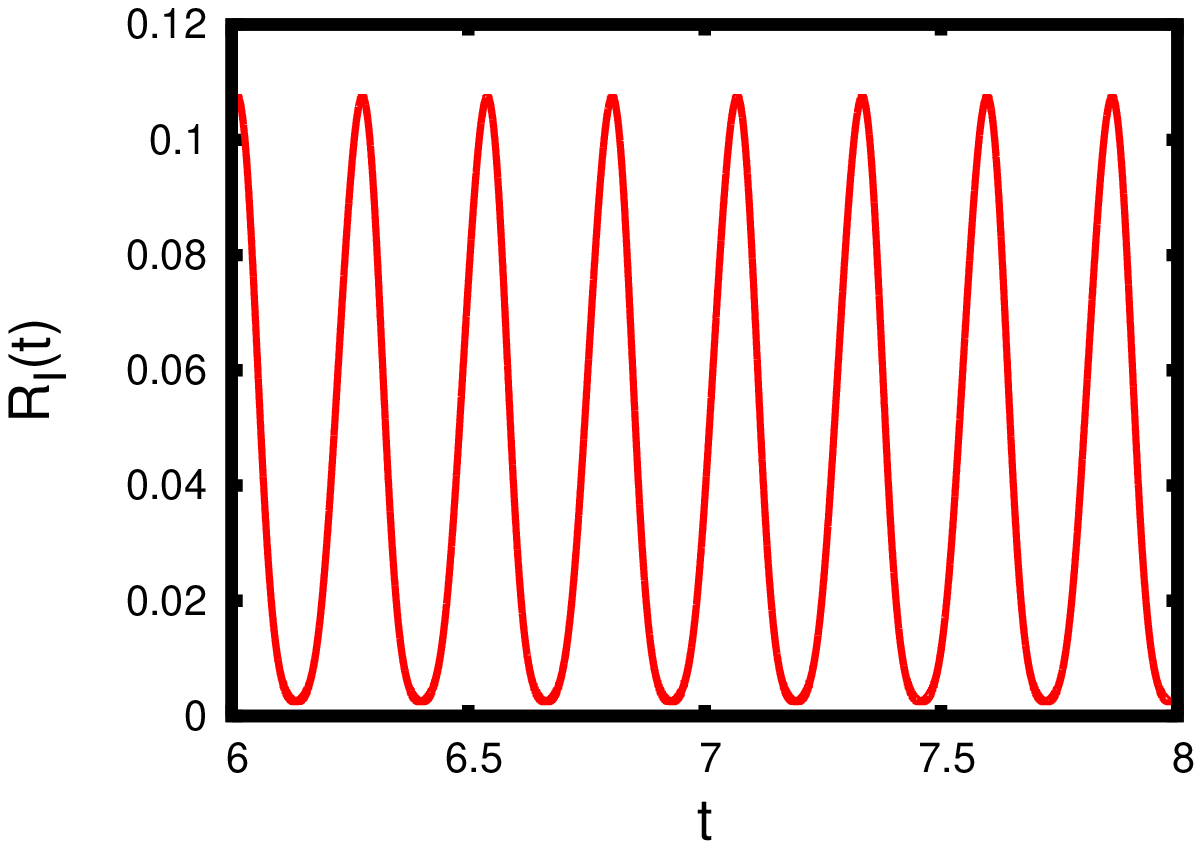}
\end{center}
\end{minipage}
\end{center}
\caption{{\bf System \eqref{modelo} (two populations: excitatory and
inhibitory) presents periodic solutions if there is a delay.-}
We consider initial data \eqref{ci_maxwel} with
$v_0^E=v_0^I=1.25$ and $\sigma_0^E=\sigma_0^I=0.0003$, $v_{ext}=20$
and the connectivity parameters
$b_E^E=0.5$, $b_I^E=0.75$, $b_I^I=4$, $b_E^I=1$
and with refractory states ($M_\alpha(t)=N_\alpha(t-\tau_\alpha)$)
where $\tau_\alpha=0.025$.
Top: Time evolution of the excitatory and inhibitory firing rates.
Bottom: Time evolution of the excitatory and inhibitory refractory states.
}
\label{oscilaciones_EI}
\end{figure}


\section{Conclusions and open problems}

In this work, we have extended the results presented in 
\cite{CCP, CP, CS17} to a  general network with two
populations (excitatory and inhibitory) with transmission delays between the
neurons, and where the neurons remain in a refractory state for a certain time.
From an analytical point of view we have explored the number
of steady states in terms of the model parameters
(Theorem \ref{th: steady states}), the long time behaviour
for small connectivity parameters (Theorem \ref{th: long1}),
 and  blow-up phenomena if there is not a transmission delay 
between excitatory neurons (Theorem \ref{th_blowup}).

Besides analytical results, we have presented a numerical resolutor
for this model \eqref{modelo}, based on high order flux-splitting WENO  schemes and
an explicit third order TVD Runge-Kutta method,
in order to describe the
wide range of phenomena displayed by the network: blow-up, asynchronous/synchronous 
solutions and instability/stability
of the steady states. The solver also allows to observe the time evolution of
not only the firing rates and refractory states, but also of 
the probability distributions of the excitatory and
inhibitory populations.

The resolutor was used to illustrate the  result of Theorem \ref{th_blowup}:
 as long as the transmission delay of the excitatory
to excitatory synapses is zero  ($D_E^E=0$),  blow-up phenomena 
 appear in the full NNLIF model, even
if there are nonzero transmission delays in the rest of the synapses. 

We  remark that the numerical results suggest  that
 blow-up phenomena disappear
when the excitatory to excitatory transmission delay  is nonzero,
and the solutions may
tend to a steady state or to a synchronous state. 
In the case of only one  average-inhibitory population
the behavior of the solutions after preventing a blow-up phenomenon
seems to depend on the strength of the external synapses  $v_{ext}$.
Furthermore, we  have also observed periodic solutions for small values of
the excitatory connectivity parameter combined with
 an initial data far from the threshold potential.
Thus,  synchronous  solutions are not a direct consequence of having 
avoided the blow-up phenomenon.

Our numerical study is completed with the stability analysis of  
 the steady states, when the network presents three of them.
In our simulations, we do not observe bistability phenomena
since the two upper stationary firing rates are unstable, while
the lowest one is stable.

Finally, to our knowledge, the numerical solver presented in this paper is the
  first deterministic solver to describe the behavior of the full
  NNLIF system involving all the characteristic phenomena of real
  networks. Including all relevant phenomena is essential to explore
  some open problems, as for instance, 
the analytical proof of the global existence
of solution when there is a nonzero excitatory to excitatory transmission delay,
the reasons why solutions sometimes tend to a steady state
and sometimes to a synchronous state, 
and 
an analytical study of the stability of the steady states 
when the connectivity parameters are not small.



\

\thanks{\em The authors acknowledge support from projects
MTM2011-27739-C04-02 and MTM2014-52056-P
of Spanish Ministerio de Econom\'\i a y Competitividad and
the European Regional Development Fund (ERDF/FEDER).
The second author was also sponsored by the grant BES-2012-057704.}

\bibliographystyle{siam}
\bibliography{Bib.bib}

\end{document}